\documentclass[10pt]{amsart}
\usepackage[utf8]{inputenc}
\usepackage[active]{srcltx}
\usepackage{amsmath,amssymb,amsthm,array,euscript, mathrsfs,version,geometry,textgreek,tabularx,
}

\usepackage[all]{xy}
\newif\ifPDF
\ifx\pdfoutput\undefined\PDFfalse
\else\ifnum\pdfoutput >0\relax\PDFtrue\else\PDFfalse\fi
\fi
\def\arxiv#1{\href{http://arxiv.org/abs/#1}{\tt arXiv:#1}} \let\arXiv\arxiv

\PDFtrue
\ifPDF
   \usepackage{aeguill}
   \usepackage[pdftex]{graphicx,color}
   \usepackage[pdftex,plainpages=false]{hyperref}
\else
   \usepackage[dvips]{hyperref}
\fi
\hypersetup{unicode=true}
\hypersetup{colorlinks=false}

\usepackage{mymacro} %11111
\rncmd\todo[1][]{}
\makeatletter

\newtheoremstyle{theorem*}% name
  {3pt}%      Space above, empty = `usual value'
  {3pt}%      Space below
  {\itshape}% Body font
  {}%         Indent amount (empty = no indent, \parindent = para indent)
  {\bfseries}% Thm head font
  {.}%        Punctuation after thm head
  {.5em}%     Space after thm head: " " = normal interword space;
  {\thmnote{#1\@ifnotempty{#3}{ }#3}}% Thm head spec

\theoremstyle{theorem*}
\newtheorem*{thm*}{Theorem}
\newtheorem*{prop*}{Proposition}

\rncmd\L\cL
\ncmd{\Dh}{\D_\hbar}
\ncmd{\ohf}{\omega^{1/2}}
\ncmd{\Lochf}{\Loc_{\ohf}}
\ncmd{\Ldet}{\L_{\on{det}}}
\ncmd{\tcurv}{\widetilde{\on{curv}}}
\ncmd\ty{\rT_\rY}
\ncmd\tx{\rT_\rX}
\ncmd\gy{\cG_\rY}
\ncmd\gx{\cG_\rX}
\ncmd\gtx{\cG_{\tx}}
\ncmd\gty{\cG_{\ty}}
\ncmd\frr{\fr[2]}
\ncmd\bo{\rB^\circ}
\ncmd\bof{(\bo)\fr}
\ncmd\xv{\rX\chk}
\ncmd\dd\gD
\ncmd\g\cG
\ncmd\loc{\Loc^\circ}
\ncmd\tth{\tilde\th}
\ncmd\yy\rY
\ncmd\om\omega
\ncmd\tc{\tilde\rC}
\ncmd\tco{\tc^\circ}
\ncmd\hig{\Higgs^\circ}
\ncmd\ddbuo{\dd_{\Bun}^\circ}
\ncmd\prg[1]{\pr_{\Gamma,#1}}
\ncmd\Tcf{T^*C\fr}
\ncmdd\thcf{\th_C\fr}
\ncmd\Smsch{\on{Sch}^{\on{sm}}_\k}
\ncmd\Gmna{\GG_m^\na}
\DMO\RcEnd{\mathit R{\cEnd}}
\DMO\RcHom{\mathit R{\cHom}}
\ncmd\mxg{\rM_{X,\g}}
\ncmd\Om\Omega
\ncmd\Oga{\Om_{\Gamma_A}}
\ncmd\iOga{\iOm_{\Gamma_A}}
\ncmd\vphi\varphi
\ncmd\rt\rT
\ncmd\iOm\varOmega
\ncmd\rmo{\rM^\circ}
\ncmd\tLoc{\wtld\Loc}
\ncmd\Locchism{\Loc^{\chi\text-\mathit{sm}}}
\ncmd{\Oloc}{\iOm_{\Loc}}
\ncmd\kkk\todo
\DefRm{\AJ\Gr}
\DefOps{\Frob}

\def\@DchL#1{\D_{#1}^{\L}}
\def\DchL@#1<#2>{\D_{#1}^{#2}}

\ncmd{\DchL}[1][\bc,\hbar]{\@ifnextchar<{\DchL@{#1}}{\@DchL{#1}}}

\def\nbhp-){\mbox{-})}
\newcommand\mynobreakpar{\par\nobreak\@afterheading}

\makeatother

\ncmd{\sm}{\SmSup{\mathit{sm}}}
\let\mathcal\CMcal
\ncmd{\Lna}{(\L,\na)}
\ncmd{\nlbk}{\nolinebreak}

\ncmd{\Connhf}{{\on{Conn}_{1/2}^{\it coh}}}
\ncmd{\Conn}{\on{Conn}^{\it coh}}
\ncmd{\Tblt}{\cT^\blt}
\ncmd{\tTblt}{\tilde\cT^\blt}

\ncmd{\Lc}{{\L^c}}
\ncmd{\Df}{\D_{\rY,\Lc}^\flat}
\ncmd{\Ds}{\D_{\rY,\Lc}^\sharp}
\ncmd{\tDf}{\tilde\D_{\rY,\Lc}^\flat}
\ncmd{\tDs}{\tilde\D_{\rY,\Lc}\SmSup\sharp}

\ncmd\bc{{\mathbf c}}
\ncmd\Dual\eD % Old notation - added for latexdiff
\DeclareRobustCommand{\SkipTocEntry}[5]{}

\numberwithin{equation}{section}
\numberwithin{thm}{section}
\numberwithin{defn}{section}
\numberwithin{rem}{section}

\DMO*{\bigboxtimes}{\vcenter{\hbox{\Huge$\boxtimes$}}}
\DeclareUnicodeCharacter{2207}{\textnabla}

\begin{document}

\title[Quantum geometric Langlands in positive characteristic]%
        {Quantum geometric Langlands correspondence in positive characteristic: the $\GL_N$ case}
\author{Roman Travkin}
%\dedicatory{draft}
%\address{%
%University of Chicago\\
%Department of Mathematics}
\email{travkin@alum.mit.edu}

\begin{abstract}
  We prove a version of quantum geometric Langlands conjecture in characteristic~$p$.  Namely,
  %Let $C$ be a smooth connected projective curve of genus $>1$ over an algebraically closed field~$\k$ of characteristic $p>0$. %, and $c\in\k\setm\Fp$.  Denote by $\Bun_N$ the stack of rank~$N$ vector bundles  on~$C$ and by~$\Ldet$ the line bundle on~$\Bun_N$ given by determinant of derived global sections.
  we construct an equivalence of certain localizations of derived categories of twisted crystalline $\D$-modules %$\D_{\Bun_N,\Ldet^c}$ and $\D\raisemath{.7ex}{_{\Bun_N,\Ldet^{-1/c}}}$.
  on the stack of rank~$N$ vector bundles on an algebraic curve $C$ in characteristic~$p$.  The twisting parameters are related in the way predicted by the conjecture, and are assumed to be irrational (\ie not in $\Fp$).
  We thus extend the results of~\cite{BB} concerning the similar problem for the usual (non-quantum) geometric Langlands.

%  Just as in the non-quantized case, the equivalence is constructed as a twisted version of Fourier--Mukai transform on the Hitchin fibration, using the Azumaya property of crystalline differential operators.  However, there are some new structures appearing in the quantum case.  In particular,
  In the course of the proof, we introduce a generalization of $p$-curvature for line bundles with non-flat connections, define quantum analogs of Hecke functors in characteristic~$p$ and construct a Liouville vector field on the space of de~Rham local systems on~$C$.
\end{abstract}

\maketitle

\setcounter{tocdepth}{3}
\tableofcontents

\section{Introduction}

Fix an algebraic curve~$C$ and a reductive group~$G$. The geometric Langlands correspondence (GLC) is a conjectural equivalence of derived categories between $\D$-modules on the moduli space $\Bun_G$ of $G$-bundles on~$C$ and quasi-coherent sheaves on the moduli space~$\Loc$ of local systems for the Langlands dual group~$^LG$.  It has a classical (commutative) limit which amounts to the derived equivalence of Fourier--Mukai type between Hitchin fibrations for $G$ and~$^LG$.  The latter is a fibration $T^*\Bun_G \to \rB$ over an affine space with generic fibers being abelian varieties (or a little more general commutative group stacks).

In~\cite{BB}, a characteristic~$p$ version of GLC is established.  Namely, the setup of crystalline (\ie without divided powers) $\D$-modules in positive characteristic is considered.  In this setting, the category of $\D$-modules does not get far from its classical limit: it is described by a $\Gm$-gerbe on the Frobenius twist of the cotangent bundle.  So the GLC becomes a twisted version of its classical limit.  Based on this reasoning, the GLC is constructed ``generically'' for the case of general linear group $G = \GL_N$.

In this paper, we apply the same technique to the \emph{quantum} version of GLC introduced by B.~Feigin and A.~Stoyanovsky in \cite{S1,S2}.  This deformation of GLC has been studied in~\cite{KW,K} in relation to quantum field theory (see also~\cite{F}), and in \cite{Ga-Whit,Ga-notes} in the new approach geometric Langlands.  It has the same classical limit, but now both sides are ``quantized.''  We note that in the abelian case when $G=\Gm$ (\ie $N=1$) the quantum GLC is realized as a Fourier--Mukai transform on the Jacobian (or the Picard stack) of~$C$, and falls in the framework of~\cite{PR}.

Now, considering quantum geometric Langlands in characteristic~$p$, we get (generically) a twisted version of the same Fourier--Mukai transform as for the classical case.  However, the proof that the twistings on both sides are interchanged by the Fourier--Mukai transform is more complicated than in the case of usual GLC, and contains several new ingredients.  Also, we restrict to the case of irrational parameter~$c$ because there is a certain degeneration happening at rational~$c$.

First, we need a description of the category of modules for a twisted differential operator (TDO) algebra.  The center of a TDO was already described in~\cite{BMR}, but the description of the corresponding gerbe presented here seems to be new.  A convenient language for this turns out to be that of {\em extended curvature}---an invariant of a line bundle with a ({\em not necessarily flat}) connection taking values in a canonical coherent sheaf~$\eF$ on the Frobenius twist of the variety.  This is a generalization of the $p$-curvature of flat connections.  Just as with the usual $p$-curvature, every section of~$\eF$ defines a gerbe whose splittings correspond to connections with that extended curvature.  Now, if $\L$ is a line bundle and $c \in \k \setm \Fp$ then the gerbe describing $\D_\Lc$ corresponds to $c \cdot \a_\L$ where $\a_\L$ is the extended curvature of the pullback of~$\L$ to the associated twisted cotangent, equipped with the canonical (``universal'') connection.

We then apply this to the determinant bundle on~$\Bun$ whose corresponding twisted cotangent is identified with~$\Loc$.  So, to construct the desired equivalence, we have to split a gerbe on the fiber product of dual $p$-Hitchin fibrations $\Loc \to \rB\fr$.  (Although for~$\GL_N$ the $p$-Hitchin fibrations are the same, we use differently scaled projections to the Hitchin base.)  This is done by constructing an explicit line bundle with connection on this fiber product.  The problem then reduces to proving certain equality involving $\tilde\th = \a_\L$ for $\L$ being the determinant bundle on~$\Bun$, and the torsor structure on the $p$-Hitchin fibration.

We prove this property for another section~$\tilde\th_0$ of~$\eF_{\Loc}$ and then show that $\tilde\th = \tilde\th_0$.
The section~$\tilde\th_0$ is constructed from a vector field~$\xi_0$ on~$\Loc$ lifting the differential of the standard $\Gm$-action on the Hitchin base.  This vector field comes from an action of fiberwise dilations of~$T^*C\fr$ on the gerbe describing $\D$-modules.  Such structure is not unique (it depends on the choice of lifting of~$C$ to the 2nd Witt vectors of~$\k$), however the corresponding vector fields~$\xi_0$ all differ by Hamiltonian vector fields and give rise to the same~$\tilde\th_0$.

% This version of the text contains only a sketch of some of the arguments (and some proofs are missing).  A more detailed exposition will be presented in the next versions of the paper.

\subsection{Quantum geometric Langlands conjecture}
Let $C$ be a smooth irreducible projective curve of genus $g > 1$ over an algebraically closed field $\k$ of characteristic 0 and $G$ a reductive algebraic group.  We denote by $\Bun_G = \Bun_G (C)$ the moduli stack of $G$-bundles on~$C$.  The quantum geometric Langlands correspondence is a conjectural equivalence between certain derived categories of twisted $\D$-modules on $\Bun_G$ and $\Bun_{^LG}$ where $^LG$ denotes the Langlands dual group.  The twistings should correspond to invariant bilinear forms on the Lie algebras of $G$ and~$^LG$ that induce dual forms on the Cartan subalgebras (up to the shift by the critical level).  When one of the forms tends to~0 the other tends to infinity, which corresponds to degeneration of the TDO algebra into a commutative algebra of functions on a twisted cotangent bundle to $\Bun_G$.  This shows that the quantum geometric Langlands is a deformation of the classical geometric Langlands, which is an equivalence between the category of (certain) $\D$-modules on $\Bun_G$ and the category of (certain) quasi-coherent sheaves on the stack $\Loc_{^LG}$ of $^LG$--local systems on~$C$.

We will be interested in the case $G = \GL_N$---the general linear group.  In this case, we think of the quantum Langlands correspondence as an equivalence
\[
    \D_{\Bun_N,\Ldet^c\tsr\ohf_{\Bun}}\hmod \isoto \D_{\Bun_N,\Ldet^{-1/c}\tsr\ohf_{\Bun}}\hmod
\]
where $\Ldet$ is the determinant line bundle on~$\Bun_N = \Bun_G = \Bun_{^LG}$ given by $(\Ldet)_b = \det \RGam(C, \cE_b)$ for any $b \in \Bun_N$ where $\cE_b$ denotes the rank~$N$ vector bundle corresponding to~$b$, and $\omega_{\Bun}$ is the line bundle of top forms (determinant of cotangent complex) on~$\Bun$.  (There is a subtle question of what kind of $\D$-modules one should consider, but we'll ignore it for now.)

\subsection{Geometric Langlands for  $\D$-modules in characteristic~$p$}
In~\cite{BB}, R.~Bezrukavnikov and A.~Braverman established a version of the classical geometric Langlands correspondence for  ``crystalline'' $\D$-modules over a field~$\k$ of characteristic $p > 0$.  Recall that, for a smooth scheme~$X$ over~$\k$, the sheaf~$\D_X$ of {\em crystalline differential operators} is defined as the universal enveloping algebra of the Lie algebroid~$\cT_X$ of vector fields on~$X$.   The main tool for studying modules over such algebras is their Azumaya property (see~\cite{BMR}).  Namely, $\D_X$ turns out to be isomorphic to (the pushforward to~$X$ of) an Azumaya algebra~$\tilde\D_X$ on $T^*X\fr$---the cotangent bundle to the Frobenius twist of~$X$.  This allows one to identify the category of $\D$-modules on~$X$ with the category of coherent sheaves on a $\Gm$-gerbe on $T^*X\fr$.

This observation is generalized in~\cite{BB} to the case of (a certain class of) algebraic stacks.  Namely, for an irreducible smooth Artin stack~$\rY$ over~$\k$ with $\dim T^*\rY = 2 \dim \rY$ (\ie $\rY$ is good in the sense of~\cite{BD}), they construct a sheaf~$\tilde\D_\rY$ of algebras on~$T^*\rY\fr$ with properties similar to the Azumaya algebra~$\tilde\D_X$ described above.  The pushforward of~$\tilde\D_\rY$ to~$\rY\fr$ is isomorphic to $\Fr_{\rY*} \D_\rY$ where $\D_\rY$ is the sheaf of differential operators as defined in~\cite{BD}. Moreover, the restriction of~$\tilde\D_\rY$ to the maximal open smooth Deligne--Mumford substack of~$T^*\rY\fr$ is an Azumaya algebra.\footnote{In  fact, this construction can be strengthened, namely, one can define a $\Gm$-gerbe on {\em all of $T^*\rY\fr$}, not just its smooth part, as we show in~\ref{subs:TDO stk}.  This gerbe classifies $\D$-modules on~$\rY$, defined in a way similar to \cite[Sect.~1.1]{BD}.  The ``regular'' (as in ``free rank~1'') $\D$-module, however, corresponds to a coherent sheaf on this gerbe which is locally free only on the smooth part---that's why $\tilde\D_\rY$, which is (opposite of) the endomorphism algebra of this coherent sheaf, is an Azumaya algebra only on that smooth locus.}

The stack $\Bun_N$ is almost ``good,'' namely, it locally looks like product of a good stack and $B\Gm$.  So one can apply the above construction to $\rY = \Bun_N$ to get an Azumaya algebra on $T^*\Bun_N\fr = \Higgs\fr$.  The latter stack is the total space of the Hitchin fibration $h\fr \colon \Higgs\fr \to \rB\fr$ whose generic fibers are Picard stacks of (spectral) curves.  On the dual side, one has the ``$p$-Hitchin'' map $\Loc \to \rB\fr$ given by $p$-curvature.  Generic fibers of this map are torsors over the same Picard stacks, and each point of such a torsor (which corresponds to $G$--local system on~$C$ with given spectral curve) gives a splitting of~$\tilde\D_{\Bun}$ on the corresponding fiber of~$H\fr$.  This splitting defines a Hecke eigensheaf corresponding to the local system.  The geometric Langlands is thus realized as a twisted version of Fourier--Mukai transform.

\subsection{Main result}
In this paper, the same ideas are applied to quantum geometric Langlands correspondence.  To that end, we generalize the above Azumaya algebra construction to the case of twisted differential operators.  The only TDO algebras we will encounter are of the form $\D_\Lc$ where $\L$ is a line bundle and $c \in \k$ (and external tensor products of such).  In this case, the situation is essentially analogous to the non-twisted case, except that the Azumaya algebra will now live on the twisted cotangent bundle, where the twisting is given by $(c^p-c)$ times the Chern class of~$\L\fr$ (cf.~\cite{BMR}).  We will only consider the case of irrational $c$ (\ie $c \not\in \Fp$): in this case one can identify this twisted cotangent bundle with the Frobenius twist of the space~$\tilde T^*_\L X$ of 0-jets of connections on~$\L$.  This is discussed in~\ref{TDO}.

It is not hard to extend it to the stack case using the above-mentioned results from~\cite{BB} for usual $\D$-modules on stacks.  Thus, for a line bundle~$\L$ on a good stack~$\rY$, one gets an Azumaya algebra~$\tilde\D_{\rY,\Lc}$ on the smooth part of~$(\tilde T^*_\L \rY)\fr$.  (For a discussion of twisted cotangent bundles to stacks, see~\ref{appx|subs:general}.)

\medskip
We apply this to the determinant bundle $\Ldet$ on $\Bun$.  One can check (see Appendix~\ref{appx}) that the corresponding twisted cotangent is identified with the moduli space~$\Lochf$ of rank~$N$ bundles on~$C$ endowed with an action of the TDO algebra~$\D_{\ohf}$.  In fact, we can identify $\Lochf$ with $\Loc$ by tensoring bundles with $\omega^{\tsr(p-1)/2}$.  Thus, both sides of the quantum Langlands are described (again, generically over the Hitchin base) by certain gerbes on~$(\Loc^\circ)\fr$.  Here $\Loc^\circ = \Loc \xx_{\rB\fr} (\bo)\fr$ and $\bo \ctd \rB$ is the open part parametrizing smooth spectral curves.  Using the $p$-Hitchin map as above (this time to $\rB\frr$), we see that these gerbes live on two torsors over the relative Picard stack mentioned above.  So we get again two ``twisted versions'' of the derived category of coherent sheaves on this Picard stack.  In contrast to the classical (non-quantum) case, however, we have both ``torsor'' and ``gerby'' twists on each side.  These two kinds of twists are interchanged by Fourier--Mukai duality.

In other words, we prove the following:
\begin{thm}\label{main_thm}
  Suppose $p$ is large enough relative to $N$ and the genus of~$C$.\footnote{This assumption will only be used in the proof of Lemma~\ref{lem transv} on generic transversality of fibers of two projections.  For the rest of the argument, $p>2$ is sufficient.}
  Then there is an equivalence between bounded derived categories of modules for $\D_c = \tilde\D_{\Bun,\Ldet^c} |_{(\Loc^\circ)\fr}$ and $\D_{-1/c} = \tilde\D_{\Bun,\Ldet^{-1/c}} |_{(\Loc^\circ)\fr}$.  The corresponding kernel is a splitting of $\D_c \bx \D_{-1/c}^{\rm{op}} \sim \D_c \bx \D_{1/c}$ on the fiber product $(\Loc^\circ)\fr \xx_{\rB\frr,c^p} (\Loc^\circ)\fr$ where the projection from the second factor to~$\rB\frr$ is modified by the action of $c^p \in\Gm$.  If we choose, locally on~$\rB\frr$, a trivialization of the torsor $\Loc\fr\to \rB\frr$, then there are splittings of $\D_c$~and~$\D_{-1/c}$ such that the equivalence is identified with the Fourier--Mukai transform on the Picard stack $\Pic((\tilde\rC^\circ)\frr/\rB\frr)$.  (Here $\tilde\rC^\circ \ctd T^*C \x \bo$ is the universal spectral curve.)
\end{thm}

Note that, although the underlying spaces of the torsors are the same on both sides (namely $(\Loc^\circ)\fr$), in order to make the duality work, one has to normalize the projection to~$\rB\frr$ differently.  This can also be guessed by considering what happens at rational~$c$ (including $c = 0,\infty$).

We also remark that, although there are no critical twists mentioned in the statement of the theorem, for a smooth variety~$X$ in characteristic $p>2$ there is always an equivalence $\D_X\hmod\isoto\D_{X,\ohf_X}\hmod$ (and similarly for $\Lc$-twisted versions) given by tensoring by $\omega_X^{\tsr(1-p)/2}$.

\subsection{Extended curvature}
So, all we need to check is that the torsors with gerbes corresponding to $\D_c$ and $\D_{-1/c}$ are interchanged by Fourier--Mukai duality.  For that purpose we need a description of gerbes attached to TDO algebras.  Recall that in the non-twisted case, the splittings of~$\tilde\D_X$ on an open subset $U\fr \ctd T^*X\fr$ correspond to line bundles on~$U$ with flat connection of $p$-curvature equal to the canonical 1-form on~$T^*X\fr$.

To extend this description to the TDO case, we introduce a generalization of the notion of $p$-curvature to non-flat connections.  For a line bundle~$\L$ with connection~$\na$ on a smooth variety~$X$, we define in~\ref{extcrv} a section $\tcurv\Lna$ (called the {\em extended curvature}) of the quotient sheaf~$\eF_X$ of $\Om^1_X$ by locally exact forms.  This sheaf maps to~$\Om^2_X$ via de~Rham differential; this map carries $\tcurv\Lna$ to the usual curvature.  On the other hand, for flat connections, $\tcurv\Lna$ is a section of closed modulo exact forms, which corresponds to the $p$-curvature of~$\na$ under Cartier isomorphism. This construction also allows, starting from a section $\a\in \eF_X$ (such a section will sometimes be referred to as a {\em generalized one-form}), to define a $\Gm$-gerbe on~$X\fr$: its splittings correspond to line bundles with connection whose extended curvature is equal to~$\a$.

Now, the pullback of any line bundle~$\L$ to its associated twisted cotangent~$\tilde T^*_\L X$ acquires a canonical connection.  If $\a_\L$ denotes the extended curvature of this connection, the gerbe on~$(\tilde T^*_\L X)\fr$ corresponding to the Azumaya algebra~$\tilde\D_\Lc$ for $c \in \k \setm \Fp$ is obtained from the above construction applied to $c\a_\L$.

\subsection{The Poincar\'{e} bundle}
Then we construct an explicit kernel of the equivalence (an analogue of the Poincar{\'e} bundle).  This is a line bundle with connection on the fiber product of two copies of $\Loc^\circ$ over the Hitchin base (see formula~\eqref{sfb}). The construction is similar to that of Poincar{\'e} bundle on the square of the Picard stack of a curve:
\[
    \text{Poincar\'e}(\L,\L') = \det \RGam(\L\tsr\L') \tsr (\det \RGam(\L) \tsr \det \RGam(\L'))^{\tsr-1} .
\]

Namely, the determinant bundle on the Picard stack gets replaced by the determinant bundle on $\Loc^\circ$ with ``tautological'' connection (the same one that is used to describe the gerbe on~$(\Loc^\circ)\fr$), while the role of tensor product of line bundles is played by the addition map on the fibers of the $p$-Hitchin map:
\[
    \Loc^\circ_1 \xx_{\rB\fr} \Loc^\circ_c  \longto \Loc^\circ_{1+c}.
\]
%%
%\pagebreak
Here subscripts indicate scaling of the projection to the Hitchin base.  The fiber of~$\Loc^\circ_c$ classifies splittings on the spectral curve of the gerbe corresponding to the canonical 1-form on $T^*C\fr$ multiplied by~$c$.  This map can then be thought of as ``tensoring over the spectral curve.''

The main difficulty is then to check that this bundle with connection has the correct $p$-curvature.
This reduces to a certain linear equality on the extended curvatures (formula~\eqref{eq:hard tth id}).  This formula can be interpreted as a kind of additivity of the generalized one-form $c^{-1}\tilde\th$ on~$\Loc^\circ_c$ with respect to the addition maps above, where~$\tilde\th$ denotes the extended curvature of the tautological connection on the determinant bundle.

\subsection{The final step}
In~\ref{subs:alt constr}, we construct another generalized one-form $\tilde\th_0$ on $\Loc^\circ$ (actually on the maximal smooth part of~$\Loc$) whose image in~$\Om^2$ coincides with that of~$\tilde\th$ (both are equal to the symplectic form on~$\Loc^\circ$) but whose behavior with respect to the $p$-Hitchin map is more controllable.  We prove the additivity property for it, and then show that $\tilde\th = \tilde\th_0$.  In fact, $\tilde\th_0$ lifts to an actual antiderivative~$\th_0$ of the symplectic form.  Such antiderivatives correspond bijectively to Liouville vector fields.  We construct such a vector field using an equivariant structure of the gerbe on $T^*C\fr$ under the Euler vector field.  Such structures correspond to liftings of~$C$ to the 2nd Witt vectors of~$\k$.  Since $\tilde\th - \tilde\th_0$ is closed, it corresponds by Cartier to a 1-form~$\b_0$ on~$(\Loc^\circ)\fr$ and we have to prove that it is~0.

The definition of~$\Loc_c$ above makes sense for all $c \in \k$; in particular, for $c=0$ it gives $\Loc_0 = \Higgs\fr$.  On $\Higgs\fr$ we have the canonical 1-form $\th_{\Higgs}\fr$ (as on a cotangent bundle).  We prove that both $\tilde\th$ and $\th_0$ are compatible with~$\th_{\Higgs}\fr$ with respect to the action map
\[
    \Loc^\circ_0 \xx_{\rB\fr} \Loc^\circ  \longto  \Loc^\circ .
\]
In the beginning of Section~\ref{sec:proof} we prove this for $\tilde\th$.  It is enough to prove it on the image of the Abel--Jacobi map in~$\Higgs$, which in turn reduces to studying how the determinant bundle (with connection) on~$\Loc^\circ$ changes when we twist the local system by a point of its $p$-spectral curve.  In fact, after passing from generalized one-forms to gerbes, the compatibility of~$\tilde\th$ has an interpretation in terms of the existence of what might be called ``quantum ($p$\nbhp-)Hecke functors'' (see Remark~\ref{rem:q-Hecke}), which behave like usual Hecke functors of classical geometric Langlands, and can be compared to the the category studied in~\cite{FL} which plays the same role for quantum geometric Langlands at rational~$c$ in characteristic~$0$.   The compatibility of~$\th_0$ with~$\th_{\Higgs}\fr$ is proved as part of the additivity for~$\th_0$.  (In fact, the additive family of 1-forms on~$\Loc^\circ_c$ constructed in~\ref{subs:alt constr} specializes to~$\th_0$ for $c=1$ and to~$\th_{\Higgs}\fr$ for $c=0$.)

From this we conclude that $\b_0$ descends to the Hitchin base.  On the other hand, in~\ref{subs:proof Jer} we study the behavior of~$\b_0$ with respect to the projection $\Loc \to \Bun$.  First, by a degree estimate we show that the restriction to the fibers of this projection have constant coefficients.  Then, a global argument shows that in fact this restriction must be~0.  The fibers of the two projections $\Loc \to \rB\fr$ and $\Loc \to \Bun$ are generically transversal (at least, we know how to prove this for one of the components of~$\Loc$ and generic~$C$), which gives the desired equality $\b_0=0$.

\subsection{Further and related research}
Here we briefly indicate some of the directions of future research:
\begin{itemize}
  \item In~\cite{Ar}, D.~Arinkin constructed a version of Fourier--Mukai transform for derived categories of compactified Jacobians of curves with plane singularities in characteristic~$0$.  In~\cite{Gr}, M.~Gr\"ochenig extended this result to characteristic~$p$ (with some assumptions on~$p$) and used this to extend the equivalence of~\cite{BB} to the locus in~$\rB$ parametrizing reduced, {\em but possibly singular}, spectral curves.  It should be possible to do the same for quantum geometric Langlands.  One can also ask whether the restrictions on~$p$ can be relaxed.
  \item It is important in our construction that the ``quantization'' parameter~$c$ is irrational ($c \not\in \Fp$). It is therefore natural to ask whether one can extend this to rational~$c$, including $c=0$ corresponding to the  case of classical geometric Langlands.  The degeneration of the categories involved is similar (or at least related) to the case of degeneration of geometric Langlands to its quasi-classical limit.  This should also reduce the potential ambiguity in the equivalence in Theorem~\ref{main_thm}.
  \item Tsao-Hsien Chen and Xinwen Zhu \cite{CZ-GL,CZ-NA Hodge} generalized the results of~\cite{BB} from $G=\GL_N$ to the case of arbitrary reductive~$G$.  One can try to combine their techniques with those of the present paper to establish quantum geometric Langlands for general~$G$.
  \item Also worth mentioning here is the paper~\cite{N} by T.~Nevins which establishes the ``mirabolic'' version of Bezrukavnikov--Braverman results.
\end{itemize}

\subsection{Acknowledgments}
 The original idea of applying the techniques of~\cite{BB} to the quantum case belongs to my adviser,
Roman Bezrukavnikov.
 I am grateful to him for various ideas, as well as support and patience in the course of working on the problem.
 I would also like to thank
Dima Arinkin,
Alexander Beilinson,
Alexander Braverman,
Dustin Clausen,
Dennis Gaitsgory
 for fruitful discussions, and
Tsao-Hsien Chen,
Michael Finkelberg and
Ivan Mirkovi\'c
 for the interest in our work.

I am indebted to Michael Finkelberg for his courage and support over the years, as well as for finding some typos in this paper.
Thanks to Tsao-Hsien Chen and Michael Gr\"{o}chenig for sharing their papers  \cite{CZ-GL} and \cite{Gr} prior to publication.            I thank Dennis Gaitsgory and the referees for suggestions on improving
the text.

Different parts of this paper were written while I was in Massachusetts Institute of Technology, University of Chicago, and Harvard University as well  as Hebrew University of Jerusalem.
I thank these institutions for  hospitality and support.

Finally, I appreciate the effort of many people who contributed to my successful study, work and life in all of  these  places.

This research was partially conducted during the period the
author was employed by the Clay Mathematics Institute as
a Research Scholar.

% \section*{Table of notation}
%
% {\nopagebreak\smallskip
% \begin{center}
% %\parbox{\flushleft
% \begin{tabularx}{.9\textwidth}{>{$}r<{$}@{\enspace--\enspace}
%     >{\raggedright\arraybackslash}X%>{\raggedright}p{.5\textwidth}}
%     }
%   \hline \rule{0pt}{2.6ex}
%   \k & an algebraically closed field of characteristic $p>0$ \\
%   G & the general linear group $\GL(N)$ \\
%   C & a complete smooth algebraic curve over $\k$ \\
%   \Bun & the moduli stack of $G$-bundles on $C$ \\
%   \rB & the Hitchin base {\large(}the affine space $\bigoplus_{i=1}^N  H^0(C,\omega_C^{\tsr i}) ${\large)} \\
%   \bo \ctd \rB & the open part classifying smooth spectral curves. \\
%   \Higgs & the total space of the Hitchin fibration, the moduli stack of Higgs bundles, $\Higgs=T^*\Bun$ \\
%   \tilde\rC & ``universal spectral curve,'' $\tilde\rC \subset T^*C\x \rB$ \\
%   \hline
% \end{tabularx}%\par}
% \end{center}
% }
\todo[may add a more complete table of notation here at some point]

\section{Differential operators in positive characteristic}\label{sec:diff op charp}

\subsection{Frobenius morphisms and twists}
For any scheme $S$ of characteristic $p$ (\ie such that $p\O_S=0$) the absolute Frobenius $\Frob_{S}\colon S\to S$ is defined as $\id_S$ on the topological space and $\Frob_{S}^\#(f)=f^p$ on functions. For any $S$-scheme $X\xra{\pi} S$ one constructs a commutative diagram
\[
\xymatrix{X\ar[r]_{\Fr_{X/S}} \ar[dr]_{\pi} \ar@/^1pc/[rr]^{\Frob_{X}}
            & X\fr[S]\ar[d]|{\pi\fr[S]} \ar[r]_{W_{X/S}} & X\ar[d]^{\pi}\\
            & S\ar[r]^{\Frob_S} & S }
\]
where the square is Cartesian.  We call the $S$-scheme $X\fr[S] \xra{\pi\fr[S]} S$ the {\em relative Frobenius twist} of~$X$ over~$S$, and call $\Fr_{X/S}$ the {\em relative Frobenius morphism}. We will denote by $\bullet\fr[S]$ the pullback by $\Frob_S$ or $W_{X/S}$ (\ie the relative Frobenius twist over~$S$).  E.g., if $\cE$ is a sheaf on~$X$ then $\cE\fr[S]:= W_{X/S}^*\cE$, similarly for a relative differential form~$\omega$ on $X$ over $S$, $\omega\fr[S]:=W_{X/S}^*\omega$.  In the case $S=\Spec\k$ we will drop ``$S$'' and write $\Fr_X$ and $X\fr$ instead of $\Fr_{X/S}$ and $X\fr[S]$.  The $k$'th iterate of $\bullet\fr$ will be denoted $\bullet\fr[k]$.

Note that  $\Fr_{X/S}$ and $W_{X/S}$ induce mutually inverse isomorphisms  of Zariski sites, which we use to identify Zariski sheaves on $X$~and~$X\fr[S]$.  Note that if $S$ is spectrum of a field then $W_{X/S}$ is actually an isomorphism of abstract schemes (but not of schemes over~$S$).

\subsection{\texorpdfstring{$\lambda$}{\textlambda}-connections} \label{lconn}
Recall that a \emph{$\la$-connection} on a vector bundle $E$ on a smooth variety $X$ is a $\k$-linear morphism of sheaves $\tilde\na \colon \cE \to \Om^1 \tsr_\O \cE $ such that
\[
   \forall f\in\O \quad \forall s\in\cE \qquad \tilde\na(fs) = f\cdot\tilde\na s + \la\cdot df\tsr s
\]
where $\cE$ is the sheaf of sections of $E$.

For $\lambda\ne0$ if $\tilde\na$ is a $\la$-connection then $\na = \la^{-1} \tilde\na$ is a connection, and vice versa. In this case, the curvature of a $\la$-connection can be expressed in terms of the ordinary curvature: $F_{\tilde\na} = \la^2 F_\na$. The case $\la=0$ can be thought of as a limit when $\la\to0$. In particular if $\cE$ is trivialized, and $\tilde\na=\la d+\th$ then $F_{\tilde\na} = \la d\th + \th\wedge\th$

For a line bundle $\L$ on $X$ and $\lambda\in\k$, define a torsor $\tilde T^*_{\L^\lambda} X$ over $T^*X$ whose sections are $\lambda$-connections on $\L$.

\begin{rem}
As a variety, $\tilde T^*_{\L^\lambda} X$ is isomorphic to $\tilde T^*_\L X := \tilde T^*_{\L^1} X$ for $\lambda\ne0$ and to $T^*X$ for $\lambda=0$.
\end{rem}

%=\subsubsection{Flat $\la$-connections}

\medskip
We now discuss the notion of curvature and flatness of a $\la$-connection.

Define the \emph{curvature} of a $\la$-connection $\tilde\na$ to be the section $F_{\tilde\na}$ of $\Om^2 \tsr \cEnd\cE$ corresponding to the $\O$-linear map $\tilde\na^2 \colon  \cE \to  \Om^2 \tsr\cE$ where $\tilde\na$ is extended to $\Om^\bullet \tsr\cE$ by the following ``Leibnitz rule'':
\[
   \tilde\na(\omega \tsr s) = (-1)^{\deg\omega}\omega \wedge \tilde\na s + \la \cdot d\omega \tsr s.
\]
An alternative definition of $F_{\tilde\na}$ is that for any $\xi,\eta \in\cT_X$ we must have
\[
   F_{\tilde\na}(\xi,\eta) = [\tilde\na_\xi,\tilde\na_\eta] - \la \cdot \tilde\na_{[\xi.\eta]}.
\]
If $F_{\tilde\na} = 0$, we say that $\tilde\na$ is \emph{flat}.

Vector bundles with flat $\la$-connection correspond to $\O$-flat $\O$-coherent modules over the algebra $\D_\la$ which is the universal enveloping algebra of the Lie algebroid $\cT_{X,\la} = \cT_X$ over $\O$ with rescaled commutator: $[\xi,\eta]_\la = \la[\xi,\eta]$. There is an inclusion $\O_{T^*X\fr} \into Z(\D_\la)$ (the center of $\D_\la$) which is an isomorphism for $\la\ne0$. It is given by $f\fr \mapsto f^p$, $\xi\fr \mapsto \hat\xi^p - \la^{p-1} \hat\xi^{[p]}$ where $\xi$ in the LHS is regarded as a fiberwise linear function on $T^*X$, and $\hat\xi$ in the RHS is the corresponding element in $\D_\la$. For $\la=0$ the inclusion is just the Frobenius map $\Fr^*\colon \O_{T^*X\fr} \to \O_{T^*X}$.

Now suppose we have a flat $\la$-connection~$\tilde\na$ on a vector bundle~$E$.  Then we can then define the \emph{$p$-spectral variety} of~$\tilde\na$ as the support of the corresponding $\D_\la$-module regarded as an $\O_{T^*X\fr}$-module. By the \emph{$p$-curvature map} of~$\tilde\na$ we will mean the map $\curv_p(\na) \colon  \cE \to \cE \tsr \Fr_X^*\Om^1_{X\fr}$ coming from the action of~$\O_{T^*X\fr}$ on~$\cE$.  For $\la\ne0$ it is related to the ordinary $p$-curvature by $\curv_p(\na) = \la^p \curv_p(\la^{-1}\na)$ (here $\la^{-1}\na$ is a usual connection).

\subsection{Twisted differential operator algebras and their centers}
\label{TDO}
If $X$ is a smooth algebraic variety and $\L$ is a line bundle on it, we define a sheaf of algebras $\D_{X,\Lc}$ for any $c\in\k$ as follows.
For any local trivialization $\phi\colon \O_U \isoto\L|_U$ of $\L$ on an open set $U$, we have a canonical isomorphism $\alpha_\phi \colon \D_U \isoto \D_{U,\Lc}$, and if $\phi'$ on $U'$ is another trivialization then the gluing isomorphism $\alpha_{\phi'}^{-1} \circ \alpha_\phi$ is given by
\begin{equation}\label{eq:TDO transition autom}
  \begin{cases}
    f  \mapsto f & \text{for $f\in\O_X$, and} \\
    \xi \mapsto \xi - c\xi(h)/h & \text{for $\xi\in\cT_X$}
  \end{cases}
\end{equation}
where $h \in \O^\times (U\cap U')$ is such that $(\phi')^{-1} \circ \phi \colon \O(U\cap U') \to \O(U\cap U')$ is given by multiplication by $h$.

The algebra $\D_{X,\Lc}$ is called the algebra of \emph{differential operators twisted by $\Lc$}.
This is justified by the fact that if $c\in\Fp$ then for any lift~$\tilde c$ of~$c$ to an integer, the algebra $\D_{X,\Lc}$ acts on the tensor power $\L^{\tsr \tilde c}$ by differential operators (although the action is not exact, just like in the non-twisted case).  In fact, in this case we have an isomorphism of sheaves of algebras $\D_{X,\Lc}\isom\cEnd_{\D_X} (\D_X\tsr_{\O_X}\L^{\tsr\tilde c})^{\on{op}} = \L^{\tsr\tilde c}\tsr_{\O_X} \D_X\tsr_{\O_X} \L^{\tsr-\tilde c}$.  This provides a Morita equivalence between $\D_X$ and $\D_{X,\Lc}$, and under this equivalence, the $\D_X$-module $\O$ corresponds to $\D_{X,\Lc}$-module $\L^{\tsr\tilde c}$.

However, in this paper, we will be interested in the cess $c\not\in\Fp$, where $\Lc$ does not make sense as a line bundle.

Just as the sheaf of usual differential operators, the algebra $\D_{X,\Lc}$ has a filtration with the associated graded isomorphic to (the pushforward of) $\O_{T^*X}$.  In particular, taking the first filtered piece gives an extension of coherent sheaves on~$X$:
\[
    0 \to\O_X \to\D_{X,\Lc}^{\le 1} \to\cT_X \to0.
\]
Moreover, it is not hard to check that the torsor of splittings of this extension is canonically identified with the torsor $\tilde T^*_\L X$ of $c$-connections on~$\L$.

Now let $\tilde\na$ be such a $c$-connection.  Then the obstruction for the corresponding map $l_{\tilde\na} \colon \cT_X \to \D_{X,\Lc}^{\le 1}$ to be a map of Lie algebras is given by the curvature of~$\tilde\na$:
\[
    [l_{\tilde\na}(\xi),l_{\tilde\na}(\eta)] - l_{\tilde\na}([\xi,\eta]) = F_{\tilde\na}(\xi,\eta),
\]
where $F_{\tilde\na} \in\Om^2(X)$ is the curvature of the $c$-connection~$\na$.  Consequently, a structure of a $\D_\Lc$-module on a given quasi-coherent sheaf~$\cE$ is equivalent to a connection~$\na'$ on~$\cE$ with the curvature given by
\[
    F_{\na'} = F_{\tilde\na} \cdot \id_\cE.
\]

\begin{prop} \label{Z(TDO)}
Denote by $Y$ the relative spectrum of the center of $\D_\Lc$: $Y = \Spec_{X\fr} Z(\Fr_{X*} \D_{X,\Lc})$.  Then $Y$ is canonically isomorphic to $\tilde T^*_{\L\fr^{c^p-c}} X\fr$ (as an $X\fr$-scheme).%
    \footnote{To make the notation less cumbersome, we drop the parentheses writing $\L\fr^{c^p-c}$ for $(\L\fr)^{c^p-c}$}
Moreover, $\Fr_{X*} \D_\Lc$ is the pushforward of an Azumaya algebra~$\tilde\D_\Lc$ on~$Y$.

If we are given a trivialization $\phi\colon \O_X \isoto \L$ then the following diagram commutes (where the vertical arrows are induced by $\phi$~and~$\phi\fr$, and the bottom arrow is the standard isomorphism (see e.g., \cite[Lemma~1.3.2]{BMR})):
\[
    \xymatrix{
      Y\ar[r]^-\sim\ar[d]^\wr_\phi & \tilde T^*_{\L\fr^{c^p-c}} X\fr \ar[d]^\wr_{\phi\fr} \\
      \Spec_{X\fr} Z(\Fr_{X*}\D_X) \ar[r]^-\sim & T^*X\fr
    }
\]
\end{prop}

\begin{proof}
  It is enough to construct the identification for the trivial line bundle~$\L$ and then prove that it is independent of the trivialization.  But for the trivial line bundle we already have such an identification (referred to as ``standard'' in the formulation).

  To show the independence of the trivialization, suppose we have an automorphism of the trivial line bundle~$\O_X$ given by an invertible function~$h$.  The corresponding automorphism~$\psi_h$ of~$\D_X$ is given by~\eqref{eq:TDO transition autom}.  Combining with the formulas for the identification of $Z(\D_X)$ with the pushforward of $\O_{T^*X\fr}$:
  \[
      \begin{cases}
        f\fr\mapsto f^p& \text{for $f\in\O_X$;} \\
        \xi\fr\mapsto \xi^p - \xi^{[p]}& \text{for $\xi\in\cT_X$}
      \end{cases}
  \]
  we see that, for any vector field~$\xi$ on (an open subset of) $X$, the action of~$\psi_h$ on the element of $Z(\D_X) \isom \Sym \cT_{X\fr}$ corresponding to~$\xi\fr$ is given by
  \[
      \psi_h(\xi^p - \xi^{[p]})  =  (\xi - c \xi(h) / h)^p  -  (\xi^{[p]} - c \xi^{[p]} (h) / h) .
  \]

  Using the identity (in any associative algebra in characteristic~$p$)
  \[
      (a+b)^p = a^p + b^p + (\on{ad}a)^{p-1}(b) \quad\text{if } [[a,b],b] = 0,
  \]
  the above expression can be rewritten as
  \[
      \psi_h(\xi^p - \xi^{[p]})  =  (\xi^p - \xi^{[p]})  -  c^p (\xi(h) / h)^p  +  c\bigl(\xi^p(h) / h  -  \xi^{p-1}(\xi(h) / h) \bigr).
  \]

  Now note that when $c=1$ we have $\psi_h = \on{Ad}h$.  Therefore, since $\xi^p-\xi^{[p]}$ is central, we must have $\psi_h(\xi^p-\xi^{[p]}) = \xi^p-\xi^{[p]}$ in this case.  Thus, the above equation gives
  \[
      (\xi(h) / h)^p  +  \xi^p(h) / h  -  \xi^{p-1}(\xi(h) / h) = 0,
  \]
  and hence, for arbitrary~$c$, the formula becomes
  \[
      \psi_h(\xi^p - \xi^{[p]})  =  (\xi^p - \xi^{[p]})  -  (c^p-c) (\xi(h) / h)^p .
  \]
  We see that this formula coincides with the action of~$h$ on linear functions on the twisted cotangent $\tilde T^*_{\L\fr^{c^p-c}} X\fr$.
\end{proof}

\begin{rem}
  Another way to finish the argument is as follows.  Let us observe that the effect of~$\psi_h$ on $\D_X$-modules is equivalent to tensoring by the line bundle with connection $(\O,d+c\cdot d\log h)$.  Therefore its effect on the $p$-support is given by the shift by the $p$-curvature of $d+c\cdot d\log h$.  This $p$-curvature equals $(c\cdot d\log h)\fr - \sC(c\cdot d\log h) = (c^p-c)\cdot d\log h\fr$ (where $\sC$ is the Cartier operator), hence we get the desired formula.
\end{rem}

The proposition implies that $\D_{X,\Lc}$ localizes to an Azumaya algebra on $\tilde T^*_{\L\fr^{c^p-c}}X\fr$ which we will denote by $\tilde\D_{X\fr,\Lc}$ or just $\tilde\D_\Lc$.  If $c\not\in\Fp$, we'll use the identification $\tilde T^*_{\L\fr^{c^p-c}}X\fr \isoto \tilde T^*_{\L\fr} X\fr$ and think of $\tilde\D_\Lc$ as an Azumaya algebra on $\tilde T^*_{\L\fr} X\fr$.

\medskip
We finish this subsection with a discussion of ``TDO algebras with Planck's constant''.

Suppose $X,\L$ are as in \ref{TDO}. Define the $\k[\bc,\hb]$-algebra $\DchL(X)$ as follows. Let $\pi\colon  \tilde X \to X$ be the principal $\Gm$-bundle corresponding to $\L$. Denote by $\Dh(\tilde X)$ the algebra of ``differential operators with parameter,'' that is, the algebra
\[
   \Dh(\tilde X) := \bigoplus_{n\ge0} \D^{\le n}(\tilde X)
\]
over $\k[\hb]$ where the inclusion $\k[\hb] \into \Dh(\tilde X)$ is given by $\hb\mapsto 1\in\D^{\le1}(\tilde X)$. Here we introduce its TDO analog. For $\xi\in\cT_{\tilde X}$ let $\hat\xi$ be the corresponding element in $\D^{\le1}\subset\Dh$. Let $\Eu$ be the Euler vector field on $\tilde X$ (the differential of the $\Gm$-action). Now set
\[
   \DchL(X) := (\pi_*\Dh(\tilde X))^\Gm.
\]
This is a $\k[\bc,\hb]$-algebra via $\hb\mapsto\hb\in\Dh(\tilde X)$, $\bc\mapsto\widehat\Eu$. Note also that $\pi_*\Dh (\tilde X)$ has two gradings: one comes from the definition of $\Dh$ as a direct sum, and the other comes from the $\Gm$-action on $\tilde X$. But on the $\Gm$-invariant part, we have only one grading (the first one), and with respect to this grading $\deg \bc = \deg\hb = 1$. The algebra $\DchL$ being graded implies that it carries an action of $\Gm$ and, in particular, if $\DchL[c,\hb_0]$ denotes the specialization $\bc\mapsto c,\ \hb \mapsto \hb_0$ of the algebra $\DchL$ (where $c,\hb_0 \in\k$), \ie the algebra of ``$(c\cdot c_1(\L))$-twisted $\hb$-differential operators'', then
\begin{equation} \label{DchL scaling}
 \DchL[c,\hb_0] \isom \DchL[\la c,\la\hb_0]
 \end{equation}
for any $\la\in\k^{\times}$.

The specialization $\bc \mapsto 0$ gives the algebra $\Dh$ defined above, and in particular, $\DchL[0,0]=(\pr_X)_*\O_{T^*X}$ where $\pr_X\colon T^*X\to X$. Furthermore, it is not hard to show that $\DchL[c,0] = (\pr'_X)_* \O_{\tilde T^*_\Lc X}$ (where $\pr'_X$ is again the appropriate projection to $X$). One can also check that specialization $\hb\mapsto1$ recovers the algebra $\D_\Lc$ from~\ref{TDO}. Taking the isomorphism~\eqref{DchL scaling} into account, we can summarize:
\[
 \DchL[c,\hb_0] \isom
 \begin{cases}
    \pr_*\O_{T^*X} & \text{if $c=\hb_0=0$;} \\
    \pr'_*\O_{\tilde T\*_{\L} X} & \text{if $c\ne0,\ \hb_0=0$;} \\
    \D_{\L^{c/\hb_0}} & \text{if $\hb_0\ne0$.}
 \end{cases}
\]

The following proposition is a generalization of Proposition~\ref{Z(TDO)}:
\nopagebreak\medskip\mynobreakpar
\begin{prop*}[\ref{Z(TDO)}$'$]   \label{Z(DchL)} \ \nopagebreak\smallskip\mynobreakpar
\begin{enumerate}
  \item  The center of the algebra $\DchL(X)$ is canonically isomorphic to $\O_{\rZ}$ where $\rZ$ is a scheme over $\k[\bc,\hb]$ whose fiber over a (closed) point $(c,\hb_0)\in\Spec\k[\bc,\hb]$ is isomorphic to\footnote{%%
            It is clear that these twisted cotangent bundles organize into a family over $\Spec\k[c,\hb]$.  More precisely $\rZ$ should be defines by the appropriate relative (in-families) version of the construction $\tilde T^*_\Lc X$.}
            %$\rZ$ is a scheme over $\k[\bc,\hb]$. One should extend the definition of $\tilde T^*_{\L^\la} X$ to the case of families over an arbitrary scheme in order to make sense of the definition of $\rZ$.}
            \[
                \rZ\xx_{\Spec\k[\bc,\hb]}\{(c,\hb_0)\} \:\isom\:\tilde T^*_{\L\fr^{c^p-c\hb_0^{p-1}} } X\fr.
            \]
  \item  Moreover, if $c,\hb_0 \in\k$, the specialization $\bc\mapsto c,\ \hb\mapsto\hb_0$ induces a map $\O_{\rZ_{c,\hb_0}} \to Z(\DchL[c,\hb_0])$ which is an isomorphism if and only if $\hb_0\ne0$, in which case $\DchL[c,\hb_0]$ is an Azumaya algebra over $\rZ_{c,\hb_0}$.
  \item  The isomorphism~\eqref{DchL scaling} is compatible with the $\Gm$-action on $\rZ$ given by scaling connections by $\la^p$.
\end{enumerate}
\end{prop*}

\subsection{Central reductions}\label{subs:c.red}
Suppose $X,\L$ are as above, $c\in\k$. Then any $(c^p-c)$-connection $\tilde\na_0$ on $\L\fr$ gives a section of the bundle $\tilde T^*_{\L\fr^{c^p-c}} X\fr$ over $X\fr$ defined above. Denote by $\D_{\Lc,\tilde\na_0}$ the pullback of $\tilde\D_\Lc$ to this section. It is an Azumaya algebra on $X\fr$. If $c\not\in\Fp$ then $(c^p-c)$-connections on~$\L\fr$ correspond bijectively to ordinary connections on~$\L\fr$ by multiplication by $c^p-c$, and, by a slight abuse of notation, we will sometimes denote the algebra $\D_{\Lc,\tilde\na_0}$ by $\D_{\Lc,\na_0}$ where $\na_0$ is the connection on~$\L\fr$ for which $\tilde\na_0 = (c^p-c)\na_0$. One can check that in this case modules over $\D_{\Lc,\tilde\na_0}$ correspond to pairs $(\cE,\na)$ where $\cE$ is a quasi-coherent sheaf on $X$ and $\na$ is a connection on $\cE$ such that
\begin{align*}
F_\na &= c \cdot F_{\na_0'}; \\
\na_\xi^p - \na_{\xi^{[p]}} &= c\cdot \bigl( (\na_0')_\xi^p - (\na_0')_{\xi^{[p]}} \bigr)
\end{align*}
for any $\xi\in\cT_X$, where $\na_0'$ is a (usual) connection on $\L$ such that $\na_0 = \na_0'^{(1)}$. (The operators in the RHS are always multiplication by a function\footnote{which is a pullback from $X\fr$} (resp.\ a two-form), and we want the LHS to be multiplication by the same function (resp.\  two-form), though on a different bundle.)  We will give another interpretation of these conditions in~\ref{extcrv}.

%More generally, for arbitrary $c,\hb\in\k$, to describe what $\DchL<\L,\tilde\na_0>$-modules are, choose any connection $\na_1$ on $\L$. Then $\DchL<\L,\tilde\na_0>$-modules ($\tilde\na_0$ is a $(c^p-c\hb^{p-1})$-connection) correspond to pairs $(\cE,\tilde\na)$ where $\cE$ is a bundle (or a quasi-coherent sheaf) and $\tilde\na$ is an $\hb$-connection on $\cE$ satisfying
%\begin{align}
%F_{\tilde\na} &= c\hb \cdot F_{\na_1}; \label{crv} \\
%%\begin{split}
%\tilde\na_\xi^p - \hb^{p-1}\tilde\na_{\xi^{[p]}} &= c\hb^{p-1}\cdot \bigl( (\na_1)_\xi^p - (\na_1)_{\xi^{[p]}} \bigr) - \Fr_X^* \langle\alpha,\xi\fr\rangle\label{pcrv} % \\
%%    &= \Fr_X^* \langle  c\hb^{p-1}\beta - \alpha, \xi\fr\rangle
%%\end{split}
%\end{align}
%for any $\xi\in\cT_X$, where $\alpha = \tilde\na_0 - (c^p-c\hb^{p-1}) \na_1\fr$ is a $1$-form on $X\fr$. The formulas from previous paragraph can be obtained by substitution $\alpha=0$. If $\na_1$ is replaced by $\na_1 + \beta$ then $\tilde\na$ is replaced by $\tilde\na + c\beta$. If $\na_1$ is flat then the RHS of \eqref{pcrv} can be expressed as  $\Fr_X^* \langle c\hb^{p-1}\theta - \alpha, \xi\fr\rangle$ where $\theta=\curv_p(\tilde\na_1)$, so the condition~\eqref{pcrv} is equivalent to the connection $\na_1$ being flat and having $p$-curvature $c\hb^{p-1}\theta-\alpha$. In particular, if
%\[
%   \tilde\na_0=(c^p-c\hb^{p-1}) \na_1\fr+ c\hb^{p-1}\theta
%\]
%then $c\hb^{p-1}\theta-\alpha= 0$ and therefore the algebra $\DchL<\L,\tilde\na_0>$ canonically splits.
%
%We will denote by $\D_{\Lc,\tilde\na_0}$ the specialization of $\DchL<\L,\tilde\na_0>$ at $\hb=1$.
Now assume that $c \not\in \Fp$.
Let $\tilde T^*_\L X$ be the twisted cotangent bundle associated to~$\L$.
%For any $\la \in \k$,
The pullback~$\L'$ of~$\L$ to~$\tilde T^*_\L X$ %$\tilde T^*_{\L\fr^\la} X\fr$
has a canonical ``universal'' %$\la$-%
connection which we will denote by~$\na_{\on{can}}^{\L}$.

\begin{prop}\label{prop:Morita eqv}
  There is a canonical Morita equivalence of Azumaya algebras on~$(\tilde T^*_\L X)\fr$:
  \[
      \tilde\D_\Lc \sim \D_{\L'^c,(\na_{\on{can}}^{\L})\fr}.
  \]
\end{prop}

\begin{proof}
  First we will construct a functor $\tilde\D_\Lc\hmod \to \D_{\L'^c,(\na_{\on{can}}^{\L})\fr}\hmod$ and then explain why it is compatible with the isomorphism of $p$-centers and is an  equivalence.  Let $\pi$ denote the projection $\tilde T^*_\L X \to X$.  We define the desired functor as the composition of the following functors:
  \[
      \tilde\D_\Lc\hmod \isoto \D_\Lc\hmod \xra{\pi^!} \D_{\L'^c}\hmod \to \D_{\L'^c,\na_{\on{can}}^{\L}}\hmod.
  \]
  Here $\pi^!$ is the usual pullback for twisted $\D$-modules (given by the $\O$-module pullback of the underlying quasi-coherent sheaves), and the last functor is given by induction (\ie central reduction).

  In  order to check that this functor is $\O_{(\tilde T^*_\L X)\fr}$-linear and is an equivalence, it is enough to consider the case when $\L$ is trivial, which reduces to the analogous statement for non-twisted $\D$-modules.  This was proved in~\cite{BB}.
\end{proof}

%It is possible to reformulate the conditions \eqref{crv} and \eqref{pcrv} in a more invariant way using the following construction (we will do it only for $\hb=1$).

\subsection{TDOs on stacks}\label{subs:TDO stk}
\ncmd{\Ysm}{\rY_{\on{sm}}}
In this section we are going to generalize the above results to the case of stacks.  So let $\rY$ be a smooth pure-dimensional Artin stack over~$\k$, $\L$ a line bundle on~$\rY$, and $c \in \k$.  The $\D$-modules and twisted $\D$-modules on stacks are discussed in \cite[Sect.~1.1]{BD} for the characteristic~0 case.  For the non-twisted case in characteristic~$p$, see \cite[Sect.~3.13]{BB}.

Recall that a quasi-coherent sheaf, resp.\ $\D$-module, on~$\rY$ is defined as a datum, for every smooth morphism $S\to\rY$ from a scheme~$S$, of a quasi-coherent sheaf, resp.\ $\D$-module, $\cF_S$ on~$S$, together with identifications $f^*_\O\cF_S \isoto \cF_{S'}$, resp.\ $f^!_\D\cF_S \isoto \cF_{S'}$ ($f^!_\D$ is the $\D$-module !-pullback which is the $*$-pullback $f^*_\O$ on the underlying $\O$-modules) for every morphism $f\colon S' \to S$ of smooth schemes over~$\rY$ (\ie $S$ runs over the smooth site~$\Ysm$ of~$\rY$).  These identifications are required to satisfy the cocycle condition for compositions.  Similarly, one can define the category of $\Lc$-twisted $\D$-modules on~$\rY$.  In this subsection we will construct a certain $\Gm$-gerbe $\dd_{\rY,\Lc}$ on $\tilde T^*_{\L\fr^{c^p-c}}\rY\fr$ such that modules over it are equivalent to $\Lc$-twisted $\D$-modules on~$\rY$.\footnote{see Remark~\ref{rem:strengthen BB} below}  For the sake of brevity of notation, we'll denote the latter twisted cotangent by~$\tilde T^*\rY\fr$, and for $(S,\pi) \in \Ysm$, denote by~$\tilde T^*S\fr$ the twisted cotangent bundle for $(\pi^*\L)\fr^{c^p-c}$.

\ncmd\tysm{(\tilde  T^* \rY\fr)_{\on{sm}}}

First of all, one can show that, in order to define a $\Gm$-gerbe on~$\tilde T^*\rY\fr$, it's enough to provide a $\Gm$-gerbe on $S\fr \xx_{\rY\fr} \tilde T^*\rY\fr$ for every $S\in\Ysm$ (together with compatibility isomorphisms).  Indeed, a~priori we have to define a $\Gm$-gerbe on any $T\in\tysm$.  But let us show that the $\Gm$-gerbes on other $T\in\tysm$ are recovered uniquely (up to a unique equivalence of  gerbes).  Indeed, fix a smooth atlas $U\to\rY$. Then, for any  $T\in \tysm$, we construct a gerbe on $U\fr\xx_{\rY\fr} T$ by descent from the \v{C}ech nerve of $U\fr\xx_{\rY \fr} T\to T$. We use the notation
\[
    \rX^n_\rZ:=\underbrace{\rX\x_\rZ \dots \x_\rZ\rX}_n
\]for a map of stacks  $\rX\to\rZ$.   Then, for any $n$, we have $U^n_\rY\in\Ysm$, so we already have a gerbe on $ (U\fr\x_{\rY\fr} {\tilde T^*\rY\fr})^n_{\tilde T^*\rY\fr} = (U^n_\rY)\fr \x_{\rY\fr} {\tilde T^*\rY\fr}$, and hence by pullback on $ (U\fr\x_{\rY\fr} T)^n_T = (U^n_\rY)\fr \x_{\rY\fr} T$.  Since $\Gm$-gerbes descend along \v{C}ech nerves of smooth morphisms, we thus get a gerbe on $T$.  Compatibilities for maps between $T$'s are easy from the construction.

%and assume that $T$ is an affine scheme (clearly, this is enough).     % one can show that the subcategory in $(\tilde  T^* \rY)_{\on{sm}}$ formed by $T$ of the form $S\fr \xx_{\rY\fr} \tilde T^*\rY\fr$ is \emph{homotopy cofinal}, so limits of functors from $(\tilde  T^* \rY)_{\on{sm}}$ to a higher category coincide with limits over this subcategory.  We will need in particular the functor to the $(3,1)$-category of $2$-groupoids, which takes $S$ to the $2$-groupoid of $\Gm$-gerbes on $S$.

Now we have a map $S\fr \xx_{\rY\fr} \tilde T^*\rY\fr \to \tilde T^*S\fr$, and we have a gerbe $\dd_{S,\pi^*\Lc}$ on~$\tilde T^*S\fr$ corresponding to the Azumaya algebra~$\tilde\D_{S,\pi^*\Lc}$, so we can let $(\dd_{\rY,\Lc})_{S\fr}$ be the pullback of~$\dd_{S,\pi^*\Lc}$ under this map.  These gerbes are compatible with pullbacks: the equivalence is given by the TDO analog of the $\D$-module $\D_{S' \to S}$ from~\cite{BB}.  The compatibility of these equivalences with compositions follows from the isomorphism $\D_{S'' \to S'} \star \D_{S' \to S} \isom \D_{S'' \to S}$ for a pair of morphisms $S'' \xra{g} S' \xra f S$, which in turn follows from the isomorphism $(fg)^! = g^!f^!$ for $\D$-module pullbacks, and similarly for the cocycle condition.

For a $\Gm$-gerbe~$\g$, we denote $\g\hmod$ the category of ``$\g$-twisted quasi-coherent sheaves,'' \ie the 1st component of the category of quasi-coherent sheaves on the ``total space'' of~$\g$.

\begin{thm}\label{Dmod-stk equiv gerbe}
  There is a natural equivalence of categories
  \[
      \D\hmod_\Lc(\rY)  \isoto  \dd_{\rY,\Lc}\hmod.
  \]
\end{thm}

\begin{proof}
  We  first construct a functor in one direction.  Namely, given an object $\tilde\cM \in \dd_{\rY,\Lc}\hmod$, we can construct an $\Lc$-twisted $\D$-module $\cM$ on~$\rY$ as follows.  Given any $(S,\pi) \in \Ysm$, we have maps $(\wtld{d\pi})\fr \colon S\fr \xx_{\rY\fr} \tilde T^*\rY\fr  \to  \tilde T^*S\fr$ and $\tilde \pi\fr\colon S\fr \xx_{\rY\fr} \tilde T^*\rY\fr  \to  \tilde T^*\rY\fr$.  So we can consider the $\dd_{S,\pi^*\Lc}$-module given by $(\wtld{d\pi})\fr_* \tilde \pi\fr^* \tilde\cM$; this makes sense because by definition we have the canonical equivalence $(\wtld{d\pi})\fr^* \dd_{S,\pi^*\Lc}  \isom  \tilde \pi\fr^* \dd_{\rY,\Lc}$.  We let $\cM_S$ be the $\D_{S,\pi^*\Lc}$-module corresponding to this $\dd_{S,\pi^*\Lc}$-module.  It is clear that for a map $f\colon S' \to S$ in~$\Ysm$ we have an isomorphism $\cM_{S'} \isom f^!\cM_S$ and these isomorphisms are compatible with compositions, so the modules~$\cM_S$ for every $S\in\Ysm$ define an $\Lc$-twisted $\D$-module $\cM$ on~$\rY$.  Moreover the assignment $\tilde\cM \mapsto \cM$ gives a functor $\dd_{\rY,\Lc}\hmod  \to  \D\hmod_\Lc(\rY)$, which we (temporarily for this proof) denote by~$\Phi$.

  Note that for any $(S,\pi)\in\Ysm$, since $\pi$ is smooth by assumption, the map $\wtld{d\pi}$ (and hence $(\wtld{d\pi})\fr$) is a closed embedding.  Therefore the functor $(\wtld{d\pi})\fr_*$ is fully faithful.  Being by definition a limit of such functors, our functor~$\Phi$ is thus fully faithful as well.  It remains to show that it is essentially surjective.  Fix an object $\cM \in \D\hmod_\Lc(\rY)$.  We see that $\cM$ is in the essential image of~$\Phi$ if and only if, for every $(S,\pi) \in \Ysm$, the $\D_{S,\pi^*\Lc}$-module $\pi^!\cM = \cM_S$ has $p$-support inside $S\fr \xx_{\rY\fr} \tilde T^*\rY\fr \ctd \tilde T^*S\fr$.

  To see that the $p$-support of~$\cM_S$ is indeed contained in this closed subscheme, consider the fiber square $Q = S \xx_\rY S$ and denote by $\pr_1,\pr_2\colon Q \toto S$ the two projections.  We have an isomorphism $\pr_1^!\cM_S \isom \cM_Q \isom \pr_2^!\cM_S$ of twisted $\D$-modules on~$Q$.  But note that the $p$-support of~$\pr_i^!\cM_S$ is inside $Q\fr\xx_{\pr_i\fr,S\fr}\tilde T^*S\fr$, and the above isomorphism forces this $p$-support to lie in the intersection of these two closed subschemes inside $\tilde T^*Q\fr$.  This intersection coincides with $Q\fr \xx_{\rY\fr} \tilde T^*\rY\fr$.  Since $\pr_i$ are surjective, the equality $\supp_p \pr_i^!\cM_S = Q\fr \xx_{S\fr} \supp_p \cM_S$ implies that $\supp_p\cM_S \ctd S\fr \xx_{\rY\fr} \tilde T^*\rY\fr$ as desired.
\end{proof}

\begin{rem}\label{rem:strengthen BB}
  In~\cite{BB}, the authors construct an coherent sheaf of algebras on the stack~$T^*\rY$ which is an Azumaya algebra on the \emph{Deligne--Mumford part} of $T^*\rY$ and think of $\D$-modules as modules over that algebra.  But if one is concerned only with the \emph{categories}, not the algebras, no such restriction is needed: although the Azumaya algebra (as such) does not extend from the Deligne--Mumford part, the corresponding gerbe does, and describes the category of $\D$-modules on~$\rY$ as defined in this subsection.
\end{rem}

Now we will prove an analog of Proposition~\ref{prop:Morita eqv} for the stack case.  For a smooth stack~$\rY$, a line bundle~$\L$ on~$\rY$, a scalar $c \in \k\setm\Fp$, and a connection~$\na$ on~$\L\fr$, define a $\Gm$-gerbe $\dd_{\rY,\Lc,\na}$ on~$\rY\fr$ as follows.  By definition, the connection~$\na$ defines a section~$s_\na$ of the projection $(\tilde T^*_\L\rY)\fr \to \rY\fr$.  Let $\dd_{\rY,\Lc,\na}$ be the pullback of~$\dd_{\rY,\Lc}$ by this section.

\begin{prop}\label{prop:Morita eqv stk}
  For $\rY$, $\L$ and $c$ as above, we have a canonical equivalence of gerbes over $((\tilde T^*_\L\rY)\fr)\sm$:
  \[
      \dd_{\rY,\Lc} |_{((\tilde T^*_\L\rY)\fr)\sm}  \isoto  \dd_{(\tilde T^*_\L\rY)\sm,\L'^c,\na_{\on{can}}\fr}
  \]
  where $(\L',\na_{\on{can}})$ is the pullback of~$\L$ to~$(\tilde T^*_\L\rY)\sm$ equipped with the canonical connection.
\end{prop}

The proposition will follow from the following lemma:

\begin{lem}\label{lem:eqvce on crspce stk}
  Let $f\colon \rY \to \rZ$ be a map of smooth stacks and $\L$ a line bundle on~$\rZ$.  Let $\tilde f\fr\colon \rY\fr \xx_{\rZ\fr} \tilde T^*_{\L\fr} \rZ\fr \to \tilde T^*_{\L\fr} \rZ\fr$ and $\wtld{df}{}\fr\colon \rY\fr \xx_{\rZ\fr} \tilde T^*_{\L\fr} \rZ\fr \to \tilde T^*_{f^*\L\fr} \rY\fr$ be the natural projections.  Then we have an equivalence
  \[
      \tilde f\fr^* \dd_{\rZ,\Lc} \isoto \wtld{df}{}\fr^* \dd_{\rY,f^*\Lc}.
  \]
\end{lem}

\begin{proof}
  We have already mentioned this fact in the case when $\rY$ and~$\rZ$ are smooth varieties.  In this case, the statement follows from the fact that there is a splitting of $\tilde\D_{\rY,f^*\Lc} \bx \tilde\D_{\rZ,\Lc}^{\on{op}}$ given by $\D_{\rY\to\rZ,\Lc}$ (see \cite[Proposition~3.7]{BB} for the non-twisted case; the twisted case is completely analogous).

  The stack case can easily be obtained by descent.
\end{proof}

\begin{proof}[Proof of Proposition~\ref{prop:Morita eqv stk}]
  We apply Lemma~\ref{lem:eqvce on crspce stk} to the map $f\colon (T^*_\L\rY)\sm \to\rY$ and the line bundle~$\L$.  Consider the diagonal map $\Delta\colon (\tilde T^*_\L\rY)\sm \to (\tilde T^*_\L\rY)\sm \xx_\rY \tilde T^*_\L\rY$.  It is easy to see that the map $\tilde f \circ\Delta$ is the inclusion $(\tilde T^*_\L\rY)\sm \into \tilde T^*_\L\rY$ and the map $\wtld{df}\circ\Delta\colon (\tilde T^*_\L\rY)\sm \to \tilde T^*_{\L'} (\tilde T^*_\L\rY)\sm$ corresponds to the connection~$\na_{\on{can}}$ on~$\L'$.  Thus, pulling back the equivalence in Lemma~\ref{lem:eqvce on crspce stk} by the map~$\Delta$, we get the desired statement.
\end{proof}

\medskip    To compare with the treatment of $\D$-modules on stacks in \cite{BB},
we now discuss what might be called the ``free  rank one'', or ``regular'' twisted $\D$-modules on stacks.

It is clear that the forgetful functor $\D\hmod_\Lc(\rY) \to \on{QCoh}(\rY)$ has a left adjoint.  Denote by $\Ds$ the image of~$\O_\rY$ under this left adjoint; this is a twisted $\D$-module (in particular, a quasi-coherent sheaf) on~$\rY$.  In other words, this is the object (co)representing the functor of global sections $\Gamma\colon \D\hmod_\Lc(\rY) \to \on{QCoh}(\rY) \to \k\text{-Vect}$.  Let also $\Df$ be the opposite to the sheaf (on $\Ysm$) of endomorphism algebras of $\Ds$ as a twisted $\D$-module, \ie $(\Df)_S = \cEnd_{\D_{S,\pi^*\Lc}} (\Ds)_S$.  Since $\O_{S\fr} \ctd Z(\D_{S,\pi^*\Lc})$ for every $(S,\pi) \in \Ysm$, we can regard $\Df$ as an $\O_{\rY\fr}$-module.  One can check that for any map $f\colon S'\to S$ in~$\Ysm$, we have $(\Df)_{S'} = f\fr^* (\Df)_S$, so $\Df$ is a quasi-coherent sheaf on~$\rY\fr$.

To describe these sheaves in more concrete terms, fix $(S,\pi) \in \Ysm$ and consider the left ideal~$\cI$ in $\D_{S,\pi^*\Lc}$ generated by the image of the map $\cT_{S/\rY} \to \D_{S,\pi^*\Lc}^{\le1}$ (this is the canonical lift of the map $\cT_{S/\rY} \to \cT_S$ which comes from the fact that the vector fields on~$S$ coming from $\cT_{S/\rY}$ have canonical lifting to the total space of $\pi^*\L$), and let $N(\cI)$ be its normalizer.  Then $(\Ds)_S = \D_{S,\pi^*\Lc} / \cI$ and $(\Df)_S = N(\cI) / \cI$.

\begin{rem}
  In the non-twisted case, the sheaves $\D^\sharp_{\rY}$ and $\D^\flat_{\rY}$ are what is denoted by $\D_\cY$ and $D_\cY$ respectively in~\cite{BD} and by $D_Y^\sharp$ and $D_Y$ respectively in~\cite{BB}.  In~\cite{BB}, the authors state that the definition of~$\D^\flat_\rY$ should be modified in characteristic~$p$ because they want $(\D^\flat_\rY)_{S'}$ to be the Zariski (or \'etale) sheaf-theoretic inverse image of $(\D^\flat_\rY)_S$ for any morphism $S' \to S$ in $\Ysm$, as it happens in the characteristic~$0$ case.  Although this property doesn't hold in characteristic~$p$, the fact that we have a quasi-coherent sheaf on~$\rY\fr$ can be regarded as a substitute.
\end{rem}

Recall that by definition, $\Df$ is the opposite of the endomorphism sheaf of $\Ds$, so $\Ds$ is a right module over $\Df$.  We also have a canonical ``unit'' global section of $\Ds$.  Thus there is a map $u\colon \Df \to \Ds$ of sheaves on~$\Ysm$.

\begin{clm}
  The map~$u$ induces an identification $\Df \isoto \Fr_{\rY*}\Ds$.  That is, for $(S\xra\pi\rY) \in \Ysm$, there is an isomorphism $\Gamma(S,(\Df)_S) \isoto \Gamma(S\fr\xx_{\rY\fr}\rY,(\pi\fr\xx_{\rY\fr}\rY)^*\Ds)$ induced by $u$.  (All functors are $\O$-module functors.)
\end{clm}

Note that for any $(S,\pi) \in \Ysm$ the quasi-coherent sheaves $(\Ds)_S$ and $(\Df)_S$ on~$S\fr$ are push-forwards from $\tilde T^*S\fr$, since they are acted on by the center of $\D_{S,\pi^*\Lc}$.  Denote by $(\tDs)_S$ and $(\tDf)_S$ the corresponding sheaves on $\tilde T^*S\fr$.  As we saw in the proof of Theorem~\ref{Dmod-stk equiv gerbe}, these sheaves are actually supported on $S\fr \xx_{\rY\fr} \tilde T^*S\fr$.  Moreover, the $\O_S$-module structure on $(\Ds)_S$ allows us to view $(\tDs)_S$ as a quasi-coherent sheaf on $S \xx_{\rY\fr} \tilde T^*\rY\fr$.  Thus we see that the collection of $(\tDs)_S$ for all $S\in\Ysm$ defines a quasi-coherent (actually, coherent) sheaf $\tDs$ on $\rY \xx_{\rY\fr} \tilde T^*\rY\fr$, whereas $(\tDf)_S$ define a coherent sheaf $\tDf$ on~$\rY\fr$.  It is clear from the above claim that $\tDs$ is the pushforward of $\tDs'$ along the map $\rY \xx_{\rY\fr} \tilde T^*\rY\fr \to \tilde T^*\rY\fr$.

\begin{prop}\label{prop:good => Azumaya}
  Assume that $\rY$ is ``good'' in the sense of \cite[Sect.~1.1.1]{BD}, \ie its cotangent stack has the expected dimension: $\dim T^*\rY = 2\dim\rY$.  Denote by $(\tilde T^*\rY\fr)\sm$ the maximal smooth open substack of~$\tilde T^*\rY\fr$,\footnote{It is easy to see that $(\tilde T^*\rY\fr)\sm$ is a smooth Deligne--Mumford stack.} and by~$\cF$ the coherent sheaf
   on $\dd_{\rY,\Lc}$ corresponding to $\Ds$.   Then
  \begin{enumerate}
  \renewcommand{\theenumi}{\alph{enumi}}
    \item\label{good => Azumaya Ds} The restriction of $\cF$ to $(\tilde T^*\rY\fr)\sm$ is locally free of rank $p^{\dim\rY}$.
    \item\label{good => Azumaya tDf} The restriction of the algebra $\tDf$ to $(\tilde T^*\rY\fr)\sm$ is an Azumaya algebra of rank $p^{2\dim\rY}$.
    \item\label{good => Azumaya G} The gerbe $\dd_{\rY,\Lc}|_{(\tilde T^*\rY\fr)\sm}$ classifies splittings of $\tDf|_{(\tilde T^*\rY\fr)\sm}$.
  \end{enumerate}
\end{prop}

\begin{proof}
  We first note that (\ref{good => Azumaya tDf})~and~(\ref{good => Azumaya G}) are direct consequences of (\ref{good => Azumaya Ds}).  Indeed, (\ref{good => Azumaya tDf}) follows since $\tDf$ is the endomorphism sheaf of~$\cF$, and the identification in~(\ref{good => Azumaya G}) is induced by $\cF$.

  Now, to prove (\ref{good => Azumaya Ds}), it is enough to show that for any $(S,\pi)\in\Ysm$ the sheaf~$\cF_S$ on $S\fr \xx_{\rY\fr} \dd_{\rY,\Lc}$ given by pullback of~$\cF$ is locally free on $S\fr \xx_{\rY\fr} (\tilde T^*\rY\fr)\sm \ctd S\fr \xx_{\rY\fr} \tilde T^*\rY\fr$ of rank $p^{\dim\rY}$.  We will achieve that by showing that the twisted $\D$-module $(\Ds)_S$ corresponding to~$\cF_S$ is Cohen--Macaulay of depth $k := \dim_\rY S = \dim S - \dim\rY$.  (We assume that this dimension is constant along~$S$, for example, it is true if $S$ is connected.)

  Denote $\D_S' = \D_{S,\pi^*\Lc}$ for short.  We can express $(\Ds)_S$ as the homology in the last term (which we put in degree~$0$) of the complex
  \begin{equation}\label{eq:Koszul}
      \D_S' \tsr_{\O_S} \Lambda^k\cT_{S/\rY} \to \D_S' \tsr_{\O_S} \Lambda^{k-1}\cT_{S/\rY} \to\dots\to \D_S' \tsr_{\O_S} \cT_{S/\rY} \to \D_S'
  \end{equation}
  where the rightmost map is constructed as in the first paragraph in this subsection, and the other maps are obtained from it by the Leibnitz rule.  There is a natural filtration by degree on this complex, and the associated graded is isomorphic to the Koszul complex for $\Sym\cT_S \Ltsr_{\Sym\cT_{S/\rY}} \O_S$.  The goodness condition on~$\rY$ implies that this Koszul complex has cohomology only in degree~$0$.  (This complex can be thought of as functions on $S \xx_\rY T^*\rY$ where $T^*\rY$ is understood in the derived sense, and the goodness condition guarantees that $T^*\rY$ is actually non-derived.)  Therefore the same is true for the complex~\eqref{eq:Koszul}.  So \eqref{eq:Koszul} is a locally free resolution of $(\Ds)_S$.

  Now, it is easy to see that applying Verdier duality to~\eqref{eq:Koszul} gives a complex with the associated graded given by the same Koszul complex up to a twist by $\pi^*\omega_\rY[\dim\rY]$, hence this dual complex also has cohomology in one degree---namely in degree~$k$.  Since the Verdier duality for $\D$-modules agrees with the Serre duality for modules over the gerbe~$\dd_{S,\pi^*\Lc}$ up to a twist by line bundle and the identification of $\tilde T^*_{\pi^*(\L\fr)^{c^p-c}}S\fr$ with $\tilde T^*_{\pi^*(\L\fr)^{-(c^p-c)}}S\fr$ via multiplication by~$-1$, we see that $(\Ds)_S$ is Cohen--Macaulay of depth~$k$, as desired.

  To see that the generic rank of~$\cF_S$ is equal to~$p^{\dim\rY}$, we have to show that the generic rank of $(\tDs)_S$ as an $\O_{\tilde T^*S\fr}$-module along its support is $p^{\dim\rY} \cdot \sqrt{\rk\tilde\D_{S,\pi^*\Lc}} = p^{\dim\rY+\dim S}$.  But this can be checked on the level of the associated graded module, which reduces to the fact that the pushforward of $\O_{S \xx_\rY T^*\rY}$ under the Frobenius map has that generic rank, which is true since $\dim (S \xx_\rY T^*\rY) = \dim\rY + \dim S$ by goodness assumption.
\end{proof}

\subsection{Extended curvature}\label{extcrv}
Let $X$ be a smooth variety. Define a (coherent) sheaf $\eF_X$ of $\O_X^p$-modules (= coherent sheaf on $X\fr$) by the exact sequence
\begin{equation}  \label{F_X}
0\to\O_{X\fr} \xra{\Fr^\#_X} \O_X\xra{d}\Om_X^1 \xra{\delta} \eF_X \to 0\,.
\end{equation}
Here $\Fr^\#_X$  is the  morphism $\O_{X\fr} \to \Fr_{X*}\O_X$  induced by the morphism of schemes $\Fr_X\colon  X\to X\fr$.
  (Recall that since $\Fr_X$ induces homeomorphism in the Zariski topology, we identify sheaves on $X$ and on  $X\fr$, and so write $\O_X$ instead of $\Fr_{X*}\O_X$.)
 Sections of $\eF_X$ will be called
   \emph{generalized one-forms} on~$X$.
   Now, from~\eqref{F_X}, we also get an exact sequence
\begin{equation} \label{PQ exseq}
0\to \Om_{X\fr}^1 \xra{\sP} \eF_X\xra{\sQ} \Om_{X,\mathit{cl}}^2 \xra{\sC} \Om_{X\fr}^2 \to0
\end{equation}
where $\sP$ is induced by (the inverse of) Cartier isomorphism $\Coker(\O_X \xra{d} \Om_{X,\mathit{cl}}^1) \isom \Om_{X\fr}^1$,
$\sQ$ is induced by $d\colon \Om_X^1\to \Om_X^2$ and $\sC$ is the Cartier operation. It is immediate from the definition that
\begin{equation} \label{delta = PC on cl}
\delta(\omega) = \sP(\sC(\omega))  \quad\text{for $\omega\in\Om_{X,\mathit{cl}}^1$.}
\end{equation}
If we define $\kappa\colon \Om_X^1\to \eF_X$ by setting
\begin{equation}\label{kappa def}
    \kappa(\omega)=\sP(\omega\fr)-\delta(\omega)
\end{equation}
then we will have an exact sequence
\begin{equation} \label{kappa_exs}
0\to (\O_X^\x)^p \to \O_X^\x \xra{d\log} \Om_X^1 \xra{\kappa} \eF_X \to0\,.
\end{equation}
(Unlike \eqref{F_X} and \eqref{PQ exseq}, \eqref{kappa_exs} is not $\k$-linear.)

Now, if $\Lna$ is a line bundle with connection, we define its \emph{extended curvature} $\tcurv\Lna\in\Gamma(X,\eF_X)$ to be locally given by
\begin{equation*}
\tcurv\Lna = \kappa(\omega)
\end{equation*}
if $\Lna \isom (\O,d+\omega)$. It is clear from~\eqref{kappa_exs} that this is independent of the trivialization, and that $\sQ(\tcurv\Lna) = F_\na$~--- the usual curvature of $\na$.

We also see from~\eqref{kappa_exs} that line bundles with connection having given extended curvature $\a\in\eF_X$ correspond to splittings of a $\Gm$-gerbe $\dd_\a = \dd_{X,\a}$ on~$X\fr$.
For $\om'\in\Om^1(X\fr)$ we define $\dd_{\om'}=\dd_{\sP(\om')}$ so that splittings of
$\dd_{\om'}$ correspond to line bundles with \emph{flat} connections of $p$-curvature $\om'$.  Note that \eqref{kappa def} implies that $\dd_{\om\fr}\isom\dd_{\sP(\om)}$ for $\om\in\Om^1(X)$.

For later reference, note also that all the constructions discussed above are compatible with pullbacks. Namely, if $f\colon X\to Y$ is a map of smooth schemes then we have a canonical morphism $ f\fr^*\eF_Y\to\eF_X$ of coherent sheaves on ~$X\fr$ induced by the corresponding map on $1$-forms. Given a section $\a\in\eF_Y$, we denote by $f^*\a$ its image under this map. Then, for a connection $\na$ on a line bundle ~$\L$ on ~$Y$, we have $\tcurv (f^*\L,f^*\na)=f^*\tcurv\Lna$. We  also have an identification of $\Gm$-gerbes $f^*\dd_\a \isoto \dd_{f^*\a}$
and for a splitting of $\dd_\a$ given by $\Lna $ its pullback to $f^*\dd_\a$ corresponds to $(f^*\L,f^*\na)$.

One can also summarize the above as follows: consider the category $\Smsch$ of smooth schemes over ~$\k$ with morphisms being arbitrary morphisms, equipped with the Grothendieck topology, where coverings are smooth faithful morphisms.  Then there is a sheaf (stack) $B\Gmna$ of groupoids on $\Smsch$  whose value on a scheme $S$ is the groupoid of line bundles with connection fitting into a fiber sequence
\[
    B\Gm\xra{\Fr^*}
    B\Gmna\xra{\tcurv}
    \eF.
\]
The first map in this sequence is pullback by Frobenius, the second one is extended curvature, and the ``connecting homomorphism'' (of sheaves of $2$-groupoids) $B\Gmna\to\Fr_X^*B^2\Gm$ is given by $\a\mapsto\dd_\a$.
(Note that it is more natural to consider the leftmost term $B\Gm$ in this sequence as sending a test scheme $U$ to the groupoid of line bundles on $U\fr$ rather than on $U$: in the case of \emph{relative} connections on $X$ over~$S$ for a smooth morphism $X\to S$ the relevant sheaf would be sending $U$ to the groupoid of line bundles on the relative Frobenius twist $U\fr[S]=U\x_{S,\Frob_S}S$.)

\begin{prop}
Suppose given a line bundle~$\L$, a connection~$\na$ on it, and $c \in \k\setm\Fp$.  Then splittings of the algebra $\D_{\Lc,\na\fr}$ defined in~\ref{subs:c.red} correspond canonically to line bundles on $X$ with connection $(\L',\na')$ such that
\begin{equation}\label{tcurv = c tcurv}
   \tcurv(\L',\na') = c\cdot\tcurv\Lna\,.
\end{equation}
In other words, we have an equivalence of gerbes $\dd_{\Lc,\na\fr} \isoto \dd_{c\cdot\tcurv\Lna}$ where $\dd_{\Lc,\na\fr}$ denotes the gerbe of splittings of the Azumaya algebra $\D_{\Lc,\na\fr}$.
\end{prop}

\begin{proof}
  As we saw in~\ref{TDO}, the connection~$\na$ on~$\L$ gives an identification of the category of $\D_\Lc$-modules~$\cM$ with the category of quasi-coherent sheaves with connection $(\cF,\na')$ such that $\na'$ is projectively flat of curvature $c F_\na$.  Moreover, $\cF$ is a line bundle if and only if $\cM$ is locally free of rank~1 over $\O_X$, which is equivalent to $\Fr_{X*}\cM$ being locally free of rank $p^{\dim X}$ over $\O_{X\fr}$.  Now denote by~$\tilde\cM$ the $\tilde\D_\Lc$-module corresponding to~$\cM$, and let $\pi\colon (\tilde T^*_\L X)\fr \to X\fr$ be the projection.  Since $\Fr_{X*}\cM = \pi_*\tilde\cM$ and $\tilde\D_\Lc$ is an Azumaya algebra of rank~$p^{2\dim X}$, the above condition means that $\tilde\cM$ is a splitting of~$\tilde\D_\Lc$ on a section of the map~$\pi$.  So all that remains to check is that this section corresponds to~$\na\fr$ if and only if \eqref{tcurv = c tcurv} is satisfied with $\L' = \cF$.

  Since this question is local, we can assume that $\L=\L'=\O$.  The trivialization of~$\L$ gives rise to an equivalence $\D_{X,\Lc}\hmod \isoto \D_X\hmod$.  Under this equivalence, the $\D_\Lc$-module $\cM$ corresponds to the following line bundle with connection:
  \[
      (\L',\na'-c\a) = (\O,d+\b-c\a),
  \]
  where $\a,\b$ are the forms of the connections $\na,\na'$ in the chosen trivializations: $(\L,\na) = (\O,d+\a)$, $(\L',\na') = (\O,d+\b)$.

  Now, if we identify $\tilde T^*_\L X$ with $T^*X$ using our trivialization of~$\L$ then using the diagram
  \[
      \xymatrix@C=4.5em{
        \Spec_{X\fr} Z(\D_\Lc) \ar[r]^-\sim \ar[d]^\wr
                & \tilde T^*_{(\L\fr)^{c^p-c}}X\fr \ar[r]^-\sim_-{(c^p-c)^{-1}} \ar[d]^\wr
                & \tilde T^*_{\L\fr}X\fr \ar[d]^\wr
                \\
        \Spec_{X\fr} Z(\D_X) \ar[r]^-\sim & T^*X\fr \ar[r]^-\sim_{(c^p-c)^{-1}} & T^*X\fr
      }
  \]
  we see that the connection on~$\L\fr$ corresponding to the support of~$\tilde\cM$ is given by $d + (c^p-c)^{-1} ((\b-c\a)\fr - \sC(\b-c\a))$.  (Note that since $d\b = F_{\na'} = cF_\na = c\,d\a$, the form $\b-c\a$ is closed, so it makes sense to apply~$\sC$ to it.)  So our condition now takes the following form:
  \[
      (c^p-c)^{-1} ((\b-c\a)\fr - \sC(\b-c\a)) = -\a\fr,
  \]
  which can  be rewritten (after multiplication by $c^p-c$ and canceling $c^p\a\fr = (c\a)\fr$) as
  \[
      \b\fr - c \a\fr = \sC(\b - c\a).
  \]

  Now, applying $\sP$ (which is injective) to both sides, and using \eqref{delta = PC on cl}~and~\eqref{kappa def}, we see that the above equation is equivalent to
  \[
      \kappa(\b) = c\kappa(\a)
  \]
  which is the same as~\eqref{tcurv = c tcurv} by definition of the extended curvature.
\end{proof}

\begin{rem}
  It is clear (by descent) that this proposition extends to smooth Artin stacks, in particular, it can be applied to $X = (\tilde T^*_{\Ldet}\Bun)\sm$.
\end{rem}

Combining this proposition with Proposition~\ref{prop:Morita eqv stk}, we arrive at the following

\begin{cor}\label{cor DLc eqv GcaL}
  Suppose $\rY,\L,c$ and $(\L',\na_{\on{can}})$ are as in Proposition~\ref{prop:Morita eqv stk} and define $\a_\L := \tcurv(\L',\na_{\on{can}}) \in\eF_ {(\tilde T^*_\L\rY)\sm}$.  Then we have canonically $$\dd_{\rY,\Lc} |_{((\tilde T^*_\L\rY)\fr)\sm} \isoto \dd_{\tilde T^*_\L\rY, c\a_\L}.$$
\end{cor}

We will also need an untwisted version of the above corollary which is proved in a similar way:
\begin{cor}[cf.~\cite{BB,CZ-GL}]\label{cor untw}
  For  a good stack $\rY$ there is a gerbe $\dd_\rY$ on $T^*\rY\fr$ such that $$\D\hmod(\rY)\isom \dd_\rY\hmod\quad\text{and}\quad
  \dd_\rY|_{(T^*\yy)\sm}\isom
  \dd_{\th_\yy\fr}:=
  \dd_{(T^*\yy)\sm,\sP(\th_\yy\fr)} \isom \dd_{(T^*\yy)\sm,\delta(\th_\yy)}
  $$for the canonical $1$-form $\th_\yy$  on
  $(T^*\yy)\sm$.
\end{cor}

Note that the last equivalence follows from the fact that $\sP(\th\fr_\yy)-\delta(\th_\yy)=\kappa(\th_\yy)$ is an
extended curvature of a connection (namely, $(\O_{T^*\yy},d+\th_\yy)$).

\section{Duality and Fourier--Mukai transform for gerbes on torsors over commutative group stacks}
\label{sec:FM}

In this section, we collect some results on duality for commutative group stacks, with some additional information in the case of Picard stacks of (families of) smooth complete curves.  Many results in this section are covered in \cite{CZ-GL} in greater detail.

\subsection{Torsors, gerbes and extensions}\label{subs:trsrs grbs exts}
We fix some base scheme~$S$.  Let $\rX$ be a commutative%
\footnote{In this paper, all commutative group stacks will be assumed ``strongly commutative'' in the sense that the 2-automorphism of the composition $\rX \xra\Delta \rX\x\rX \xra{m} \rX$ induced by the commutativity constraint is identity.%   In other words, for any $S'\to S$, the groupoid $\Hom_S(S',\rX$
}
 group stack\footnote{we will assume for safety that all stacks in this section are Artin stacks, and actually below we restrict to commutative group stacks with a specific property called ``nice,'' see Definition~\ref{def:nice}} over~$S$.  A {\em torsor} over~$\rX$ is a stack~$\rT$ over $S$ equipped with a morphism $a\colon \rX\x_S\rT\to\rT$ and an (2\nbhp-)isomorphism of compositions $a\circ(\id_\rX\x a) \isoto a\circ(m\x\id_\rT)\colon \rX\x_S\rX\x_S\rT \to \rT$ (where $m\colon \rX\x_S\rX\to\rX$ is the group structure on~$\rX$); these data should satisfy a cocycle condition for maps $\rX\x_S\rX\x_S\rX\x_S\rT \to \rT$ and the map $\rT\to S$ should admit an fppf-local section.  This definition can be generalized to the case where $S$ is replaced by an algebraic stack.

In particular, if $\rX =B\Gm =B(\Gm)_S$ is the classifying stack of the multiplicative group, the $\rX$-torsors are otherwise known as {\em $\Gm$-gerbes}.  An example of a $\Gm$-gerbe is provided by the moduli stack $\g_\cA$ of (local) splittings of an Azumaya algebra~$\cA$ on~$S$, and essentially all $\Gm$-gerbes appearing in this paper are of this form.

If $\g$ is a $\Gm$-gerbe on~$\rX$ then every quasi-coherent sheaf on~$\g$ carries a canonical action of $\Gm$ on it (coming from the inertia stack of $\g$), and therefore the category $\QCoh(\g)$ of quasi-coherent sheaves on~$\g$ decomposes into a product over characters of~$\Gm$ (which are numbered by~$\ZZ$):
\[
    \QCoh (\g)   =   \prod_{n \in \ZZ}   \QCoh (\g)_n .
\]
Moreover, $\QCoh(\g)_0$ is canonically equivalent to $\QCoh(\rX)$; all the other $\QCoh(\g)_n$'s are ``twisted versions'' of the category $\QCoh(\rX)$.  One can define the $n$'th power (iterated Baer sum) $\g^n$ of~$\g$; then we have a canonical equivalence $\QCoh(\g)_n \isoto \QCoh(\g^n)_1$.  We will sometimes use the notation $\g\hmod := \QCoh(\g)_1$, motivated by the case of an Azumaya algebra mentioned above, where $\g_\cA \hmod$ is equivalent to the category $\cA \hmod$ of coherent sheaves of $\cA$-modules.  With this notation, the above decomposition becomes
\begin{equation}\label{eq:QCoh of a gerbe}
    \QCoh (\g)   =   \prod_{n \in \ZZ}   \g^n\hmod .
\end{equation}

Now let $\rX,\rY$ be two commutative group stacks over~$S$.  Then $\rX_\rY := \rX\x_S\rY \to \rY$ is a commutative group stack over~$\rY$, so it makes sense to speak about $\rX_\rY$-torsors.  Let $\rZ$ be such a torsor.  We want to explain what structure one needs to specify on~$\rZ$ in order for $\rZ$ to become a commutative group stack over~$S$.  Let $\pr_1,\pr_2,m_\rY\colon \rY\x_S\rY \to\rY$ be the projections and the multiplication map, respectively.  Then our structure is an isomorphism of $\rX_\rY$-torsors $m_\rY^*\rZ \isoto \pr_1^*\rZ \x^{\rX_\rY} \pr_2^*\rZ$ where $\x^{\rX_\rY}$ is the Baer sum of torsors, as well as an equivariant structure on $m$ under the $\gS_2$-action on both sides, lifting the permutation action on $\rY\x_S\rY$.  In addition to that, there should be compatibility isomorphisms  on the triple product, which are subject to cocycle conditions on the quadruple product.  We will not write out them all, because they won't be used explicitly, but they can be recovered from the definition of strictly symmetric monoidal category.

This structure will be referred to as a {\em multiplicative structure} on the $\rX_\rY$-torsor $\rZ$.  For $\rX=B\Gm$ we will say that $\rZ$ is a {\em multiplicative $\Gm$-gerbe} on~$\rY$.  From this data one can define a commutative group structure on~$\rZ$ and construct a Cartesian square of commutative group stacks:
\[
    \cxym{\rX\ar[r]\ar[d]&  \rZ\ar[d]\\
          \pt_S=S\ar[r]&        \rY}
\]
with the right vertical map being locally essentially surjective in fppf topology (as a morphism of functors of points).  In this situation we will call $\rZ$ an {\em extension} of~$\rY$ by~$\rX$.

We list some properties of extensions of commutative group stacks:
\begin{enumerate}
  \item\label{cgs 3-step ext} If $\rY_1$ is an extension of $\rX_2$ by $\rX_1$, and $\rZ$ is an extension of $\rX_3$ by~$\rY_1$, then one can canonically construct another pair of extensions $\rX_2\to\rY_2\to\rX_3$ and $\rX_1\to\rZ\to\rY_2$.
  \item\label{cgs ext by ZZ} Let $\ZZ_S=\coprod_{n\in\ZZ}S$ be the ``constant group scheme'' with fiber~$\ZZ$ and $\rX\to\wtld\rX\to\ZZ_S$ an extension of commutative group stacks.  Then $\rT:= \wtld\rX\x_{\ZZ_S}\{1\}_S$ is naturally an $\rX$-torsor, and the correspondence $\wtld\rX \mapsto \rT$ gives an equivalence of 2-groupoids between extensions of $\ZZ_S$ by~$\rX$ and $\rX$-torsors.  (Note that it is important here that we use \emph{strictly}  commutative group stacks.)
  \item\label{cgs ext G and BG} If $G$  is a commutative group scheme and $\rX$ is a commutative group stack then there is an equivalence of categories between extensions of~$\rX$ by~$G$ and morphisms of commutative group stacks from~$\rX$ to the classifying stack~$BG$.
\end{enumerate}

Let us apply point~\ref{cgs 3-step ext} to $\rX_1=B\Gm$, $\rX_3=\ZZ_S$, and $\rX_2=\rX$ being any commutative group stack.  From point~\ref{cgs ext by ZZ} we see that the choice of $\rY_2$ is equivalent to the choice of a torsor $\rT$ over $\rX$; similarly, $\rY_1\to\rZ\to\ZZ$ gives a torsor $\wtld\rT$ over $\rY_1$.  On the other hand, $\rY_1$ and $\wtld\rT$ are $\Gm$-gerbes on $\rX$ and $\rT$ respectively, so let us denote them by $\gx=\rY_1$, $\g_\rT=\wtld\rT$.  The action of $\gx$ on~$\g_\rT$ can be described similarly to the way we described the structure of multiplicative gerbe: if $m$~and~$a$ are the multiplication map on~$\rX$ and the action of~$\rX$ on~$\rT$ then one should have an isomorphism $a^*\g_\rT\isoto \pr_1^*\gx\cdot\pr_2^*\g_\rT$ together with the usual compatibilities with the isomorphism $m^*\gx\isoto \pr_1^*\gx\cdot\pr_2^*\gx$ on higher products. (Here ${\cdot}={\x^{B\Gm}}$ is the Baer sum of $\Gm$-gerbes.)

\subsection{Duality and Fourier--Mukai transform}
Given $\rX$ a commutative group Artin stack, let $\rX\chk=\cHom_{\rm cgs}(\rX,B\Gm)$ (``cgs'' stands for ``commutative group stacks'') be the group stack whose $S'$-points for any $S'\to S$ correspond to morphisms $\rX_{S'} := \rX\x_SS'\to B(\Gm)_{S'}$ of commutative group stacks over~$S'$.  By point~\ref{cgs ext G and BG} above, they also correspond to extensions of $\rX_{S'}$ by $(\Gm)_{S'}$.  Then $\xv$ has a structure of a commutative group stack, although it might not be an Artin stack, merely a stack in the fppf-topology.

We have a canonical universal morphism $\rX\x\rX\chk\to B\Gm$ which gives a line bundle~$\cP$ on~$\rX\x\rX\chk$ called {\em Poincar\'e bundle}.  From this morphism, we also get a map $r\colon \rX\to\rX^{\vee\vee}$.
\begin{defn}[cf.~\cite{Ar-app}]\label{def:nice}
  We say that a commutative group Artin stack $\rX$ is {\em nice} if it satisfies the following four  conditions:
  \begin{enumerate}
    \item The stack $\xv$ is an Artin stack.
    \item The map $r\colon \rX\to\rX^{\vee\vee}$ is an isomorphism of group stacks.
    \item There exists a left adjoint $\pr_{2!}$ to $\pr_2^*\colon D^b(\rZ)\to D^b(\rX\x_S\rZ)  $ for $\pr_2\colon \rX\x_S\rZ\to\rZ$ for any Artin stack  $\rZ$ over $S$  %        \kkk[is there a criterion for this?]
        where $D^b(\rY)$ denotes the bounded derived category of quasi-coherent sheaves on a stack~$\rY$. Thus we have a {\em Fourier--Mukai functor} $\Phi_\rX\colon D^b(\rX)\to D^b(\rX\chk)$ by the  formula
        \begin{equation}\label{eq:F-M formula}
          \Phi_\rX(\cF)=\pr_{2!}(\pr_1^*\cF\tsr\cP_\rY)
        \end{equation}
        for $\pr_1\colon \rX\x_S\rX\chk\to\rX$, $\pr_2\colon \rX\x_S\rX\chk\to\rX\chk$ the natural projections.

    \item The functor $\Phi_\rX\colon D^b(\rX)\to D^b(\rX\chk)$ is an equivalence of categories.
  \end{enumerate}
%  In particular, $\rX$ is nice if it locally over $S$ in fppf-topology looks like a product of
\end{defn}

As explained in~\cite{BB}, examples of nice group stacks include $\ZZ_S$, $B(\Gm)_S$ and abelian schemes over~$S$, as well as fiber products over~$S$ of such.  The duality operation $\chk$ preserves products over~$S$, interchanges $\ZZ$ and $B\Gm$, and takes an abelian scheme to the dual one in the classical sense.  Also, the property of being nice is fppf-local on~$S$. More generally, it  includes Beilinson's 1-motives defined in \cite[Appendix~A]{CZ-GL}.

In particular, if $T\to S$ is a family of smooth complete curves (\ie smooth projective morphism of relative dimension~$1$), then the Picard stack $\Pic(T/S)$ is a nice group stack. It is well-known that the Picard stacks of curves are self-dual, \ie $\Pic(T/S)\chk=\Pic(T/S)$.  This is the only example to be used in the rest of the paper.   The Poincar\'e bundle is constructed as follows: given an $S'$-point of $\Pic(T/S)\x_S\Pic(T/S)$, let $\L,\L'$ be the corresponding pair of line bundles on $T\x_SS'$ and $\pi\colon T\x_SS'\to S'$ the projection.  Then the pullback of the Poincar\'e bundle to~$S'$ is given by $\det R\pi_*\O_{T\x_SS'} \tsr(\det R\pi_*\L)^{-1} \tsr(\det R\pi_*\L')^{-1} \tsr\det R\pi_*(\L\tsr\L')$.

\begin{rem}
  Note that in the example $\rX=\ZZ_S$ we cannot replace $\pr_{2!}$ by $\pr_{2*}$ in the definition of $\Phi_\rX$ because the latter does not preserve quasi-coherence (for every smooth $S'\to \rZ$ and $\cF\in D^b(\rX\xx_S\rZ)$, the object $(\pr_{2*}\cF)_{S'}$ is the \emph{product}, not the direct  sum, of pushforwards from connected  components  of $\rX\x_S\rZ=\ZZ\x\rZ$ (cf.\ the  correction to the Fourier--Mukai functor  in \cite{Gr}). Probably it is  possible  to generalize the notion  of commutative group stack so as to include examples  like $B\GG_a$ and its  dual (which is the divided power neighborhood in $\GG_a$---not an Artin stack).
\end{rem}

Now let us discuss the twisted Fourier--Mukai transform for torsors and $\Gm$-gerbes.  We begin with the following observation.  Let $\rX\to\rZ\to\rY$ be an extension of commutative group stacks over~$S$, as in~\ref{subs:trsrs grbs exts}, and assume that it splits fppf-locally over~$S$.  Then $\rY\chk\to\rZ\chk\to\rX\chk$ is also a locally split extension.  Also we see that if $\rX$~and~$\rY$ are nice, then so is~$\rZ$.\footnote{In~\cite[Lemma~2.5]{BB} it is claimed that this is true for any extension of nice stacks.  But the author couldn't prove this because of the problem showing that $\rZ\chk\to\rX\chk$ is surjective.}

Now let $\rX$ be a nice group stack, and $\rX\to\wtld\rX\to\ZZ_S$ an extension as in~\ref{subs:trsrs grbs exts}, point~\ref{cgs ext by ZZ}.  Then we have the corresponding $\rX$-torsor $\rT$.  Since $\ZZ_S\chk = B(\Gm)_S$, we see from the previous paragraph that there is a dual extension (note that any extension of~$\ZZ_S$ is automatically locally split): $B\Gm \to\g_{\rX\chk}:=\wtld\rX\chk \to\rX\chk$.  The group stack $\g_{\rX\chk}$ is nice and it is a multiplicative $\Gm$-gerbe on~$\rX\chk$.  The action of $B\Gm$ on $\g_{\rX\chk}$, as on any $\Gm$-gerbe, gives an action of~$\Gm$ on any object of $D^b(\g_{\rX\chk})$, which gives a decomposition of the latter category into a sum over characters of~$\Gm$ (see~\eqref{eq:QCoh of a gerbe}):
\[
    D^b(\g_{\rX\chk})  =\prodpr_{n\in\ZZ}  D^b(\g_{\rX\chk})_n,
\]
where the prime indicates the subcategory of the product category consisting of objects that are bounded with respect to the product of the standard t-structures on the factors.
Under the Fourier--Mukai equivalence $\Phi_{\wtld\rX}\colon D^b(\wtld\rX)\isoto D^b(\g_{\rX\chk})$ this decomposition corresponds to that of $D^b(\wtld\rX)$ coming from the decomposition $\wtld\rX=\coprod_{n\in\ZZ}\rT_n$ where $\rT_n=\wtld\rX\x_{\ZZ_S}\{n\}_S$. Thus, for any~$n$, we get an equivalence $D^b(\rT_n) \isoto D^b(\g_{\rX\chk})_n$.  In particular, for $n=1$, we have $\rT_1=\rT$, so our equivalence takes the form:
\[
    D^b(\rT)  \isoto  D^b(\g_{\rX\chk})_1.
\]

Now let $\g_{\wtld\rX}$ be a multiplicative $\Gm$-gerbe on~$\wtld\rX$, \ie an extension of $\wtld\rX$ by~$B\Gm$.  Then on~$\g_{\wtld\rX}$ we have a 3-step filtration (in the sense of point~\ref{cgs 3-step ext}) with ``quotients'' $B\Gm$, $\rX$ and $\ZZ$.  Since $\rX\to\wtld\rX\to\ZZ_S$ locally splits, we have that $B\Gm\to\g_{\wtld\rX}\to\wtld\rX$ is locally split if and only if $B\Gm\to\gx\to\rX$ is.  (Here $\gx$ is the restriction of $\g_{\wtld\rX}$ to~$\rX$; in other words, $\gx:= \g_{\wtld\rX} \xx_{\ZZ_S} \{1\}_S$.)  Assume that these extensions are locally split; then for the dual group stack $(\g_{\wtld\rX})\chk$ we also have a filtration with quotients $B\Gm$, $\rY:=\rX\chk$ and~$\ZZ_S$.

So we have multiplicative $\Gm$-gerbes $\gy:=(\wtld\rX)\chk$ and $\g_{\wtld\rY}:=(\g_{\wtld\rX})\chk$ on $\rY$ and $\wtld\rY:=(\gx)\chk$, respectively.  Put $\rT_n=\wtld\rX\x_{\ZZ_S}\{n\}_S$, $\rT_n'=\wtld\rX\x_{\ZZ_S}\{n\}_S$, $\g_{\rT_n}=\g_{\wtld\rX}\x_{\ZZ_S}\{n\}_S$, $\g_{\rT_n'}=\g_{\wtld\rY}\x_{\ZZ_S}\{n\}_S$.  Then we get decompositions
\[
    D^b(\g_{\wtld\rX})  =\prodpr_{m,n\in\ZZ} D^b(\g_{\rT_m})_n, \quad
    D^b(\g_{\wtld\rY})  =\prodpr_{m,n\in\ZZ} D^b(\g_{\rT_m'})_n,
\]
and one can check that $\Phi_{\g_{\wtld\rX}}$ interchanges $m$~and~$n$ in these decompositions, more precisely, it maps $D^b(\g_{\rT_m})_{-n}$ to $D^b(\g_{\rT_n'})_m$.  In particular, setting $m=n=1$, we obtain an equivalence $D^b(\g_\rT)_{-1} \isoto D^b(\g_{\rT'})_1$.  So we get the following

\begin{prop}\label{prop:dual gerbe on torsor}
  Fix a nice commutative group stack~$\rX$ and denote by~$\rY=\rX\chk$ its dual group stack.  Let $\tx$ be an $\rX$-torsor, $\gx$ a multiplicative $\Gm$-gerbe on~$\rX$ and $\gtx$ a $\Gm$-gerbe on $\tx$ with an action of $\gx$.  Then there is a canonical ``dual'' data $(\ty,\gy,\gty)$ on~$\rY$ and a splitting~$\cQ$ of $\pr_1^*\gtx\cdot\pr_2^*\gty$ on $\tx\x_S\ty$, such that $\cQ^{-1}$ induces an equivalence $D^b(\gtx)_1 \isoto D^b(\gty^{-1})_1$ given by a formula similar to~\eqref{eq:F-M formula}.
\end{prop}

The following proposition is useful for showing that two given gerbes on torsors are dual:

\begin{prop}\label{prop:trsr-grb duality data}
  Let $\rX$ and $\rY$ be dual commutative group stacks, and let $\tx$, $\gx$ and $\gtx$ be as in Proposition~\ref{prop:dual gerbe on torsor}.  Suppose also that we are given similar data $\ty,\gy,\gty$ for~$\rY$.  Then the data of identification of $(\ty,\gy,\gty)$ with the dual of $(\tx,\gx,\gtx)$ is equivalent to the following data:
  \begin{enumerate}\rncmd\theenumi{D\arabic{enumi}}
    \item\label{idfn trsrs w duals to grbs} an identification of $\gy$ with the dual gerbe to $\tx$ and of $\gx$ with dual of~$\ty$ as multiplicative gerbes;
    \item\label{splitting Q} a splitting $\cQ$ of the gerbe $\pr_1^*\gtx\cdot\pr_2^*\gty$ on $\tx\x_S\ty$;
    \item\label{isoms aX aY} an isomorphism $\a_\rX \colon \pr_{1,3}^*\cP_{\ty}\tsr\pr_{2,3}^*\cQ \isoto \mu^* ((a_{\tx}\x\id_{\ty})^*\cQ)$ where $\mu\colon \pr_1^*\gx\cdot\pr_2^*\gtx\cdot\pr_3^*\gty \isoto (a_{\tx}\circ\pr_{1,2})^*\gtx\cdot\pr_3^*\gty$ is the equivalence of gerbes on $\rX\x\tx\x\ty$ coming from the action of $\gx$ on~$\gtx$  (here $a_{\tx}\colon \rX\x\tx \to\tx$ is the action map, and $\cP_{\ty}$ is the universal splitting of $\gx\x\ty$ (as a gerbe on $\rX\x\ty$) corresponding to the identification in part~\ref{idfn trsrs w duals to grbs}); and also a similar isomorphism~$\a_\rY$ for~$\rY$.  These isomorphisms should satisfy the natural cocycle conditions on $\rX\x\rX\x\tx\x\ty$, $\tx\x\rY\x\rY\x\ty$ and $\rX\x\tx\x\rY\x\ty$.
  \end{enumerate}
\end{prop}

\begin{proof}[Proof sketch]
  We begin by noting that, by Proposition~\ref{prop:dual gerbe on torsor}, we have a canonical triple $(\ty',\gy',\g_{\ty'})$ dual to $(\tx,\gx,\gtx)$.  We want to show that the data of~(\ref{idfn trsrs w duals to grbs})--(\ref{isoms aX aY}) is equivalent to an identification $(\ty,\gy,\gty) \isoto (\ty',\gy',\g_{\ty'})$.  One can see that the data of~(\ref{idfn trsrs w duals to grbs}) amounts to an isomorphism of $\rY$-torsors between $\ty$~and~$\ty'$, and an isomorphism of multiplicative $\Gm$-gerbes on~$\rY$ between $\gy$~and~$\gy'$.

  Denote by $\wtld\rX$ and $\g_{\wtld\rX}$ the extension of~$\ZZ$ by~$\rX$ and the $\Gm$-gerbe on it corresponding to $\tx$~and~$\gtx$, respectively.  Then $\g_{\ty'}$ classifies multiplicative splittings of~$\g_{\wtld\rX}$: in other words, for an $S$-scheme $S'$, the groupoid $\Hom_S(S',\g_{\ty'})$ is canonically equivalent to the groupoid of multiplicative splittings of~$\g_{\wtld\rX}\x_SS'$ considered as a multiplicative $\Gm$-gerbe on the group stack $\wtld\rX\x_SS'$ over~$S'$.

  But one can argue that such multiplicative splittings are determined by their restriction (together with the isomorphisms giving multiplicative structure) to $\wtld\rX\x_\ZZ\{0,1\}= \rX \sqcup \tx$. Therefore a map $\psi\colon \gty \to \g_{\ty'}$ as mere stacks over~$S$ is equivalent to a multiplicative splitting $\tilde\cQ$ of $(\gx \sqcup \gtx)\x \gty$ on $(\rX \sqcup \tx)\x \gty \to \gty$.  If $\psi$ is actually a map of $\Gm$-gerbes over~$\ty=\ty'$ then the restriction of~$\tilde\cQ$ to $\rX \x \gty$ (the first component of the disjoint union) is identified with $\cP_{\ty}$ and the restriction of~$\tilde\rQ$ to $\tx \x \gty$ gives the splitting~$\cQ$ from~(\ref{splitting Q}) together with an isomorphism~$\a_\rX$ as in~(\ref{isoms aX aY}) satisfying the relevant cocycle condition.  Finally, $\a_\rY$ and the other cocycle conditions are responsible for the compatibility of~$\psi$ with the action of $\gy$ on $\gty$~and~$\g_{\ty'}$.
\end{proof}

\begin{rem}\label{rem:isom aX enough}
%  \begin{enumerate}
  Note that, if we are just given the data of~(\ref{idfn trsrs w duals to grbs}) then one can construct the needed identification up to a twist by pullback of a $\Gm$-gerbe on~$S$.  This is because the data of $\gx,\tx,\gtx$ amounts to the data of a  torsor over the (non-commutative)   Heisenberg-type extension of $\rX\x\rX\chk$ by $B\Gm$.  The dual data corresponds then to a torsor over the similar extension for $\rY$ induced from the first one by the identification $\rX\x\rX\chk\isoto\rX\x\rY\isoto\rY\x\rX\isoto\rY\x\rY\chk$.  On the other hand (\ref{idfn trsrs w duals to grbs}) gives just an identification of $\rY\x\rY\chk$-torsors, so $(\gy,\ty,\gty)$ differs from the dual of $(\gx,\tx,\gtx)$ by twisting the third component by a $B\Gm$-torsor on $S$.
\end{rem}

\begin{rem}\label{rem: data for dty up to gerb on base}
  Also, suppose that, in the setting of the above proposition, only the identification of~$\gx$ with dual of $\ty$, the splitting~$\cQ$ and the isomorphism~$\a_\rX$ are given.  On the category $\gty^{-1} \hmod$ we have functors of tensor multiplication by quasi-coherent sheaves on $\ty$.  On  the other hand,  $\gtx\hmod$ is acted on by convolutions with objects in $\gx \hmod$ (thanks to the action of $\gx$ on $\gtx$).  The isomorphism $\a_\rX$ then ensures that the functor $\Phi_\cQ \colon \gtx\hmod \to \gty^{-1}\hmod$ defined by $\cQ^{-1}$ intertwines these actions.  (Note that $\QCoh(\ty)$ is already identified with $\gx\hmod$ by Fourier--Mukai with kernel $\cP_{\ty}$.)  This is already enough to conclude that $\Phi_\cQ$ is an equivalence because, locally over~$S$, each of the two categories if ``freely generated'' by one object (a family of skyscrapers for convolution, and a global splitting for multiplication) and $\Phi_\cQ$ interchanges these generators.
%  \end{enumerate}
\end{rem}

\subsection{The case of Picard stack of a curve}\label{subs:Pic(C) case}
In this subsection, we will consider the case where $\rX$ is the Picard stack of a family~$C$ of complete smooth curves over~$S$ (in other words, $C\to S$ is a proper smooth morphism of relative dimension~$1$, and $\rX=\Pic(C/S)$).  As mentioned above, $\rX$ is nice in this case, and we have canonically $\rX\chk \isom \rX$.  In fact, it is easy to construct a morphism $\rX\chk\to\rX$ by means of the Abel--Jacobi map $\AJ\colon C \to\Pic(C/S)$ sending $x\in\Hom_S(S',C)$ to the the object in $\Hom_S(S',\Pic(C/S))$ corresponding to the line bundle $\O_{C\x_SS'}(\Gamma_x)$.  Namely the map is given by composition
\begin{equation}\label{eq:the AJ* chain}
  \begin{split}
    \rX\chk = \cHom_{{\rm cgs }/S}(\rX,B\Gm) \to \cHom_S(\rX,B\Gm) \xra{\AJ^*}  \\
        \xra{\AJ^*}   \cHom_S(C,B\Gm) = \Pic(C/S) = \rX.
  \end{split}
\end{equation}

The main idea of this subsection is that in the case of $\Pic(C)$ the data in Proposition~\ref{prop:trsr-grb duality data} is determined by its pullback under the map~$\AJ$.  The above construction of the isomorphism $\rX\chk\isoto\rX$ is an example of this phenomenon.  Indeed, the fact that the composition of the maps in~\eqref{eq:the AJ* chain} is an isomorphism means that the maps of commutative group stacks $\Pic(C/S) \to B\Gm$ (\ie ``multiplicative line bundles'' on $\Pic(C/S)$) are uniquely determined by their precomposition with (\ie pullback by) $\AJ$.

Now we want to state and prove similar result for $\Gm$-gerbes on $\Pic(C)$ and $\Pic(C)$-torsors.

Note that by \cite[Lemma~A.3.4]{CZ-GL}, such gerbes are automatically split locally over~$S$ (the proof of this is based on a cohomology  vanishing result of \cite{Br}), so we can apply  the results of previous subsection to them without mentioning this condition.

\begin{prop}\label{prop:mult grb on Pic}\par
\begin{enumerate}
  \item\label{gerbes on pic}  Pullback by $\AJ$ gives an equivalence of $2$-groupoids  \[\{\text{\upshape multiplicative gerbes on  $\Pic(C/S)$}\}
  \ \isoto\
  \{\text{\upshape $\Gm$-gerbes on $C$ split \'etale locallly  over $S$}\}.\]
  \item\label{dual to grb on pic} Given a multiplicative gerbe $\g$ on $\Pic(C/S)$, the dual torsor (\ie the torsor of fiberwise multiplicative splittings of~$\g$) is identified with the  torsor of  \emph{all} (fiberwise) splittings of $\AJ^*\g$.
\end{enumerate}
\end{prop}
\begin{proof}
  Part~(\ref{gerbes on pic}) follows  from the above-mentioned statement about line bundles. Indeed, since multiplicative gerbes on $\Pic(C/S)$  split \'etale locally over ~$S$, we see that the above map of $2$-groupoids  is obtained from the analogous map on Picard (\ie commutative monoidal) $1$-groupoids by taking the prestack on $S$ of classifying $2$-groupoids and \'etale-sheafifying.

  Part~(\ref{dual to grb on pic}) follows from (\ref{gerbes on pic}) by taking the groupoid of $1$-morphisms from the trivial gerbe to $\g$, resp., $\AJ^*\g$ on both sides.
\end{proof}
Below we discuss actions of multiplicative gerbes on $\Pic(C)$ on gerbes on $\Pic(C)$-torsors.
\begin{lem}\label{lem:ext of action from AJ image}
  Let $\gx$ be a multiplicative $\Gm$-gerbe on $\rX = \Pic(C/S)$.  Also let $\rT$ be an $\rX$-torsor and $\g_\rT$ a $\Gm$-gerbe on~$\rT$ with an action of~$\gx$.  Suppose given a multiplicative splitting $\cS_\rX$ of~$\gx$ and a splitting $\cS_\rT$ of~$\g_\rT$.  Put $\g_C=\AJ^*\gx$, $\cS_C=\AJ^*\cS_\rX$  Denote by $m\colon \rX\x\rX\to\rX$ and $a\colon \rX\x\rT \to \rT$ the group structure on~$\rX$ and the action of $\rX$ on~$\rT$, and by $m_{\AJ},a_{\AJ}$ their precomposition with $\AJ\x\id_\rX \colon C\x\rX \to \rX\x\rX$ and $\AJ\x\id_\rT \colon C\x\rT \to \rX\x\rT$, respectively.  Finally, assume that we have an isomorphism $\a_{\AJ}\colon a_{\AJ}^*\cS_\rT\isoto \pr_1^*\cS_C \cdot \pr_2^*\cS_\rT$ of two splittings of $\pr_1^*\g_C\cdot\pr_2^*\g_\rT \isoto a_{\AJ}^*\g_\rT$ such that the composite isomorphism $a_{2,\AJ}^*\cS_\rT \xra{(\id_C\x a_{\AJ})^*\a_{\AJ}} \cS_C\bx a_{\AJ}^*\cS_\rT \xra{\id_{\cS_C}\bx \a_{\AJ}} \cS_C\bx\cS_C\bx\cS_\rT$ (where $a_{2,\AJ} \colon C\x C\x\rT\to\rT$ is the restriction of the ``triple addition'' map $\rX\x\rX\x\rT\to\rT$) is equivariant under switching the first two factors.

  Then $\a_{\AJ}$ can be extended uniquely to an isomorphism $\a\colon m^*\cS_\rT \isoto \cS_\rX \bx \cS_\rT$ satisfying the cocycle condition on $\rX\x \rX\x \rT$ so that it gives an ``action'' of $\cS_\rX$ on~$\cS_\rT$.
\end{lem}

\begin{proof}
  We begin by noting that we can use $\cS_\rT$ to identify $\gx$ with the trivial gerbe (as multiplicative gerbes) so that $\cS_\rY$ becomes the trivial splitting.  Also, since our statement is local over~$S$, we may assume that $\rT$ is the trivial torsor: $\rT=\rX= \Pic(C/S)$, $\g_\rT$ is the trivial gerbe on it and the action of~$\gx$ on~$\g_\rT$ is compatible with the chosen splittings.  Having made these identifications, we can think of~$\cS_\rT$ as a line bundle on~$\rX$ which we will denote by~$\L$, and of~$\a_{\AJ}$ as an isomorphism $m_{\AJ}^*\L  \isoto  \O_C \bx \L$.  Our task then becomes to extend $\a_{\AJ}$ to an isomorphism $\a\colon m^*\L \isoto \O_\rX \bx \L$ subject to the usual cocycle condition (note that this is the same as decent data for~$\L$ along $\rX\to S$).

  Also, without loss of generality, we can assume that the base~$S$ is connected and the family $C \to S$ has (geometrically) connected fibers.  Then it makes sense to speak of the genus of~$C$, which we will denote by $g(C)$.  For any integer $d\ge1$, define $m_{d.\AJ}$ to be the composition
  \[
      \underbrace{C\x_S\dots\x_SC}_d \x_S \rX \xra{\AJ\x\dots\x\AJ\x\id_\rX} \underbrace{\rX\x_S\dots\x_S\rX \x_S \rX}_{d+1} \to \rX,
  \]
  where the last map is the $(d+1)$-fold addition map and let
  \[
      \pi_d\colon \underbrace{C\x_S\dots\x_SC}_d \x_S \rX  \to \rX
  \]
  be the projection to the last factor.  Using the isomorphism $\a_{\AJ}$, we can construct (by induction on~$d$) an isomorphism $\a_{d,\AJ} \colon m_{d,\AJ}^*\L \isoto \pi_d^*\L$.  The symmetry condition on~$\a_{2,\AJ}$ imposed in the statement of the lemma implies that $\a_{d,\AJ}$ is also equivariant with respect to the symmetric group~$\gS_d$ permuting the $d$ copies of~$C$.  It follows that $\a_{d,\AJ}$ can be descended to $C^{\angs{d}}\xx_S \rX$, where $C^{\angs{d}}$ is the $d$'th symmetric power of~$C$.

  Now let $\rX^{[d]}$ be the part of $\rX=\Pic(C/S)$ parametrizing line bundles of degree~$d$, so that $\rX = \coprod_{d\in\ZZ} \rX^{[d]}$.  We have the natural map $\AJ_d \colon C^{\angs{d}} \to \rX^{[d]}$.  We claim that, for $d$ large enough, any line bundle on $C^{\angs{d}}\xx_S \rX$ canonically descends to $\rX^{[d]}\xx_S \rX$.  To see this, recall that the stack~$\rX^{[d]}$ is a $\Gm$-gerbe over certain abelian scheme over~$S$ (the ``$d$'th Jacobian'' of~$C$), which we will denote by~$J^d$.  By a classical result, for $d\ge 2g(C)-1$ the composed morphism $C^{\angs{d}} \xra{\AJ_d} \rX^{[d]} \to J^d$ is a smooth fibration with fibers isomorphic to $\PP^{d-g(C)}$.  The map $\AJ_d$ is isomorphic locally over~$J^d$ to the morphism $\PP^{d-g(C)}_U \to (B\Gm)_U$ corresponding to the line bundle $\O(1)$ on $\PP^{d-g(C)}_U$ (where $U$ is a neighborhood in~$J^d$).  Hence our claim about descent to~$\rX^{[d]}$ (for $d\ge d_0 := \max \{g(C)+1,2g(C)-1\}$) follows from the following statement (whose proof is left to the reader):

  \begin{lem}
    Let $T$ be a scheme and $n>0$ an integer.  Then pullback by the map $\PP^n_T \to (B\Gm)_T$ corresponding to the line bundle $\O_{\PP^n_T}(1)$ induces an equivalence between groupoids of line bundles on $(B\Gm)_T$ and $\PP^n_T$.
  \end{lem}

  Thus, for $d\ge d_0$, the isomorphism $\a_{d,\AJ}$ uniquely descends to $\rX^{[d]}$, so we have constructed the desired isomorphism~$\a$ on $\rX^{\ge d_0} \xx_S\rX \ctd \rX\xx_S\rX$ where $\rX^{\ge d_0} := \coprod_{d\ge d_0} \rX^{[d]} \ctd \rX$; denote this isomorphism by~$\tilde\a$.  The cocycle condition for~$\tilde\a$ follows from the construction.  Now, if $x,x'$ are two $S'$-points of~$\rX$ then $\tilde\a$ gives an identification $\a_{x,x'-x} \colon x^*\L \isoto x'^*\L$ whenever the difference $x'-x$ (with respect to the group structure on~$\rX$) lands in $\rX^{\ge d_0}$, and the cocycle condition guarantees that $\a_{x',x''-x'} \circ \a_{x,x'-x} = \a_{x,x''-x}$ whenever all three isomorphisms are defined.  But this is clearly enough to construct $\a_{x,x'-x}$ for any $x,x'$, which gives our isomorphism $\a\colon m^*\L \to \O_\rX \bx \L$.
\end{proof}

\begin{rem}\label{rem:S2-symmetry automatic}
  Suppose that, in the statement of Lemma~\ref{lem:ext of action from AJ image}, the family $C\to S$ has geometrically connected fibers.  Then one can show that the $\gS_2$-equivariance condition on $\a_{2,\AJ}$ is satisfied automatically by properness considerations.  Indeed, it always holds on the diagonal $C\x\rT \xra{\Delta_C\x\id_\rT} C\x C \x\rT$.  The obstruction to $\gS_2$-equivariance of $\a_{2,\AJ}$ is an invertible function on $C\x C\x \rT$.  But any function on $C\x C\x\rT$ descends to~$\rT$, so if it equals~$1$ on the diagonal, it has to equal~$1$ everywhere, as desired.
\end{rem}

\begin{cor}\label{cor:trsr-grb duality data on Im AJ}
  Suppose we are in the setup of Proposition~\ref{prop:trsr-grb duality data} with $\rX=\Pic(C/S)$, and the data {\upshape(\ref{idfn trsrs w duals to grbs})--(\ref{splitting Q})} is given.  Suppose that we have an isomorphism  $\a_{\rX,\AJ}\colon (\AJ\x \id_{\tx\x\ty})^* \pr_{1,3}^*\cP_{\ty}\tsr\pr_{2,3}^*\cQ \isoto (\AJ\x \id_{\tx\x\ty})^* (\mu^* ((a_{\tx}\x\id_{\ty})^*\cQ))$.  Then there is a unique isomorphism $\a_\rX \colon \pr_{1,3}^*\cP_{\ty}\tsr\pr_{2,3}^*\cQ \isoto \mu^* ((a_{\tx}\x\id_{\ty})^*\cQ)$ satisfying the cocycle condition on $\rX \x \rX \x \tx \x \ty$ of Proposition~\ref{prop:trsr-grb duality data}, part~(\ref{isoms aX aY}), such that $\a_\rX = (\AJ\x\id_{\tx\x\ty})^* \a_{\rX,\AJ}$.

  Moreover, if we have similar isomorphism~$\a_\rY$ on $\tx \x \rX \x \ty$ then the compatibility on $\rX \x \tx \x \rY \x \ty$ can be checked on $C \x \tx \x C \x \ty$ for $\a_{\rX,\AJ}$ and $\a_{\rY,\AJ}$.
\end{cor}

\begin{proof}[Proof sketch]
  The first statement in the corollary follows from Lemma~\ref{lem:ext of action from AJ image} with $S$, $C$, $\rX$, $\rT$, $\gx$, $\g_\rT$, $\cS_\rX$, $\cS_\rT$ replaced by $\ty$, $C\x\ty$, $\rX\x\ty$, $\tx\x\ty$, $\gx\x\ty$, $\gtx\x\ty$, $\cP_{\ty}$, $\cQ$.  (Although Lemma~\ref{lem:ext of action from AJ image} requires $S$ to be a scheme rather than a stack, the stack case can be deduced from the scheme case by descent.)  The second part also not hard to prove using the fact that $\Pic(C)$ is generated by the image of the Abel--Jacobi map.  The details are left to the reader.
\end{proof}

\section{\texorpdfstring{$\D$}{D}-modules on \texorpdfstring{$\Bun_N$}{Bun\_N} and the Hitchin fibration}
\label{sec:hitchin}

In this section we will recall some results and constructions from~\cite{BB} about the geometry of Hitchin fibration (and its twisted version in characteristic~$p$) and its application to $\D$-modules on~$\Bun$. % and extend them to twisted $\D$-modules.

\subsection{Hitchin fibration}\label{subs:hitchin fib}
Let $C$ be a smooth connected projective curve over~$\k$ of genus greater than~$1$, and fix an integer $N>1$.  Denote by $\Bun_N$ the moduli stack of rank~$N$~bundles on~$C$.  By definition, for a scheme~$S$, the groupoid of maps $S \to \Bun_N$ is equivalent to the groupoid of rank~$N$~vector bundles on $C\x S$.  It is classical that the cotangent bundle to $\Bun_N$ is identified with the stack $\Higgs$ of Higgs bundles.  Recall that for a vector bundle~$\cE$ on~$C$, a structure of \emph{Higgs bundle} on it, also  known as a \emph{Higgs field}, is  an $\O$-linear map $\cE\to\cE\tsr\omega_C = \cE\tsr\Om^1_C$.   Denote by~$\rB$ the scheme which the affine space corresponding to the vector space
\[
    \dsum_{i=1}^N \Gamma(C,\omega_C^{\tsr i}) .
\]

Define the Hitchin map $H\colon \Higgs\to\rB$ as follows.  For a $\k$-point $y$ of $\Higgs$ corresponding to a Higgs bundle $(\cE,a)$, we define~$H(y)$ to be the point of~$\rB$ given by $(\tr a,\tr \Lambda^2a,\dots,\tr\Lambda^Na=\det a)$ (one can extend this to $S$-points in a straightforward way).  Here is another interpretation of the Hitchin map.  Note that a Higgs field on a given vector bundle~$\cE$ is equivalent to a map $\cT_C\tsr\cE\to\cE$, and therefore to an action of $\Sym\cT_C$ on~$\cE$.  In other words, a Higgs bundle of rank~$N$ is equivalent to a coherent sheaf~$\tilde\cE$ on~$T^*C$ whose pushforward to~$C$ is a rank~$N$ vector bundle.  Now define a divisor $\tilde C \ctd T^*C$ as the ``support with multiplicities'' of~$\tilde\cE$ (\ie each irreducible component of $\supp\tilde\cE$ is taken with multiplicity equal to the length of the stalk of~$\tilde\cE$ at the generic point of that component).

It is clear that the divisor~$\tilde C$ is finite of degree~$N$ over~$C$.  We claim that such divisors are naturally parametrized by~$\rB$. Indeed, let $\pi\colon T^*C\to C$ be the projection, and let $s$ be the canonical section of~$\pi^*\omega_C$.  Then any $S$-point $b$ of~$\rB$ given by $(\tau_1,\dots,\tau_N)$ where $\tau_i\in\Gamma(C\x S,\omega_C^{\tsr i} \bx \O_S)$ defines a section~$t_b$ of $\pi^*\omega_C^{\tsr N} \bx\O_S$ by the formula
\[
    t_b = s^{\tsr N}\bx 1 + \sum_{i=1}^N (-1)^i (\pi\x\id_S)^*\tau_i\tsr (s^{\tsr N-i} \bx 1).
\]
The divisor~$\tilde C_b$ of zeroes of~$t_b$ is finite of degree~$N$ over $C\x S$, and it is easy to see that $b\mapsto \tilde C_b$ defines a one-to-one correspondence between maps $S\to\rB$ and divisors in $C\x S$ finite of degree~$N$ over~$S$.  Moreover, in the situation of the previous paragraph, the point $H(y)$ corresponds to the divisor~$\tilde C$: $\tilde C=\tilde C_{H(y)}$.  (This is essentially because the support divisor of~$\tilde\cE$ can be computed using the characteristic polynomial of~$a$.)  Also, we will need the universal spectral curve $\tilde\rC\ctd T^*C\x\rB \to C\x\rB$ (so that, in the above notation, $\tilde C_b = \tilde\rC\xx_\rB \{b\}$).

The Hitchin map is equivariant with respect to the natural actions of~$\Gm$: namely, the action on $\Higgs$ is defined by $\la\cdot(\cE,a)=(\cE,\la a)$ and the action on $\rB$ is given by
\begin{equation} \label{eq:Gm acts on B}
    \la\cdot(\tau_i)_{i=1}^N = (\la^i\tau_i)_{i=1}^N
\end{equation}
In terms of spectral curve~$\tilde C$, the latter action corresponds to dilation of~$\tilde C$ along the fibers of the map $T^*C \to C$.

We will be interested in the open subset of~$\rB$ parametrizing smooth spectral curves, that is, the maximal open subset $\bo\ctd\rB$ for which the map $\tilde\rC^\circ:={\tilde\rC\xx_\rB\bo}\to\bo$ is smooth.  One can show that $\bo$ is non-empty, and that fibers of $\tilde\rC^\circ\to\bo$ are irreducible (and smooth).  Denote also by $\Higgs^\circ$ the preimage of~$\bo$ under the Hitchin map~$H$:
\[
    \Higgs^\circ := {\Higgs} \xx_\rB \bo\xra{H^\circ}\bo    .
\]
We conclude this subsection by stating the following \begin{prop}[{cf.~\cite[Corollary~4.5]{BB}}]
There is a natural identification $\Higgs^\circ \isom \Pic(\tilde\rC^\circ/\bo)$.
\end{prop}
 \begin{proof}[Proof sketch]We explain the construction of the map $\hig\to\Pic(\tco/\bo)$: for an $S$-point~$y$ of $\Higgs^\circ$, one can define a line bundle on $\tilde C_{H(y)} := \tilde\rC\xx_{\rB,H\circ y}S$ as follows.  Let $(\cE,a)$ be the $S$-family of Higgs bundles corresponding to~$y$.  As discussed above, this is the same as a coherent sheaf~$\tilde\cE$ on $T^*C\x S$, and the support of this sheaf is~$\tilde C_{H(y)}$.  Moreover, since $H(y)\in\bo(S)$, the spectral curve $\tilde C_{H(y)}$ is smooth over~$S$, and $\tilde\cE$ must be the pushforward of a line bundle on~$\tilde C_{H(y)}$. This is the desired line bundle, which gives an $S$-point of $\Pic(\tilde\rC^\circ/\bo)$.
\end{proof}

\subsection{The $p$-Hitchin fibration}\label{p-hitchin}
In this subsection, we will present a description of the stack $\Loc=\Loc_N$ of de~Rham local systems of rank~$N$ on~$C$, analogous to the one given above for $\Higgs$.  Recall that by ``de~Rham local system'' we just mean a vector bundle with a flat connection (since $C$ is one-dimensional, all connections on it are automatically flat), so for a given test scheme~$S$, the groupoid $\Loc_N(S)$ is defined as that of rank~$N$ vector bundles on $C\x S$ equipped with a connection in the $C$-direction.

The construction of the $p$-Hitchin map is similar to that of the ordinary Hitchin map, but uses the notion of $p$-support, and so exists only in positive characteristic.  It is a map
$$\chi\colon \Loc\to\rB\fr$$
 defined as follows.  Suppose we are given an $S$-point of~$\Loc$ defined by an $S$-family of local systems, \ie a vector bundle~$\cE$ on $C\x S$ of rank~$N$ with a connection~$\na$ relative to~$S$.  We can think of~$(\cE,\na)$ as an $S$-family of $\D$-modules on~$C$; in particular, similarly to the Higgs field case, we can define its $p$-support with multiplicities---this is a divisor in $T^*C\fr\x S$ finite of degree~$N$ over $C\fr\x S$.  The corresponding $S$-point of~$\rB\fr$ is by definition the value of~$\chi$ on~$(\cE,\na)$.  Again, another way to define it is to apply the invariant polynomials to the $p$-curvature map $\curv_p(\na)\colon \cE\to\cE \tsr (\Fr_C^*\Om^1_{C\fr}\bx\O_S)$.

One can show that, \'etale locally over~$\rB\fr$, the $p$-Hitchin fibration $\chi\colon \Loc\to\rB\fr$ looks like the (Frobenius twisted) usual Hitchin fibration $H\fr\colon \Higgs\fr\to\rB\fr$ (see~\cite{Gr,CZ-NA Hodge}).  The identification of formal neighborhoods of fibers over a given point of~$\rB\fr$ can be constructed using a splitting of the Azumaya algebra~$\tilde\D_C$ on the formal neighborhood of the corresponding spectral curve.  Similarly, an \'etale local identification near a given point $b\in\rB\fr$ can be obtained from a splitting of the pullback of~$\tilde\D_C$ to $\tilde\rC\fr\xx_{\rB\fr}U$ where $U\to\rB\fr$ is an \'etale neighborhood of~$b$ in~$\rB\fr$.  (See \emph{ibid.}\ for the proof of the existence of such a splitting.)  The identification is canonical up to the action of a section of the group stack $\Pic(\tilde\rC\fr/\rB\fr)$ on $\Higgs\fr$.  %, where the latter stack is identified with the ``compactified'' relative Picard stack of~$\tilde \rC\fr$  over~$\rB\fr$.

We will be mostly concerned with the part of~$\Loc$ lying over $(\bo)\fr$ which we will denote by
\[
    \Loc^\circ := \Loc \xx_{\rB\fr} (\bo)\fr \xra{\chi^\circ}(\bo)\fr.
\]
As explained above, we have an identification $\Higgs^\circ \isom \Pic(\tilde\rC^\circ/\bo)$, and hence $(\Higgs^\circ)\fr \isom \Pic((\tilde\rC^\circ)\fr/(\bo)\fr)$.  Moreover, these identifications are compatible with the action of the corresponding Picard stacks.  Therefore, from the results discussed in the previous paragraph, we see that the stack $\Loc^\circ$ carries a natural structure of $\Pic((\tilde\rC^\circ)\fr/(\bo)\fr)$-torsor.  This torsor can be described as that of fiberwise (along fibers of $(\tilde\rC^\circ)\fr \to (\bo)\fr$) splittings of the Azumaya algebra $(\pr^{\tilde\rC^\circ}_{T^*C})\fr^*\tilde\D_C$ where $\pr^{\tilde\rC^\circ}_{T^*C}$ is the natural projection $\tilde\rC^\circ \to T^*C$ (obtained by restriction from the projection $\tilde\rC \to T^*C$).

So we get
\begin{prop}\label{prop:loc=splittings}
  The stack $\loc$ is identified with the $\Pic((\tco)\fr/\bof)$-torsor of fiberwise splittings of $(\pr^{\tco}_{T^*C})\fr\dd_C$.
\end{prop}
\subsection{$\D$-modules on \texorpdfstring{$\Bun$}{Bun}, the Abel--Jacobi map, and Hecke functors}\label{subs:Dmod(Bun) AJ Hecke}
Now we apply the above results to the study of $\D$-modules on~$\Bun$.  According to Theorem~\ref{Dmod-stk equiv gerbe} (with $\L=\O$), $\D$-modules on~$\Bun$ are classified by a certain gerbe $\dd_{\Bun}$ on $\Higgs\fr = T^*\Bun\fr$.  The class of this gerbe on the smooth part $(\Higgs\sm)\fr$ of $\Higgs\fr$ (in particular, on~$(\Higgs^\circ)\fr$) corresponds to the canonical $1$-form $\th_{\Higgs}\fr$ on $(\Higgs\sm)\fr$ as on (the smooth part of) a cotangent bundle.

In~\cite{BB}, it is shown that the gerbe $\dd_{\Bun}|_{(\Higgs^\circ)\fr}$ has a canonical structure of a multiplicative gerbe with respect to the group structure on $(\Higgs^\circ)\fr \to (\bo)\fr$.  Moreover, it is also proved that the pullback of this gerbe by the Abel--Jacobi map is identified with the pullback of the gerbe $\dd_C$ by the map $(\tilde\rC^\circ)\fr \to T^*C\fr$.  One way to see both statements is on the level of the corresponding $1$-forms.  Namely, denote by $\th_{T^*C}$ the canonical $1$-form on~$T^*C$.  Then our statements about $\dd_{\Bun}$ follow from the fact that
\[
    \pr_1^*(\pr^{\tilde\rC^\circ}_{T^*C})^*\th_{T^*C} + \pr_2^*\th_{\Higgs} = m_{\AJ}^*\th_{\Higgs}
\]
where $\pr_1,\pr_2\colon \tilde\rC^\circ\xx_{\bo}\Higgs^\circ \to \Higgs^\circ$ are the projection maps, and $m_{\AJ}\colon \tilde\rC^\circ\xx_{\bo}\Higgs^\circ \to \Higgs^\circ$ is the composition of $\AJ_{\tilde\rC^\circ} \xx_{\bo} \id_{\Higgs^\circ}$ and  the group structure on $\Higgs^\circ \to \bo$.  Now, since multiplicative splittings of multiplicative gerbes on Picard stack of a curve correspond to (all) splittings of its pullback by Abel--Jacobi map (see~\ref{subs:Pic(C) case}), we get that the torsor of multiplicative splittings of $\dd_{\Bun}|_{(\Higgs^\circ)\fr}$ is identified with $\Loc^\circ$, which gives rise to the geometric Langlands equivalence of~\cite{BB}.

The above identity on $1$-forms follows, in turn, from an identification of the graph of~$m_{\AJ}$ with (the intersection of $\Higgs^\circ$ with) the conormal bundle to a certain substack $\rH_1$ of (or rather a stack mapping to) $C\x\Bun\x\Bun$.  The stack~$\rH_1$ classifies inclusions $\cE_1\into\cE_2$ of rank~$N$~bundles with cokernel of length~$1$.  When we pass from $1$-forms to gerbes, the identity gives a splitting of $\dd_{C\x\Bun\x\Bun}$ on the conormal bundle to~$\rH_1$, and the resulting functor $\D\hmod(\Bun)\to\D\hmod(C\x\Bun)$ is identified with the pull-push via the correspondence~$\rH_1$ (again, localized to $\Higgs^\circ$).  This pull-push functor is known as the \emph{Hecke functor} of geometric Langlands.

Using Proposition~\ref{prop:mult grb on Pic} and Proposition~\ref{prop:loc=splittings} we now get
the following statement, which is the main theorem of~\cite{BB}:
\begin{prop}[{\cite[Theorem~4.10]{BB}}]\label{prop: loc dual dbun}
  We have:
   \begin{enumerate}
     \item  The gerbe $\dd_{\Bun}^\circ:=\dd_{\Bun}|_{(\hig)\fr}$ has a structure of multiplicative gerbe on the group stack $(\hig)\fr\to\bof$.
     \item The dual torsor to $\ddbuo$ in the sense of Proposition~\ref{prop:dual gerbe on torsor} is identified with $\loc\to\bof$.
%     \item
   \end{enumerate}
\end{prop}

%\subsection{$\D$-modules on \texorpdfstring{$\Bun$}{Bun}}
%Now we apply the above results to the main objects of study in this paper---twisted $\D$-modules on~$\Bun$.  Let us begin by recalling the non-twisted case.  According to Theorem~\ref{Dmod-stk equiv gerbe} (with $\L=\O$), $\D$-modules on~$\Bun$ are classified by a certain gerbe on $\Higgs\fr = T^*\Bun\fr$.  The class of this gerbe on the smooth part $(\Higgs\sm)\fr$ of $\Higgs\fr$ (in particular, on~$(\Higgs^\circ)\fr$) corresponds to the canonical $1$-form on $(\Higgs\sm)\fr$ as on (the smooth part of) a cotangent bundle.
%
%Now we turn to the twisted case.  The twists that we will consider are of the form $\Ldet^c$.  Here $\Ldet$ denotes the determinant bundle on~$\Bun$: for any vector bundle~$\cE$ on~$C$ of rank~$N$, the fiber of~$\Ldet$ at the corresponding point of~$\Bun$ is given by $\det \RGam(\cE)$ (and similarly for families).  In Appendix~\ref{appx}, we prove that the corresponding twisted cotangent bundle is isomorphic to the moduli space $\Lochf$ of $\omega_C^{\tsr1/2}$-twisted de~Rham local systems of rank~$N$ on~$C$.  Therefore, according to Theorem~\ref{Dmod-stk equiv gerbe}, for any $c\in\k\setm\Fp$, the $\Ldet^c$-twisted $\D$-modules on~$\Bun$ are classified by a $\Gm$-gerbe on $\Lochf\fr$.

\section{Proof of Theorem~\ref{main_thm}}\label{sec:proof}

\subsection{Outline of  the argument}
\label{subs:outline}

After introducing all the necessary tools in the previous sections, we are ready to give a proof of Theorem~\ref{main_thm}, which is what this section is devoted to.  So fix $c\in\k\setm\Fp$ and a smooth irreducible projective curve $C$ of genus $g(C)>1$.  Recall that in the statement of Theorem~\ref{main_thm} we are interested in the twisted $\D$-modules on $\Bun$ with the  twisting given by $\Ldet^c$ and $\Ldet^{-1/c}$.  Here $\Bun$ is the stack of rank~$N$ vector bundles on~$C$ and $\Ldet$ is the line bundle given by taking the determinant of derived global sections of a vector bundle.

By Theorem ~\ref{Dmod-stk equiv gerbe}, the category $\D\hmod_{\Ldet^c}$ is described by a $\Gm$-gerbe $\dd_{\Bun,\Ldet^c}$    on $(\tilde T^*_{\Ldet}\Bun)\fr$, which by Theorem~\ref{appx|main thm} is identified with $\Lochf\fr$ where $\Lochf$ is the moduli space of rank~$N$ bundles with $\ohf_C$-twisted connection.  We will make use of the identification
$$\Lochf\isoto\Loc$$
 given by twisting vector bundles by $\omega_C^{\tsr(p-1)/2}$.

 We let $\chi\fr$ be the Frobenius twist of the $p$-Hitchin map defined in~\ref{p-hitchin}
 and  we restrict $\dd_{\Bun,\Ldet^c}$ to the preimage $\loc\isoto\Lochf^\circ$ of
  $$B:=(\bo)\frr$$ under the map $\chi\fr\colon \Loc \fr \to \rB\frr$.  We will denote the resulting  gerbe by
$$\dd_c=\dd_{\Ldet^c}|_{(\Loc^\circ)\fr}\longto(\Loc^\circ)\fr.$$
Replacing $c$ by $1/c$, we define another gerbe $\dd_{1/c}$ on $(\Loc^\circ)\fr$.

Theorem~\ref{main_thm} now says   that there is an equivalence $$\dd_c\hmod \isoto\dd_{1/c}^{-1}\hmod$$
which is given by a splitting of the gerbe
$$(\dd_c\bx\dd_{1/c})\xx_{B\x B}\Delta_{c^p}(B)$$ (here $\bx$ is the Baer sum of pullbacks of $\Gm$-gerbes along two projections $(\loc)\fr\x(\loc)\fr\to(\loc)\fr$ and $\Delta_{c^p}\in B\x B$ is the graph of the action  map
$$[c^p]\colon B\to B$$
of $c^p\in\Gm(\k)$ on $B$, see~\eqref{eq:Gm acts on B}).

 \smallskip

The strategy of proof is as follows:   we apply Proposition~\ref{prop:trsr-grb duality data} to show that the categories on both sides are categories of coherent sheaves for dual gerbes on torsors, and then the equivalence follows from Proposition~\ref{prop:dual gerbe on torsor}. In more detail, recall from \ref{p-hitchin} that the stack $\Loc^\circ$ has a structure of torsor for the relative group stack $(\Higgs^\circ)\fr= \Pic(\tilde\rC^\circ/\bo)\fr$ and thus by Frobenius twist, $(\Loc^\circ)\fr$ is a torsor for $(\Higgs^\circ)\frr$.

To match the notation of section~\ref{sec:FM}, we take
\begin{align*}
% \nonumber to remove numbering (before each equation)
 S= B &:= (\bo)\frr,\\  \rX=\rY&:=  \Pic(\tilde\rC^\circ/\bo)\frr =(\Higgs^\circ)\frr\xra{(H^\circ)\frr}B,\\  \tx&:=(\Loc^\circ)\fr\xra{(\chi^\circ)\fr} B,\\  \ty&:=(\Loc^\circ)\fr \xra{[c^p]\circ (\chi^\circ) \fr} B,\\
  \gtx&=\dd_c=\dd_{\Bun,\Ldet^c}|_{(\Loc^\circ)\fr}\longto(\Loc^\circ)\fr=\tx,\\  \gty&=\dd_{1/c}=\dd_{\Bun,\Ldet^{1/c}}|_{(\Loc^\circ)\fr}\longto (\Loc^\circ)\fr=\ty
 .
\end{align*}
%(Here $[c^p]$ is the image of $c^p\in\Gm(\k)$ via the natural action of $\Gm$ on~$B\ctd \rB\frr$.)
%On $\tx$ we have a gerbe $\gtx$ given by restriction of the gerbe $\dd_{\Bun,\Ldet^c}$ from $\Lochf\frr\isoto\Loc\frr$ to  $(\Loc^\circ)\frr = \rY$.  Similarly, on $\ty$ we have the restriction~$\gty$ of $\dd_{\Bun,\Ldet^{-1/c}}$.

We will also need ``Frobenius untwists'' (over~$\k$) of $\tx,\ty$, so for a scalar $c\in\k^\x$ we denote by
 \[\rT_c \to (\bo)\fr
 \]the torsor over the group scheme $(\Higgs^\circ)\fr\to(\bo)\fr$ which is obtained by pullback from the standard torsor $\Loc^\circ \to (\bo)\fr$ under the action of $c\in\k^\x$ on $(\bo)\fr$.  Then we can identify
 \[ \tx = \rT_1\fr,\quad    \ty = \rT_c\fr.\]
(Note that $[c^p]$ becomes $[c]$ after untwisting.)

Now, to apply Proposition~\ref{prop:trsr-grb duality data}, we need to have multiplicative gerbes $\gx$~and~$\gy$ on $\rX$~and~$\rY$ and their action on $\gtx$ and $\gty$, respectively.  For $\gy$ we take the restriction of the gerbe $\dd_{\Bun\fr}$ to
\[\rY = (\Higgs^\circ)\frr \ctd \Higgs\frr = T^*\Bun\frr,
\] and similarly we take $\gx$ to be the restriction of $[c^p]^*\dd_{\Bun\fr}$ to the same subset (but now identified with~$\rX$):
\[
    \gy=\dd_{\Bun\fr}|_\rY,\quad\gx=[c^p]^*\gy  .
\]where $[c^p]$ is now
the action on $(\hig)\fr$.

The duality between $\gy$ and $\tx$ (as required in the data~(\ref{idfn trsrs w duals to grbs}) from Proposition~\ref{prop:trsr-grb duality data}) is essentially the result of~ \cite{BB} and it  follows from Proposition~\ref{prop: loc dual dbun}.  We see also that $\gx$ is dual to $\ty$ since these are obtained from $\gy$ and $\tx$ by pullback under $[c^p]\colon B\to B$.  We begin our proof by first constructing an action of $\gx$ on $\gtx$ and of $\gy$ on  $\gty$ which is done in Sect.~\ref{sec:qhecke}. This reduces to an equality (Proposition~\ref{prop A}) on the generalized one-form $$\tilde\th=\a_{\Ldet}=\tcurv\Lna_{\det}$$ where $\Lna_{\det}:=(\Ldet',\na_{\det})$ is the pullback of  $\Ldet$ to its twisted cotangent, identified with $\Lochf\isoto\Loc$ as above.

 By Remark~\ref{rem: data for dty up to gerb on base}, this gives the desired equivalence up to twisting by $\Gm$-gerbe on the base.  In other words, the splitting $\cQ$  of  $\gtx\bx_B\gty$ from Proposition~\ref{prop:trsr-grb duality data} is defined locally  on~$B$ up to a line bundle pulled back from ~$B$.  It also gives rise  to a quantum version of Hecke functors.

In order to split this ``difference'' gerbe on $B$, we give an explicit formula for $\cQ$ in \ref{subs:poincare}.  Since by Corollary~\ref{cor DLc eqv GcaL}  the gerbes $\gtx,\gty$ are of the form $\dd_\a$ for a generalized $1$-form~$\a$, the splittings correspond to line bundles with connection of a certain extended curvature.

This line bundle with connection is constructed in a certain way from $\Lna_{\det}$ (formula~\eqref{sfb}), and we have to check that the connection has the right extended curvature.  We show that if it does, then we can also construct the other data in Proposition~\ref{prop:trsr-grb duality data}.

Unlike the equality~\eqref{tth id} of Proposition~\ref{prop A}, the necessary equation~\eqref{eq:hard tth id} for the construction of~$\cQ$ involves a linear combination of pullbacks of~$\tilde\th$ with non-integer coefficients, so it cannot be expressed as an equation on the extended curvature of a connection.  This equation involves certain compatibility of $\tilde\th$ with the addition map
\begin{equation}\label{eq:a}
 a\colon \rT_1 \xx_{(\bo)\fr} \rT_{c} \ \longto \ \rT_{1+c}
\end{equation}
and is proved in~\ref{subs:alt constr} and \ref{subs:proof Jer} by constructing a representative one-form $\th_0\in\delta^{-1}(\tilde\th)$ on $\Loc^\circ$.

The form $\th_0$ is constructed from its symplectic dual vector field $\xi_0$. This vector field is constructed by lifting the  action of the Lie algebra of infinitesimal dilations from $T^*C\fr$ to the gerbe $\dd_C$.   In \ref{subs:alt constr} we only prove the required equality $\tth=\delta(\th_0)$  up to a summand of the form $(\chi^\circ)^*\sP(\beta_0)$ for $\b_0\in\eF(\bo)$ (we actually show it extends to all of $\rB$), as well as an analog of \eqref{eq:hard tth id}
for $\th_0$.  The proof of the equality $\b_0=0$ requires an additional degree estimate on $\th_0$ along the fibers of $\Loc\to\Bun$ done in the final subsection~\ref{subs:proof Jer}.

%Since we are going to construct our quantum geometric Langlands equivalence as a twisted Fourier--Mukai transform, we have to apply Propositions \ref{prop:dual gerbe on torsor}~and~\ref{prop:trsr-grb duality data}.

%According to \cite{BB}, the gerbe $\gy$ is a multiplicative gerbe and the dual torsor is identified with $\tx = (\Lochf^\circ)\fr$.
%This matches So,
%Thus, to apply Proposition~\ref{prop:trsr-grb duality data}, we need to construct an action of $\gx$ on $\gtx$ and of $\gy$ on $\gty$, as well as the data (\ref{splitting Q})--(\ref{isoms aX aY}).

\subsection{Quantum Hecke functors}\label{sec:qhecke}

In this subsection we will prove the following
\begin{prop}
Denote $(\ty',\gy',\gty')$ be the dual data to $(\tx,\gx,\gtx)$ as described in Proposition~\ref{prop:dual gerbe on torsor}.
  There is  an action of $\gx$ on $\gtx$ and of $\gy$ on $\gty$.
  Together with the isomorphisms $\ty\isoto\ty'$ and $\gy\isoto\gy'$, this implies by Remark~\ref{rem: data for dty up to gerb on base} that there is a $\Gm$-gerbe
  $\g_B$ on $B$ such that we have
  $$\gty\isoto\gty'\cdot(\chi^\circ)\fr^*\g_B.$$
%  which gives a F
\end{prop}
Below we construct an action of $\gx$ on $\gtx$.
The action of $\gy$ on $\gty$ is constructed similarly.

According to Proposition~\ref{prop:Morita eqv stk} and Corollary~\ref{cor DLc eqv GcaL}, we have
\begin{equation}\label{gtx}\gtx\sim\dd_{\Loc^\circ,\Ldet'^c,\na}
\sim\dd_{\loc,c\tth}
 \end{equation}
 where $\Ldet'$ is the pullback of the determinant line bundle under $\Loc^\circ\to\Bun$, and $\na$ is the canonical (``universal'') connection on this pullback. One can also see from Corollary~\ref{cor untw} that
\begin{equation}\label{gx}\gx \sim \dd_{(\Higgs^\circ)\fr,c^p\th\frr}
\sim\dd_{(\hig)\fr,\delta(c\th\fr)}
 \end{equation}
 where we denote by $\th=\th_{\Higgs}|_{\hig}$ the restriction of the  canonical $1$-form on the smooth part of~$\Higgs$ coming from the identification $\Higgs\isoto T^*\Bun$.

Denote by~$\Gamma$ the graph of action of $(\Higgs^\circ)\fr$ on~$\Loc^\circ$. It comes equipped with three projections: $$\prg1\colon \Gamma\to(\Higgs^\circ)\fr,\quad  \prg2,\prg3\colon \Gamma\to\Loc^\circ.$$  (Note that the graph of the action of $\rX$ on~$\tx$ is identified with $\Gamma\fr$.)  We are going to use the following proposition:

\begin{prop} \label{prop A}
The generalized $1$-form $\tilde\th$ satisfies the following identity:
\begin{equation}\label{tth id}          \ncmd{\emspace}{\hspace{1em}}
  \prg2^*\tilde\th - \prg3^*\tilde\th =     \delta(\prg1^*\th\fr)
\end{equation}
\end{prop}

Now it is easy to see that Proposition~\ref{prop A} allows to construct an action of~$\gx$ on~$\gtx$.
Indeed, multiplying both sides of~\eqref{tth id} by $c$ and applying the construction $\b\mapsto\dd_\b$
of Sect.~\ref{extcrv} for $\b\in\eF_\Gamma$,
we get an equivalence of gerbes
\[\prg1\fr^*\gx\cdot\prg2\fr^*\gtx \sim \prg3\fr^*\gtx\]
where we used the expressions \eqref{gx}~and~\eqref{gtx}
for the gerbes $\gx,\gtx$.
This gives the desired action of $\gx$ on $\gtx$  (up to  compatibility isomorphisms, which follow from the corresponding equalities of generalized $1$-forms on  multiple products,
or rather from the fact that $\a\mapsto\dd_\a$ is map of sheaves of symmetric monoidal 2-groupoids from $\eF_\Gamma$ (considered as a \emph{discrete} 2-groupoid) to the sheaf of 2-groupoids of $\Gm$-gerbes on $\Gamma\fr$).
\begin{rem}\label{rem:q-Hecke}
  The existence of natural actions of~$\gx$ on~$\gtx$ and of~$\gy$ on~$\gty$ has two consequences.  First, if we restrict (say) the first action to the image of the Abel--Jacobi map  $\AJ\frr\colon (\tco)\frr\to \Pic((\tco)\frr/B)\isoto \rX$, the resulting splitting of $\pr_1^*\AJ\frr^*\gx^{-1}\tsr\pr_2^*\gtx^{-1}\tsr\pr_3^*\gtx$  gives an $\O_{C\fr}\bx\Ldet^c\bx\Ldet^{-c}$-twisted $\D$-module on $C\fr\x\Bun\x\Bun$.  There is a similar construction for non-twisted $\D$-modules giving a $\D$-module on $C\x\Bun\x\Bun$, and this $\D$-module coincide with (the localization to $(\Higgs^\circ)\fr$ of) the one defining the Hecke functors (see~\ref{subs:Dmod(Bun) AJ Hecke}).  Therefore it is natural to use the twisted $\D$-module on $C\fr\x\Bun\x\Bun$ just described, to define a functor $\D\hmod_{\O_{C\fr}\bx\Ldet^c} (C\fr\x\Bun) \to \D\hmod_{\Ldet^c}(\Bun)$ which could be called a \emph{quantum ($p$\nbhp-)Hecke functor}.  In fact, just like the usual Hecke functors, this functor can be constructed from a certain $G(\k\dbkts t)$-equivariant twisted $\D$-module on $\Gr_G$ (the affine Grassmannian for~$G$).  We note that in characteristic~$0$ the category of such twisted $\D$-modules is trivial for irrational~$c$, but not necessarily for $c\in\QQ$, in which case it was described by M.~Finkelberg and S.~Lysenko in~\cite{FL}.

  Second, from the Remark~\ref{rem: data for dty up to gerb on base} we see that  the existence of these two actions already allows to define the desired equivalence $\gtx\hmod \isoto \gty\hmod$ locally over~$B$ up to twisting by a line bundle pulled back from~$B$.  Under this equivalence, the $\D$-modules corresponding to ``skyscraper sheaves'' on the corresponding gerbes\footnote{Since $(\Loc^\circ)\fr$ is a stack rather than a scheme, one should be careful about what is meant by ``skyscraper'' here (basically, these are pushforwards of splittings of the gerbe on points), but we won't go into details.} go to eigen-$\D$-modules with respect to the quantum $p$-Hecke functors described above.
\end{rem}

\bigskip

Below we will need an explicit description of the action of $\gx$ on $\gtx$ by a line bundle with connection on $\rT_1\xx_{\bof}\rT_c$.

%Using the results of Section~\ref{sec:diff op charp}, we see that $\gx$ is equivalent to the gerbe~$\dd_{c^p\th_{\Higgs}\frr}$ %(classifying line bundle with flat connection of $p$-curvature $c^p\th_{\Higgs}\frr$ where the factor~$c^p$ comes from the scaling of projections to~$\rB\frr$),%
%and $\gtx\sim\dd_{c\tilde\th}$ (the gerbe classifying line bundles with connection of extended curvature~$c\tilde\th$).  Therefore the gerbe
From \eqref{gx} and \eqref{gtx} we see that the gerbe
$$\prg2\fr^*\gtx\tsr\prg3\fr^*\gtx^{-1}\tsr\prg1\fr^*\gx^{-1}$$ on~$\Gamma$ that we need to split is equivalent to $$\dd_{c\,\prg2^*\tilde\th-c\,\prg3^*\tilde\th-\prg1^*\sP(c^p\th_{\Higgs}\frr)}.$$

Now, Proposition~\ref{prop A} implies that
$$c\,\prg2^*\tilde\th - c\,\prg3^*\tilde\th - \prg1^*\sP(c^p\th_{\Higgs}\frr) = c\delta(\prg1^*\th_{\Higgs}\fr) - c^p\sP(\prg1^*\th_{\Higgs}\frr)= -\kappa(c\,\prg1^*\th_{\Higgs}\fr).$$
Therefore $(\O_\Gamma,d-c\,\prg1^*\th_{\Higgs}\fr)$ gives the desired splitting.  The compatibility isomorphism on $\Loc^\circ\xx_{\rB\fr}\loc\xx_{\rB\fr}(\Higgs^\circ)\fr$ is given by the identity morphism $\id_\O$ on the corresponding line bundles with connections.  Compatibility of this isomorphism with connections also follows from~\eqref{tth id} and the cocycle condition is evident.

\bigskip Now we prove Proposition~\ref{prop A}.  It is easy to see that  it is equivalent to the following statement:
\begin{prop}\label{prop A'}
The line bundle with connection
\[
    (\cS,\na_\cS):= \prg1^*(\O,d+\th_{\Higgs}\fr)\tsr \prg2^*(\Ldet',\na)\tsr \prg3^*(\Ldet',\na)^{\tsr-1}
\]
on~$\Gamma$ is flat and has $p$-curvature $\prg3^*\th_{\Higgs}\frr$.
\end{prop}

For further reference, also define $\na_\cS'=\na_\cS-\prg1^*\th_{\Higgs}\fr$ so that
\[
   (\cS,\na_\cS'):= \prg2^*(\Ldet',\na)\tsr \prg3^*(\Ldet',\na)^{\tsr-1}.
\]

In terms of $\na_\cS'$, Proposition~\ref{prop A'} amounts to the following identity:% (in the middle):
\begin{equation}\label{tth id2}          \ncmd{\emspace}{\hspace{1em}}
  %\prg2^*\tilde\th - \prg3^*\tilde\th =
  \tcurv(\na_\cS')=
   \sP(\prg1^*\th_{\Higgs}\frr) - \kappa(\prg1^*\th_{\Higgs}\fr) = \delta(\prg1^*\th_{\Higgs}\fr)
\end{equation}
where $\sP,\kappa,\delta$ are defined in~\ref{extcrv}.

\ncmd{\GAJ}{\Gamma\!_{\AJ}}

\begin{proof}[Proof of Proposition~\ref{prop A'}]
  By an argument similar to \cite[Section~4.14]{BB}, it is enough to prove that the restriction of the two sides of~\eqref{tth id} to $$\GAJ:=(\tilde\rC^\circ)\fr\xx_{\rB\fr}\Loc^\circ \xra{\AJ\fr\x\id} (\Higgs^\circ)\fr\xx_{\rB\fr}\Loc^\circ \isoto\Gamma$$ are equal.  (Here $\AJ$ is the Abel--Jacobi map $\tilde\rC^\circ \to \Higgs^\circ = \Pic(\tilde\rC^\circ/\bo)$.)  In other words, we need to prove that the restriction of the line bundle with connection $(\cS,\na_\cS)$ from Proposition~\ref{prop A'} to $\GAJ\to\Gamma$ is flat and has $p$-curvature
  \[
    \pr_1\frr^*\AJ_{\tilde\rC^\circ/\bo}\frr^* \th_{\Higgs}\frr = \pr_1\frr^*(\pr^{\tilde\rC^\circ}_{T^*C})\frr^* \th_{T^*C}\frr,
  \]
  where $\th_{T^*C}$ is the canonical form on $T^*C$ (see~\ref{subs:Dmod(Bun) AJ Hecke} and \cite[Sect.~4.16]{BB} for the proof of the equality $\AJ_{\tilde\rC^\circ/\bo}^* \th_{\Higgs} = (\pr^{\tilde\rC^\circ}_{T^*C})^* \th_{T^*C}$).
  This will follow from an alternative description of $(\cS,\na_\cS)$. Namely, we have an equivalence $\D_C\hmod \isoto \D_{T^*C,\th_{T^*C}}\hmod$. Moreover, the ``in-families'' version of this is true. So if $\rZ$ is any stack then the category of ``$\rZ\fr$-families of $\D_C$-modules'' (\ie quasi-coherent sheaves on $C\x\rZ\fr$ equipped with a connection along $C$) is equivalent to the category of $\D_{T^*C\x\rZ,\pr_1^*\th_{T^*C}}$-modules. If we replace $C$ by $C\fr$ and $\rZ$ by $\Loc^\circ$ then we have the universal bundle with connection (along $C\fr$) on $C\fr\x(\Loc^\circ)\fr$, so applying the equivalence gives a $\D$-module on $T^*C\fr\x \Loc^\circ$ with $p$-curvature $\pr_1^*\th_{T^*C}\fr$. It is supported on $(\tilde\rC^\circ)\fr\xx_{(\bo)\fr} \Loc^\circ\isom\GAJ$ and therefore corresponds to a $\D_{\GAJ,\pr_1^*\th_{T^*C}\frr}$-module which we denote by $\Lna_{\on{univ},\GAJ}$.

  \begin{lem} \label{chdet}
  The restriction to $\GAJ$ of the line bundle with connection $(\cS,\na_\cS)$ from Proposition~\ref{prop A'} is isomorphic to $\Lna_{\on{univ},\GAJ}$ where we identify $\Loc$ with $\Lochf$ (and thus $\loc$ with $\Lochf^\circ$) via twisting by $\omega_C^{(p-1)/2}$.
  \end{lem}

  \begin{proof}
    We will show that the lemma follows directly from Proposition~\ref{appx|prop:det on I}.  Indeed, first note that $\Lna_{\on{univ},\GAJ}$ is isomorphic to the pullback of the bundle $\Lna_{\on{univ}}$ defined in~\ref{appx|subs:charp} under certain map $e\colon \GAJ\to \rI$.  (See ibid.\ for the definition of~$\rI$.)  Namely, from the definition of $\GAJ$ it is easy to see that a point $\gamma\in\GAJ$ corresponds to a pair of rank~$N$ bundles with $\ohf$-connections $(\cE_1,\na_1)$, $(\cE_2,\na_2)$ that fit into a short exact sequence of $\D_{C,\ohf}$-modules
    \begin{equation}\label{EEF}
       0\to \cE_2\to \cE_1\to \cF\to 0
    \end{equation}
    where $\cF\isom\delta_{x,\xi}$ is an irreducible $\D_{C,\ohf}$-module with $p$-support at a point $(x,\xi)\in T^*C\fr$.  Let $b_\cF$ be the point of the stack~$\rI$ corresponding to~$\cF$.  Then $e$ is defined by setting $e(\gamma) = b_\cF$.  The isomorphism
    \[
        \Lna_{\on{univ},\GAJ} \isom e^*\Lna_{\on{univ}}
    \]
    is easy from comparison of constructions of the two line bundles with connection.  On the other hand, the projections $\prg i|_{\GAJ}\colon \GAJ\toto\Lochf \ctd \Connhf$ ($i=2,3$) correspond to the bundles $\cE_1,\cE_2$.  So we can apply Lemma~\ref{appx|lem:Ldet ses-compat} to the triple $(\prg2|_{\GAJ},\prg3|_{\GAJ},e)$, we find that
    \[
        \begin{split}
          e^*\Lna_{\det} &\isom (\prg2|_{\GAJ})^*\Lna_{\det} \tsr (\prg3|_{\GAJ})^*\Lna_{\det}^{-1} \\
                         &= (\cS,\na_\cS)|_{\GAJ} \tsr (\O,d-\prg1^*(\pr^{\tilde\rC^\circ}_{T^*C})\fr^* \th_{T^*C}\fr).
        \end{split}
    \]
    Combining this with \eqref{appx|eq:det on I} finishes the proof.
  \end{proof}

  As explained above, to prove Proposition~\ref{prop A} it is enough to prove~\eqref{tth id} on~$\GAJ$, which follows from Lemma~\ref{chdet} by computing extended curvature.
\end{proof}

\subsection{Construction of the Poincar\'e bundle}\label{subs:poincare}

Now, to finish the proof of Theorem~\ref{main_thm}, we need to provide a splitting $\cQ$ of the gerbe $\pr_1^*\gtx\cdot\pr_2^*\gty$ on $\tx\x_B\ty$ and a pair of isomorphisms as indicated in points (\ref{splitting Q})--(\ref{isoms aX aY}) of Proposition~\ref{prop:trsr-grb duality data}.  Corollary~\ref{cor DLc eqv GcaL} shows that $\gtx \sim \dd_{\rT_1, c\tilde\th}$ and $\gty \sim \dd_{\rT_c, c^{-1}\tilde\th}$, so we get
\[
    \pr_1^*\gtx\cdot\pr_2^*\gty \sim \dd_{\rT_1\xx_{(\bo)\fr}\rT_c, c\,\pr_1^*\tilde\th + c^{-1}\pr_2^*\tilde\th},
\]
hence our $\cQ$ must correspond to a line bundle $\Lna_{\ker}$ with connection on $\rT_1 \xx_{(\bo)\fr} \rT_{c}$ satisfying
\begin{equation}
%\begin{split}
    \label{need to check}
\tcurv(\Lna_{\ker}) %&= c\cdot \tcurv(\pr_1^*\Lna_{\det}) + c^{-1}\cdot \tcurv(\pr_2^*\Lna_{\det})\\    &
                    = c\,\pr_1^*\tilde\th + c^{-1}\pr_2^*\tilde\th.
%\end{split}
\end{equation}
%% this will give rise to a splitting~$\cP$ of the gerbe $\pr_1^*\dd_{c\tilde\th}\tsr\pr_2^*\dd_{c^{-1}\tilde\th}$ on $(\rT_1 \xx_{(\bo)\fr} \rT_{c})\fr$, and we'll also have to show that, locally over the base $(\bo)\frr$ after splitting $\dd_{c\tilde\th}$ and $\dd_{c^{-1}\tilde\th}$, $\cP$ can be identified with the Poincar\'e line bundle on the square of the Picard stack $\Pic((\tilde\rC^\circ)\frr/\rB\frr)\xx_{\rB\frr} \Pic((\tilde\rC^\circ)\frr/\rB\frr)$.

The sought-for bundle with connection will be given by
\begin{equation} \label{sfb}
    \Lna_{\on{ker}} := a^*(\Ldet',\na_{\det})\tsr \pr_1^*(\Ldet'^{\tsr-1},\na_{\det}^*)\tsr \pr_2^*(\Ldet'^{\tsr-1},\na_{\det}^*)
\end{equation}
where $(\Ldet',\na_{\det})$ is the universal line bundle with connection on $\Loc$, and $a$ is the ``addition'' map $$a\colon \rT_1 \xx_{(\bo)\fr} \rT_{c} \ \longto \ \rT_{1+c}$$ (one can check that the torsor $\rT_{1+c}$ is the sum of the torsors $\rT_{1}$ and $\rT_{c}$). %Define $\tilde\th = \tcurv\Lna_{\det}$. Then
Substituting \eqref{sfb} in~\eqref{need to check} yields
\begin{equation}
\begin{split}
    (\ref{need to check}.\rm{LHS})-(\ref{need to check}.\rm{RHS}) &= \tcurv(\Lna_{\ker}) - c\cdot \pr_1^*\tilde\th - c^{-1}\cdot \pr_2^*\tilde\th \\
    &= a^*\tilde\th - \pr_1^*\tilde\th - \pr_2^*\tilde\th  - c\cdot \pr_1^*\tilde\th - c^{-1}\cdot \pr_2^*\tilde\th \\
    &= a^*\tilde\th - (1+c)\cdot \pr_1^*\tilde\th - (1+c^{-1})\cdot \pr_2^*\tilde\th \,.
\end{split}
\end{equation}
So we need to show that
\begin{equation}\label{eq:hard tth id}
    (1+c)^{-1}a^*\tilde\th = \pr_1^*\tilde\th + c^{-1}\pr_2^*\tilde\th \,.
\end{equation}

To prove formula~\eqref{eq:hard tth id}, we proceed as follows. Let $\tilde\a = (1+c)^{-1}a^*\tilde\th - \pr_1^*\tilde\th - c^{-1}\pr_2^*\tilde\th$; we want to prove $\tilde\a=0$. Consider the two projections
\[
    \pr_{1,3},\pr_{2,3}\colon \rT_1\xx_{(\bo)\fr} \rT_1\xx_{(\bo)\fr} \rT_c\rightrightarrows \rT_1\xx_{(\bo)\fr} \rT_c.
\]
As a first step, we prove that the difference between two pullbacks $\pr_{1,3}^*\tilde\a - \pr_{2,3}^*\tilde\a = 0$. Let $\pr_i'\ (i=1,2,3)$ be the projection from $\rT_1\xx_{(\bo)\fr} \rT_1\xx_{(\bo)\fr} \rT_c$ to the $i$'th factor, and $a_{i,3}=a\circ\pr_{i,3}\ (i=1,2)$. We also have a ``difference'' map $s\colon \rT_1\xx_{(\bo)\fr}\rT_1 \to(\Higgs^\circ)\fr$; denote $s_{1,2}=s\circ\pr_{1,2}$. Now we calculate
\begin{equation}\label{pra-pra}
    \begin{split}
     \pr_{1,3}^*\tilde\a - \pr_{2,3}^*\tilde\a &= (1+c)^{-1}a_{1,3}^*\tilde\th - \pr_1'^*\tilde\th - c^{-1}\pr_3'^*\tilde\th \\
                                            &\quad-[(1+c)^{-1}a_{2,3}^*\tilde\th - \pr_2'^*\tilde\th - c^{-1}\pr_3'^*\tilde\th]\\
                                               &= (1+c)^{-1}(a_{1,3}^*\tilde\th - a_{2,3}^*\tilde\th) - (\pr_1'^*\tilde\th - \pr_2'^*\tilde\th) \\
                                               &= (1+c)^{-1}\delta\bigl(s_{1,2}^*((1+c)\th\fr)\bigr)  - \delta(s_{1,2}^*\th\fr) = 0 \,,
     \end{split}
\end{equation}
where in the last line we used formula~\eqref{tth id}.

The formula~\eqref{pra-pra} implies that $\tilde\a=\pr_2^*\phtr\tilde\a'$ for some $\phtr\tilde\a'\in\Gamma(\rT_c,\eF_{\rT_c})$. Similarly, we can show that $\tilde\a=\pr_1^*\phtr\tilde\a''$ for $\phtr\tilde\a''\in\Gamma(\rT_1,\eF_{\rT_1})$. It follows that $\tilde\a=(\chi^\circ \xx_{\bo} \chi^\circ)\fr^*\tilde\b$ for some $\tilde\b\in\Gamma(\bof,\eF_{\bof})$. So we need to show that $\tilde\b=0$, which we prove in \ref{subs:alt constr}.

\subsection{Construction of isomorphisms $\a_\rX,\a_\rY$}
We will now sketch the construction of the isomorphisms $\a_\rX,\a_\rY$ required by Proposition~\ref{prop:trsr-grb duality data} (modulo \eqref{eq:hard tth id}).
In other words, we need to check that $\cQ$ looks like a Poincar\'e line bundle.  The construction is based on the formula ~\eqref{sfb} which itself looks like the formula for the Poincar\'e bundle on $\Pic(\tilde\rC^\circ/\rB^\circ)$.  The argument in this subsection is a rather technical calculation and the reader is advised to skip it at the first reading.

Let us construct $\a_\rX$.  According to the formulation of Proposition~\ref{prop:trsr-grb duality data}, $\a_\rX$ should be an isomorphism of splittings of certain $\Gm$-gerbe on $\rX\x_B\tx\x_B\ty$.  The gerbe in
question is $$\gx\bx_B\gtx\bx_B\gty$$
 which corresponds to the generalized 1-form $$\pr_1^*\sP(\th\fr_{\Higgs})+\pr_2^*\tilde\th + c\,\pr_3^*\tilde\th,$$ and the splittings are given by line bundles with connection $$\pr_{1,3}^*\Lna_{\text{BB}} \tsr \pr_{2,3}^*\Lna_{\on{ker}},\quad \text{and}\quad (\O,d+c\,\pr_1^*\th_{\Higgs}\fr)\tsr(a_{\rT_1}\xx_{(\bo)\fr}\id_{\rT_c})^*\Lna_{\on{ker}}.$$
 Here $\Lna_{\on{BB}}$ is the line bundle with flat connection on $(\Higgs^\circ)\fr\xx_{(\bo)\fr}\Loc^\circ$ with $p$-curvature $\pr_1^*\th_{\Higgs}\frr$ which gives the equivalence $\dd_{(\Higgs^\circ)\fr,\th_{\Higgs}\frr}\hmod  \isoto  \on{QCoh}((\Loc^\circ)\fr)$ which is essentially the equivalence of~\cite{BB} for $C\fr$.

Let $\tilde\cK$ be the ``difference'' between our two splittings: it's a line bundle on $\rX\x_B\tx\x_B\ty$, and the data of $\a_\rX$ amounts a trivialization of~$\tilde\cK$.  The ``difference'' (or ``ratio'') of the corresponding line bundles with connection is then $(\cK,\na_\cK)$  on  $(\Higgs^\circ)\fr\xx_{(\bo)\fr}\rT_1\xx_{(\bo)\fr}\rT_c$ with $\cK=\Fr^*\tilde\cK$ and $\na_\cK$ being the canonical connection on Frobenius pullback.  Substituting formula~\eqref{sfb} for $\Lna_{\ker}$, we get the following formula for $(\cK,\na_\cK)$:
\begin{equation}\label{eq:l.bundle K}
  \begin{split}
    (\cK,\na_\cK) &= (a_{\rT_1}\xx_{(\bo)\fr}\id_{\rT_c})^* \Lna_{\on{ker}} \tsr \pr_{2,3}^*\Lna_{\on{ker}}^{\tsr-1} \tsr \pr_{1,3}^*\Lna_{\on{BB}} \tsr (\cO,d-\pr_1^*\th\fr)  \\
              &= (\id_{(\Higgs^\circ)\fr}\xx_{(\bo)\fr}a)^* (\cS,\na_\cS')_{1+c} \tsr \pr_{1,3}^*(\cS,\na_\cS')_c \tsr \pr_{1,3}^*\Lna_{\on{BB}} \tsr (\cO,d-\pr_1^*\th\fr)  \\
              &= (\id_{(\Higgs^\circ)\fr}\xx_{(\bo)\fr}a)^* (\cS,\na_\cS)_{1+c} \tsr \pr_{1,3}^*(\cS,\na_\cS)_c \tsr \pr_{1,3}^*\Lna_{\on{BB}} .
  \end{split}
\end{equation}
(Out of six $\Lna_{\det}$ factors that we get from two copies of $\Lna_{\ker}$, two cancel out and the other four group into two pullbacks of $(\cS,\na_\cS')$; the $d-\pr_1^*\th\fr$ comes from the $\mu^*$ in Proposition~\ref{prop:trsr-grb duality data}.)  Here $(\cS,\na_\cS)_{c'}$ is the pullback of $(\cS,\na_\cS)$ under the identification $(\Higgs^\circ)\fr \xx_{(\bo)\fr} \rT_{c'} \isoto (\Higgs^\circ)\fr \xx_{(\bo)\fr} \rT_1 \isoto (\Higgs^\circ)\fr \xx_{(\bo)\fr} \Loc^\circ \isoto \Gamma$ (the first of these identifications ``dilates'' the Hitchin base by~$c'$), and similarly for $\na_\cS'$.

First we are going to construct an isomorphism $\a_{\rX,\AJ}$ between pullbacks of these splittings (which is the same thing as trivialize the pullback of~$\tilde\cK$) under the Abel--Jacobi map $(\tilde\rC^\circ)\frr\x_B\tx\x_B\ty   \xra{\AJ\x\id\x\id}  \Pic((\tilde\rC^\circ)\frr/B) \x_B\tx\x_B\ty  = \rX\x_B\tx\x_B\ty$.  The ingredients of the construction of this trivialization are the isomorphism in Lemma~\ref{chdet} and formula~\eqref{tth id}.  To begin with, notice that $\Lna_{\on{univ},\GAJ}$ gives rise to the universal splitting of the gerbe $((\pr^{\tilde\rC^\circ}_{T^*C})^* \dd_C) \xx_B (\Loc^\circ)\fr$ on $(\tilde\rC^\circ)\frr \xx_B (\Loc^\circ)\fr$ which is the pullback of $\cP_{\ty}$ under $\AJ_{(\tilde\rC^\circ)\frr/B} \xx_B \id_{\ty}$.  This gives an isomorphism $(\AJ\x_{(\bo)\fr}\id)^*\Lna_{\on{BB}} \isoto \Lna_{\on{univ},\GAJ}$.  On the other hand, Lemma~\ref{chdet} allows to also rewrite in terms of $\Lna_{\on{univ},\GAJ}$ the pullback by $\AJ\x\id\x\id$ of the other two factors in the last line in~\eqref{eq:l.bundle K}.  Using these identifications, we get the following isomorphism:
\begin{multline*}
    (\AJ \xx_{(\bo)\fr} \id_{\rT_1} \xx_{(\bo)\fr} \id_{\rT_c})^*(\cK,\na_\cK)  \isoto\\
      \isoto (\id_{(\tilde\rC^\circ)\fr} \xx_{(\bo)\fr} a)^* \Lna_{\on{univ},\GAJ}  \tsr  (\pr_{1,3}^*\Lna_{\on{univ},\GAJ} \tsr \pr_{2,3}^*\Lna_{\on{univ},\GAJ})^{\tsr-1}.
\end{multline*}
But the construction of~$a\colon \rT_1\xx_{(\bo)\fr}\rT_c \to \rT_{1+c}$ (as tensoring over the spectral curve) and of $\Lna_{\on{univ},\GAJ}$ give a tautological trivialization of the triple product in the right-hand side.

Thus we have constructed the isomorphism $\a_{\rX,\AJ}$.  According to Corollary~\ref{cor:trsr-grb duality data on Im AJ} and Remark~\ref{rem:S2-symmetry automatic}, we get a canonical extension~$\a_\rX$ of~$\a_{\rX,\AJ}$ to $\rX\x_B\tx\x_B\ty$.  According to Remark~\ref{rem:isom aX enough}, this is enough to conclude that the functor constructed from~$\cQ$ is an equivalence.  Alternatively, one can construct the isomorphism~$\a_\rY$ in the same way, check the compatibility of $\a_{\rX,\AJ}$ and $\a_{\rY,\AJ}$ on $(\tilde\rC^\circ)\frr \x_B \tx \x_B (\tilde\rC^\circ)\frr \x_B \ty$, and apply Proposition~\ref{prop:trsr-grb duality data}.

%One can extend the  isomorphism $\a_{\rX,\AJ}$ to all of $\rX\x_B\tx\x_B\ty$ as follows.  By ``iterating''  $\a_{\rX,\AJ}$, we can construct a trivialization of the pullback of~$\cK$ under the map $((\tilde\rC^\circ)\frr)^{\x_B k}\x_B \tx\x_B\tx$.

\subsection{Alternative construction of $\tilde\th$}\label{subs:alt constr}
The goal of this subsection and the next  one is to prove formula~\eqref{eq:hard tth id}.  This will be done by lifting~$\tilde\th$ to a one-form, \ie constructing $\th_0\in\Om^1(\Loc^\circ)$ such that
\begin{equation}\label{eq:tth=delta(th0)}
    \tilde\th = \delta(\th_0)
\end{equation}
(where $\delta$ is defined in~\ref{extcrv}), and prove a formula similar to~\eqref{eq:hard tth id} for~$\th_0$.  In this subsection we will prove that the images of both sides of~\eqref{eq:tth=delta(th0)} in~$\Om^2(\Loc^\circ)$ under~$\sQ$ coincide.  The full identity~\eqref{eq:tth=delta(th0)} will be proved in the next subsection.

Let $\Oloc$ be the canonical $2$-form on the smooth part~$\Loc\sm$ of~$\Loc$, so that $\Oloc=F_{\na_{\det}} = \sQ(\tilde\th)$.  Since $\Loc\sm$ is a $\Gm$-gerbe over the corresponding coarse moduli space $\ul{\Loc\sm}$, the differential forms on $\Loc\sm$ are the same as differential forms on $\ul{\Loc\sm}$.  We want to construct a 1-form $\th_0$ on $\Loc\sm$ such that $\Oloc=d\th_0$. Since $\ul{\Loc\sm}$ is symplectic, this is equivalent to constructing a vector field $\xi_0$ on $\ul{\Loc\sm}$  which is Liouville, \ie $L_{\xi_0}\Oloc = \Oloc$ where $L$ denotes the Lie derivative.  We will actually construct $\xi_0$ as a vector field on $\Loc$: the one on $\ul{\Loc\sm}$ is then obtained by restricting to $\Loc\sm$ and then descending to $\ul{\Loc\sm}$ (again using the fact that differential forms descend uniquely from gerbes).

Here is the general plan of the  construction.     First of all, we know that $\Loc$ is  the moduli space of  certain coherent sheaves on the gerbe $\dd_C$ on $T^*C\fr$.  On $\Tcf$ there is an action of $\Gm$ by dilations and hence an action of  its Lie algebra.  We lift this action to the gerbe $\dd_C\sim\dd_{T^*C,\thcf}$ (where $\thcf$ is the canonical 1-form on $\Tcf$).  In fact, we construct an action of this Lie algebra on the gerbe $\dd_{c\thcf}$ for all $c\in\k$, compatibly with identifications $$\dd_{c\thcf}\cdot\dd_{c'\thcf}\isoto\dd_{(c+c')\thcf}.$$ Now the torsor $\rT_c$ is obtained from $\dd_{c\thcf}$  by taking $(\pr^{\tco}_{T^*C})\fr^*$ and taking torsors of splittings,  so we get  a vector field $\xi_{0,c}$ on $\rT_c$, and compatibility with the above identification (with $c'=1$) implies compatibility of $\xi_0:=\xi_{0,1}$ with the map~$a$ in~\eqref{eq:a}
(note though that by the identification $\rT_c\isoto\loc$ for $c\neq0$ we have $\xi_{0,c}=\xi_0$).

A description of $\Oloc$ in terms of Serre duality on $\dd_C$ implies that $\xi_0$ is a Liouville vector field for $\Oloc$, which allows to convert the above properties of $\xi_0$ to properties of the symplectic dual $\th_0$  to $\xi_0$  and in particular prove an analog of \eqref{eq:hard tth id} for $\th_0$
(see equation \eqref{a th0}).

Taking $c=0$  in \eqref{a th0}, we get an analog of \eqref{tth id} for $\th_{0,c}$, from which we conclude that $\tth-\delta(\th_0)$ descends to $\bof$.
We actually show that $\tth-\delta(\th_0)=\chi^{\circ*}\sP(\b_0)$
where $\b_0$ extends to all of $\rB\frr$ which allows to prove in \ref{subs:proof Jer} that $\b_0=0$ after a certain degree estimate.

\bigskip
First, we need another interpretation of the form $\Oloc$.  Note that for a $\Gm$-gerbe~$\g$ on a smooth variety~$X$, and two $\g$-modules $\cM$ and $\cN$ with proper support, we have the following version of Serre duality: $\RHom(\cM,\cN)^*\isom \RHom(\cN,\cM\tsr\omega_X)[\dim X]$ where $\omega_X$ is the sheaf of top degree differential forms on~$X$. In particular, if $X$ is a symplectic surface and $\cM=\cN$, we have a nondegenerate (in fact, antisymmetric) bilinear form on the space $\Ext^1(\cM,\cM)$ which is identified with the tangent space at~$\cM$ to the moduli space $\mxg$ of $\g$-modules on~$X$ with proper support, so we get a non-degenerate 2-form
 on this moduli space. One can use properties of the Serre duality to show that this form is closed, so it gives a symplectic structure on $\mxg$.

\begin{rem}
  Here by a \emph{symplectic structure}  on an (Artin) stack $\rY$ (say, over~$\k$) we mean a closed 2-form (in the sense of K\"ahler differentials) on $\rY$  which induces a symplectic structure on the tangent space (\ie $H^0$ of the tangent complex) to every $\k$-point of $\rY$.  (Note that this allows $\rY$ to be singular.) In fact, it is more conceptually correct to consider symplectic \emph{derived} stacks, so that we have a nondegenerate pairing on the  whole tangent complex (which would pair $H^{>0}(\cT_{\rY,y})$, which measures singularity/``derivedness'' of $\rY$, with $H^{<0}(\cT_{\rY,y})$, measuring ``stackiness'' of~$\rY$). With such a notion, the non-derived part of a derived symplectic stack would be a symplectic (usual) stack in the above sense.  Note also that a basic example of a \emph{smooth} symplectic stack  is provided by a $G$-gerbe over a symplectic variety $Y$ for a smooth group scheme $G$ over~$Y$, and this is essentially the most general example of a smooth symplectic stack (one should get the general case by taking $X$ to be a Deligne--Mumford stack).  The stacks $\mxg^\circ$ that we are interested in are of this form.
\end{rem}

We will also need the following statement, which is a  relative version of the above construction, and which is proved in a similar way:
\begin{prop}\label{prop:mxg symp}
   Let $S$ be a $\k$- (or $\Fp$\nbhp-) scheme, $\cL$ a line bundle on it, $\pi\colon X\to S$ a smooth morphism of relative dimension~$2$ with an $\cL$-valued symplectic form $\om\colon \Lambda^2\cT_{X/S}\isoto\pi^*\cL$ and suppose given a $\Gm$-gerbe $\g$ on~$X$.  Let $\mxg$ be the stack over~$S$ associating to a scheme $S'\to S$ the groupoid of $S'$-flat coherent sheaves $\cF\in\g_{S'}\hmod$ with $\supp\cF$ proper over $S'$ where $\g_{S'}=\g\x_SS'$.
   Then $\mxg$ has a canonical $2$-form
   $\iOm_{\mxg}\in\Om^2_{\mxg/S}
   \tsr\pi_{\mxg}^*\cL$ where  $\pi_{\mxg}\colon \mxg\to S$ is the structure map and $\Om^2_{\mxg/S}$ is the second exterior power of the sheaf of relative K\"ahler differentials. The form is compatible with pullbacks along $S$, and for $S=\Spec\k$, $\L=\O_S$ it coincides with the one coming from Serre duality, as described above.
\end{prop}
\begin{proof}
  The proof is  similar to the absolute case discussed above.  Namely, suppose given a map $S'\to S$ and an $S'$-point of $\mxg$ corresponding to $\cF\in\g_{S'}\hmod$.  Replacing $S,X,\g$ by $S',\ X_{S'}=X\x_SS',\ \g_{S'}=\g\x_SS'$, we can assume that $S'=S$. Now, given two families of tangent vectors $\xi,\xi'$ to $\mxg$ at $\cF$ corresponding to self-extensions $\wbar\cF,\wbar\cF'$ of~$\cF$, let $c,c'$ be their classes in $\Ext^1(\cM,\cM)$.  Then, by relative Serre duality, we have a perfect pairing $R\pi_* \RcHom(\cF,\cF) \tsr   R\pi_* \RcHom(\cF,\cF\tsr\om_{X/S}[2])\to\O_S$, which can be rewritten as $(R\pi_*\RcHom(\cF,\cF))^{\tsr2}\to\cL[-2]$. It induces a map $B_\cF\colon \Ext^1(\cF,\cF)^{\tsr2}\to H^0(S,\cL)$. We define the value of $\om_{\mxg}$ on $\xi,\xi'$ to be $B_\cF(c\tsr c')$.
  The symmetry property of the Serre duality implies that this form is anti-symmetric, so it defines a $2$-form on $\mxg$ which satisfies the conditions of the proposition.
\end{proof}
\begin{rem}
  It is actually true that the non-degenerate $2$-form $\iOm_{\mxg}$ is closed, so that $\mxg\to S$ is a relative symplectic stack, but we will not need this fact (more precisely, in all cases of interest for us, it is going to hold by other considerations), so we will not give a proof.
\end{rem}

\begin{lem}
The form $\Oloc$ coincides with $\iOm_{\mxg}$  from Proposition~\ref{prop:mxg symp} for $S=\Spec\k$, $X=T^*C\fr$, and $\g=\dd_C$ is the gerbe corresponding to the Azumaya algebra $\tilde\D_C$.
\end{lem}

\begin{proof}[Proof {\todo[sketch]}]
It is known that (in any characteristic) the category of $\O$-coherent $\D$-modules on a smooth variety $Z$ admits a Serre functor $\cS_Z$ which, moreover, is canonically isomorphic to the shift by $2\dim Z$. The same is true for $\D'$-modules where $\D'$ is any twisted differential operator algebra. We will need the case $\D'=\D_{C,\ohf}$. The lemma will follow from the following two statements:
\begin{itemize}
  \item The curvature of $\Lna_{\det}$ coincides with the 2-form constructed from this isomorphism $\cS_C\isom[2]$.
  \item The composite equivalence $$\D_{C,\ohf}\hmod \stackrel{\tsr\omega^{\tsr(1-p)/2}}\isoto\D_C\hmod \isoto \dd_C\hmod$$
       is compatible with the trivializations of Serre functors.
\end{itemize}
The first statement makes sense in any characteristic and is proved in Appendix~\ref{appx}, Proposition~\ref{appx|prop:curv Ldet}.  As for the second statement, the difference between two trivializations is an invertible function\footnote{The difference is an automorphism of the identity functor.  To see that it is given by multiplication by a function, it is better to consider families of $\O$-coherent $\D$-modules, \ie for a scheme~$S$ we consider the category of $\D_C\bx\O_S$-modules which are coherent as $\O_{C\x S}$-modules.} on $T^*C\fr$, therefore a constant.  %With a little more work,

 To show that this constant is equal to~$1$, it  suffices to compare the two symplectic forms at points of $\rM_{\Tcf,\dd_C}$ corresponding to $\dd_C$-modules of length~$1$ supported at points of $\Tcf$. Note that such modules are parametrized by an open and closed substack in $\rM_{\Tcf,\dd_C}$ isomorphic to the total space of the gerbe~$\dd_C$ and which corresponds to the stack denoted $\rI$ in \ref{appx|subs:charp} via the identification $\rM_{\Tcf,\dd_C}\isoto \Connhf$.  Now the restriction to~$\rI$ of the symplectic form constructed from the Serre functor on $\dd_C\hmod$ is clearly just the pullback to $\rI\isom\dd_C$ of the canonical symplectic form on $\Tcf$.  The fact that the Serre functor on $\D_{C,\ohf}$ gives the same form follows (for example) from the computation of the curvature of $\Lna_{\det}$ by using Propositions \ref{appx|prop:curv Ldet}~and~ \ref{appx|prop:det on I} (this has essentially been used in the proof of Proposition~\ref{prop A}).
 \todo[Maybe say this better]
\end{proof}

In the situation above, consider the open substack $\mxg^\circ \ctd \mxg$ consisting of modules which (locally after splitting $\g$) look like (pushforward of) a line bundle on a smooth curve in~$X$. Denote by $\bo_X$ the moduli space of proper smooth curves in~$X$, and let $\rC^\circ_X\ctd \bo_X\x X \to\bo_X$ be the universal family of curves. Then we have a map $\chi_{X,\g}\colon \mxg^\circ \to \bo_X$ given by taking supports, which presents $\mxg^\circ$ as a torsor for the relative Picard stack $\Pic(\rC^\circ_X/\bo_X)$.
\begin{lem}
The map $\chi_{X,\g}$ defines an integrable system, \ie the fibers are Lagrangian.
 \end{lem}

 \begin{proof}
   Indeed, the tangent space to $\mxg^\circ$ at a point corresponding to $\cF\in\g\hmod$ is given by $H^1(\RGam(X,\RcEnd\cF))$, and the Serre duality pairing on this vector space comes from a (quasi\nbhp-)iso\-mor\-phism  of the complex $(\RcEnd\cF)[1]$ with its own Serre dual.  Moreover, the relative tangent space for $\mxg^\circ\to\bo_X$ is the middle term of the 2-step filtration coming from the canonical filtration on the complex of sheaves $\RcEnd\cF$.
   Now, since $\cF\in \mxg^\circ$, it is easy to see that the cohomology sheaves of the complex $(\RcEnd\cF)[1]$ are line bundles on $\supp\cF$ (which is a smooth projective curve in~$X$); in particular, they are Cohen--Macaulay of dimension~$1$.  Therefore the 2-step canonical filtration on $(\RcEnd\cF)[1]$ is preserved (up to a shift) by the Grothendieck duality functor, hence we get the desired Lagrangian property of the fibration $\mxg^\circ\to\bo_X$.
 \end{proof}

\ncmd{\Du}{\rD}
Denote 
$$\Du=\Spec(\k[\eps]/\eps^2).$$
A vector field on $\Loc\sm$ is the same as an automorphism of $\Loc\sm\x\Du$ over~$\Du$ which is identity on $\Loc\sm \ctd \Loc\sm\x\Du$. Let $h$ be the automorphism of $T^*C\fr\x\Du$ corresponding to the Euler vector field on $T^*C\fr$.  Since $H^2(T^*C\fr,\O)=0$, there exists a (non-unique) equivalence
\[
   \vphi\colon \pr_1^*\dd_C \isoto h^*\pr_1^*\dd_C,
\]
where $\pr_1$ is the projection $\Tcf\x\Du\to\Tcf$,
which is identity on $T^*C\fr\x\pt \subset T^*C\fr\x\Du$. Now, if we have a $\D$-module $\cM$ on $C$, let $\cM'$ be the corresponding $\dd_C$-module, and let $\overline{\cM'} =\vphi^{-1}h^*(\cM'\bx\O_\Du)$~-- this is a $\pr_1^*\dd_C$-module on $T^*C\fr\x\Du$, and denote by $\overline \cM$ the corresponding $\D_C\bx\O_\Du$-module. By construction, $\overline \cM/\eps\overline \cM\isom \cM$, so $\overline \cM$ defines a tangent vector to $\Loc\sm$ at~$\cM$. This way we get a vector field on $\Loc\sm$ which is the desired field $\xi_0$. Denote $\th_0=\iota_{\xi_0}\Oloc$.

\begin{prop}\label{prop:Liouville}
The vector field $\xi_0$ is Liouville. Equivalently, $d\th_0=\Oloc$.
\end{prop}

\begin{proof}
The proposition follows from the canonicity statement in Proposition~\ref{prop:mxg symp}. Namely, we take the trivial family of symplectic surfaces $X=\Tcf\x\Du\to S=\Du$ with gerbe $\g=\dd_C\x\Du$ over it and $\L$ being the trivial line bundle $\L=\O_\Du$.  The resulting symplectic stack $\mxg$ over~$\Du$ is isomorphic to $\Connhf\x\Du$.  Now we take the automorphism of the situation which acts on $\Tcf\x\Du$  by $h$, on $\dd_C\x\Du$ by $\vphi$, and on $\L$ by $1+\eps$ (note that the latter is forced by the compatibility with the $\L$-valued symplectic form on $X$, since  $h^*\iOm_{T^*C\fr} = (1+\eps)\iOm_{T^*C\fr}$).
 Denoting by $\tilde h$ the restriction to $\Loc\sm\x\Du$ of the induced automorphism  of $\mxg=\Connhf\x\Du$, we see that canonicity of $\iOm_{\mxg}$ implies that $\tilde h^*\iOm_{\Loc} = (1+\eps)\iOm_{\Loc}$ which means (since $\tilde h= 1+\eps\xi_0$ by definition) that $L_{\xi_0}\Oloc= \Oloc$.
\end{proof}
%%
%%\begin{prop}
%%The class of $\xi_0$ modulo Hamiltonian vector fields does not depend on the choice of~$\vphi$.
%%\end{prop}
%%
%%\begin{proof}[Proof sketch]
%%Suppose we have two equivalences $\vphi_1,\vphi_2\colon \pr_1^*\dd_C \isoto h^*\pr_1^*\dd_C$. Then they differ by an auto-equivalence $\vphi_1^{-1}\circ\vphi_2$ of $\pr_1^*\dd_C$ which corresponds to an element $u\in H^1(T^*C\fr,\O)$. From any such element we can construct a function $f'_u$ on $\rB\fr$ as follows. If a point $b\in\rB\fr$ corresponds to a smooth spectral curve $\tilde C\ctd T^*C\fr$ (\ie if $b\in(\bo)\fr$) then we just put $f'_u(b)= \angs{u|_{\tilde C}, \thcf|_{\tilde C}}$ where $\angs{\cdot,\cdot}$ denotes the Serre duality pairing. (It can also be defined for $b\not\in(\bo)\fr$.) One can check that the pullback $f_u$ of $f'_u$ to $\Loc\sm$ satisfies $H_{f_u}=\xi_{0,1}-\xi_{0,2}$ where $\xi_{0,1}$ and $\xi_{0,2}$ are the $\xi_0$'s corresponding to $\vphi_1$ and $\vphi_2$, respectively.
%%\end{proof}

We will now describe a class of $\vphi$'s as above, for which the corresponding anti-derivative $\th_0$ of $\Oloc$ represents the extended curvature of~$\na_{\det}$.

Recall (%see, e.g. \todo[{\cite[??]{CZ-GL}}],
this is essentially equivalent to additivity property of $\dd_\a$ in~$\a$) that for any smooth variety~$X$, the gerbe $\dd_C$ on $\Tcf$ has a natural multiplicative structure with respect to the relative group scheme structure on $\Tcf\to C\fr$.  It is easy to deduce from this that
%\ref{subs:c.red}) that for a smooth variety~$X$ and a $1$-form $\th$ on $X\fr$ we denote by $\dd_\th$ the $\Gm$-gerbe on~$X\fr$ corresponding to the Azumaya algebra $\D_\th$. It is straightforward to generalize this to the case of families $X\to S$ for arbitrary $\k$-scheme $S$.  We will need the case when $S=\Du$.  We know that the gerbe~$\dd_C$ on $T^*C\fr$ is equivalent to $\dd_\thcf$ where $\thcf$ is the canonical 1-form on $T^*C\fr$. One easily checks that $\dd_\th$ depends additively on~$\th$ in the sense that $\dd_{\th_1+\th_2} \sim \dd_{\th_1}\cdot\dd_{\th_2}$. Therefore
we have the following isomorphisms of gerbes on $T^*C\fr\x\Du$:
\[
   \pr_1^*\dd_C^{-1} \cdot h^*\pr_1^*\dd_C \sim %\dd_{h^*\thcf-\thcf} = \dd_{\eps\thcf}
   %=\dd_{e^*\thcf} \sim
   e^*\pr_1^*\dd_C
\]
where %$\thcf$ here is considered as a relative 1-form on $T^*C\fr\x\Du$ over~$\Du$, and
$e\colon T^*C\fr\x\Du \to T^*C\fr\x\Du$ is given by fiberwise multiplication by~$\eps$. Now, $e$ factors through the 1st infinitesimal neighborhood $Z_1$ of the zero section in $T^*C\fr$. Therefore the above equivalence~$\vphi$ can be constructed from any trivialization $\Psi$ of~$\dd_C$ on $Z_1$ which coincides with the canonical trivialization on the zero section. We assume from now on that $\vphi$ is obtained in this way.

Recall that we need to prove formula ~\eqref{eq:hard tth id} involving certain generalized $1$-form on ~$\loc$.  We will deduce it from the following two propositions:
\begin{prop}\label{prop:a th0}
If the equivalence $\vphi$ is constructed from a splitting $\Psi$ as described above then, in the notation of~\eqref{eq:hard tth id} (for any $c\in\k\setm\Fp$), we have
 \begin{equation}\label{a th0}
    (1+c)^{-1}a^*\th_0 = \pr_1^*\th_0 + c^{-1}\pr_2^*\th_0 \,.
 \end{equation}
\end{prop}
\begin{prop}\label{prop Jer}
The extended curvature $\tilde\th$ of the determinant line bundle with connection $\Lna_{\det}$ on $\Loc\sm$ is equal to $\delta(\th_0)$.
\end{prop}

Since $\delta$ is $\k$-linear and compatible with pullbacks, Propositions \ref{prop:a th0}~and~\ref{prop Jer} imply formula~\eqref{eq:hard tth id} and therefore Theorem~\ref{main_thm}.
The rest of the subsection is devoted to the proof of this proposition.
Proposition~\ref{prop Jer} will be proved in subsection~\ref{subs:proof Jer}.

We begin by observing that, given $\Psi$, one can construct a family of equivalences
\begin{equation}\label{Phi_c}
    \vphi_c\colon \pr_1^*\dd_{c\thcf} \isoto h^*\pr_1^*\dd_{c\thcf}
\end{equation}
parametrized by $c\in\k$ (we just have to replace $e$ by fiberwise dilation by $c\eps$). One can check that $\vphi_c$ is additive in $c$, \ie compatible with the equivalence $\dd_{(c+c')\thcf}\isoto \dd_{c\thcf}\cdot \dd_{c'\thcf}$ and, if $c\in\k^\x$, $\vphi_c$ is obtained from $\vphi$ by conjugation with the isomorphism between $\dd_C=\dd_{\thcf}$ and $\dd_{c\thcf}$ lifting the dilation by~$c$ on $T^*C\fr$.

Note that similarly to the way we defined $\xi_0$, we can define a vector field $\xi_{0,c}$ on $\rT_c$ using the identification $\rT_c\isoto\rM_{\Tcf,\dd_{c\thcf}}$ and the infinitesimal automorphism (\ie $\Du$-family of automorphisms) of the pair $(\Tcf,\dd_{c\thcf})$ given by $(h,\vphi_c)$, and the above compatibility implies the following:
\begin{lem}\label{lem:a xi0}
  The maps $a_{c,c'}\colon \rT_c\xx_{\bof}\rT_{c'}\to\rT_{c+c'}$ are compatible with the vector fields $\xi_{0,c},\xi_{0,c'},\xi_{0,c+c'}$ in the sense that the ``diagonal'' vector field on $\rT_c\x\rT_{c'}\x\rT_{c+c'}$ constructed from these three vector fields is tangent to the graph of $a_{c,c'}$.
\end{lem}

%\begin{proof}[Proof {\todo[sketch]}]
We will deduce Proposition ~\ref{prop:a th0}  from the following identity:
\begin{equation}\label{c OLoc}
(1+c)^{-1}a^*\Oloc = \pr_1^*\Oloc + c^{-1}\pr_2^*\Oloc\,,
\end{equation}
and Lemma~\ref{lem:a xi0}.
Formula~\eqref{c OLoc} will follow from a more general statement:

\begin{lem}\label{lem:lagr graph general}
  Let $\pi\colon A\to B$ be a family of abelian varieties  with a symplectic  form $\iOm_A$ on $A$, such that the fibers are Lagrangian subvarieties, \ie the morphism $\pi$ defines an integrable system.  Let $\Gamma_A\ctd A\x_BA\x_BA$ be the graph of the group structure morphism $m\colon A\x_BA\to A$ (where the target of $m$ is identified with the third factor of $A\x_BA\x_BA$ so that $\Gamma_A$ consists of points of the form $(x,y,m(x,y))$). Then $\Gamma_A$ is Lagrangian in $(A\x A\x A,\iOm_A\bx1\bx1+1\bx\iOm_A\bx1-1\bx1\bx\iOm_A)$ if and only if the zero section of $A$ is Lagrangian.
\end{lem}

\begin{proof}
  We are only going to prove the ``if'' part, since this is the harder part  and the only part we will use.
  Also note that since $\dim A=2\dim B$ and $\pi$ is smooth, it is an easy dimension calculation that both the zero section in~$A$ and $\Gamma_A$ in $A\x A\x A$ have half the dimension of the ambient scheme, so for them being Lagrangian is equivalent to being isotropic.
  
  Denote
  $$\iOm_{\Gamma_A} =
  (\iOm_A\bx1\bx1+1\bx\iOm_A\bx1-1\bx1\bx\iOm_A)
  |_{\Gamma_A}.$$ 
  Thus we need to prove that $\iOm_{\Gamma_A}=0$.  Let $s\colon B\to A$ be the zero section and consider the subschemes in $A\x_BA\isoto\Gamma_A$ which are images of the maps $i_1,i_2\colon A\to  A\x_BA$  given by $i_1=\id_A\x_Bs,\ i_2 =s\x_B\id_A$.  Using the fact that the zero section is Lagrangian, \ie $s^*\iOm_A=0$, it is easy to see that
  \begin{equation}\label{eq:i1*=i2*=0}
    i_1^*\iOga=i_2^*\iOga=0.
  \end{equation}

  Now consider the filtration $F^\blt\Oga^2$ on $\Om^2_{\Gamma_A}$  with successive quotients given by $\gr_F^0\Oga^2=\pi'^*\Om^2_B$, $\gr_F^1\Oga^2=\pi'^*\Om^1_B\tsr\Om^1_{\Gamma_A/B}$, $\gr_F^2\Oga^2=\Om^2_{\Gamma_A/B}$, where $\pi'$ is the projection $\Gamma_A\to B$, and the analogous filtration $F^\blt\Om^2_A$ on $\Om^2_A$.  Since the restriction of $\iOm_A$ to the fibers of $\pi$ is $0$, the form $\Oga$ must lie in $F^1\Oga^2$.  But since the relative cotangent bundle $\Om_{A/B}$ descends to $B$ it is easy to see by projection formula that the natural morphisms
  $$\pi'_*\gr_F^i\Oga^2\xra{(i_1^*,i_2^*)}\pi_*\gr_F^i\Om_A^2\oplus \pi_*\gr_F^i\Om_A^2$$
  given by pullbacks by $i_1,i_2$ (\ie coming from the adjunction morphisms $\id\to i_{1*}i_1^*$, $\id\to i_{2*}i_2^*$) are injective for $i=0,1$. From this we see that \eqref{eq:i1*=i2*=0} implies $\iOga=0$ as desired.
\end{proof}

\begin{lem}\label{Lagr graph}
Let $(X,\iOm)$ be a symplectic surface, $\g,\g',\g''$ three $\Gm$-gerbes on it, and suppose we are given an equivalence $\g \isoto \g'\cdot\g''$. Consider the corresponding $\Pic(\rC^\circ_X/\rB_X^\circ)$-torsors $\rT =\mxg^\circ$, $\rT' =\rM_{X,\g'}^\circ$, $\rT''=\rM_{X,\g''}^\circ$ endowed with symplectic structures %defined using Serre duality on~$S$.
$\iOm_\rT,\iOm_{\rT'},\iOm_{\rT''}$
as in Proposition ~\ref{prop:mxg symp}.
Clearly, $\rT$ is identified with the Baer sum of torsors $\rT'$ and $\rT''$, so we can define the graph of addition $\Gamma\ctd\rT \xx_{\rB_X^\circ} \rT' \xx_{\rB_X^\circ} \rT''$.  Assume also that the zero section of $\Pic(\rC^\circ/\bo)$ is Lagrangian with respect to the Serre duality symplectic form on $\Pic(\rC^\circ/\bo)$.  Then $\Gamma$ is %a Lagrangian subvariety in 
isotropic in
$(\rT\x\rT'\x\rT'',\iOm_{\rT}\bx1\bx1 -1\bx\iOm_{\rT'}\bx1 - 1\bx1\bx\iOm_{\rT''})$.
\end{lem}

\begin{proof}
Suppose we want to prove that $\Gamma$ is % Lagrangian in a formal neighborhood of 
isotropic at some point $\gamma\in\Gamma$, and let $\tilde C\ctd X$ be the spectral curve corresponding to the image of $\gamma$ under the projection $\Gamma\to\bo_X$. Clearly one can replace $X$ by the formal neighborhood ${\tilde C}^\wedge$ of~$\tilde C$ in~$X$. The point $\gamma$ gives splittings of $\g$~and~$\g'$ on $\tilde C$. Moreover, we can extend these splittings to ${\tilde C}^\wedge$ since the obstruction to extending a splitting from $k$'th to $(k+1)$st infinitesimal neighborhood  lies in $H^2(\tilde C,(\cN_{X/\tilde C}^*)^{\tsr k+1})=0$ (here $\cN_{X/\tilde C}^*$ is the conormal bundle to $\tilde C$ inside $X$).

Thus we can assume that $\g'$ and $\g''$ are trivial gerbes.  Then $\Gamma$ is identified with the graph of addition on $\Pic(\rC^\circ_X/\rB_X^\circ)$. Since $\gamma$ corresponds to a point in the zero section in ~$\Gamma$, it is sufficient to prove that the graph of addition is isotropic for $\Pic^0(\rC^\circ_X/\bo_X)$, which in turn reduces to the family of Jacobians $\ul\Pic^0(\rC^\circ_X/\bo_X)$ (over which $\Pic^0(\rC^\circ_X/\bo_X)$ is a $\Gm$-gerbe). Since $\ul\Pic^0(\rC^\circ_X/\bo_X)$ is an abelian scheme over $\bo_X$, the fact that this graph is isotropic (actually, Lagrangian) follows from the condition on the zero section in $\Pic(\rC^\circ_X/\bo_X)$ by Lemma ~\ref{lem:lagr graph general}.
\end{proof}

\begin{cor}\label{cor:a th c c'}
For $c\in\k$, let $\rT_c\to\bof$ be the torsor of fiberwise splittings of the gerbe $(\pr^{\tco}_{\Tcf})\fr^*\dd_{\Tcf,c\thcf}$ along the fibers of the projection $(\tco)\fr\to\bof$. (For  $c\neq0$ this torsor is identified with the one defined in~\S\ref{subs:outline}.)  We have an addition map
$$a_{c,c'}\colon \rT_c\x_{\bof}\rT_{c'}\to\rT_{c+c'}$$
for any $c,c'\in\k$ and consider the projections $\pr_1\colon \rT_c\x_{\bof}\rT_{c'}\to\rT_{c}$ and $\pr_2\colon \rT_c\x_{\bof}\rT_{c'}\to\rT_{c'}$.  Also recall the symplectic form $\iOm_{\rT_c}$ on $\rT_c$, the vector field $\xi_{0,c}$ on $\rT_c$ constructed from a splitting $\Psi$ as explained before, and the dual $1$-form $\th_{0,c}=\iota_{\xi_{0,c}}\iOm_{\rT_c}$.  Then we have
\begin{equation}\label{eq:a th c c'}
   a_{c,c'}^*\th_{0,c+c'} = \pr_1^*\th_{0,c} + \pr_2^*\th_{0,c'}.
\end{equation}
\end{cor}

\begin{proof}
We first apply Lemma~\ref{Lagr graph} to $X=\Tcf,\ \g'=\dd_{c\thcf},\ \g''=\dd_{c'\thcf}$. In order to do it, we first have to check the Lagrangian property of the zero section in $\rM_{\Tcf,B\Gm\x\Tcf}
%=\Pic(\rC^\circ/\bo)\fr\isoto (\hig)\fr
$ where $B\Gm\x\Tcf$ is considered as the trivial $\Gm$-gerbe on $\Tcf$.  Clearly we can remove the Frobenius twists and reduce to proving that the zero section in $\rM_{T^*C,B\Gm\x T^*C}$ is Lagrangian.  Now, under the identification $\rM_{T^*C,B\Gm\x T^*C}=\Pic(\rC^\circ/\bo) \isoto \Higgs^\circ$ the Serre duality symplectic form corresponds to the canonical symplectic form on the cotangent bundle (one can prove this by a simpler version of the argument in the proof of Proposition~\ref{appx|prop:curv Ldet}) and the zero section in $\Pic(\rC^\circ/\bo)$ corresponds to (an open subset of) the fiber of $\Higgs\to\Bun$ over the point associated with the bundle $\dsum_{i=0}^{N-1}\cT_C^{\tsr i}$, and hence is Lagrangian.

Recall that by definition we have $\rt_c=\rmo_{\Tcf,c\thcf}$ and the addition map referred to in Lemma~\ref{Lagr graph} in our case is exactly $a_{c,c'}$.
The lemma thus shows that
the graph $\Gamma_{c,c'}$ of the map $a_{c,c'}$ is Lagrangian inside the symplectic stack $(\rt_c\x\rt_{c'}\x\rt_{c+c'}, \ \iOm_{c,c'})$ where $$\iOm_{c,c'}=\iOm_{\rt_c}\bx1\bx1 + 1\bx\iOm_{\rt_c}\bx1 - 1\bx1\bx\iOm_{\rt_{c+c'}}.$$
On the other hand, we know by Lemma~\ref{lem:a xi0} that the vector field
$$\eta_{c,c'}:= \xi_{0,c}\bx1\bx1 +1\bx\xi_{0,c'}\bx1+ 1\bx1\bx\xi_{0,c+c'}$$
is tangent to $\Gamma_{c,c'}\ctd  \rt_c\x\rt_{c'}\x\rt_{c+c'}$.  These two statements together imply that the $1$-form
$$\iota_{\eta_{c,c'}}\iOm_{c,c'}= \th_{0,c}\bx1\bx1 + 1\bx\th_{0,c'}\bx1 - 1\bx1\bx\th_{0,c+c'}$$
 vanishes on $\Gamma_{c,c'}$, which is equivalent to formula~\eqref{eq:a th c c'}.
%$$a_{c,c'}^*\iOm_{\rT_{c+c'}} = \pr_1^*\iOm_{\rT_c} + \pr_2^*\iOm_{\rT_{c'}}.$$
%Note that for $X=T^*C\fr,\ \g=\dd_{c\thcf}$ the torsor $\mxg^\circ$ is identified with $\rT_c$ from \S\ref{subs:outline}. Moreover, under this identification, we have $c^{-1}\Oloc=\iOm_{\mxg^\circ}$. Therefore, formula~\eqref{c OLoc} follows from Lemma~\ref{Lagr graph} applied to $\g=\dd_{(1+c)\thcf}$, $\g'=\dd_\thcf$, $\g''=\dd_{c\thcf }$.  The only thing
%
%Now, we deduce formula \eqref{eq:a th c c'} from the above identity and Lemma ~\ref{lem:a xi0}.
%
%Denote by $\xi_{0,c}$ the vector field on $\rT_c$ obtained from $\xi_0$ under the identification $\rT_c\isom\Loc^\circ$. Then, in order to prove Proposition~\ref{a th0}, we need to prove that the vector field $\eta_c= \pr_1^*\xi_{0,1+c} + \pr_2^*\xi_{0,1} + \pr_3^*\xi_{0,c}$ on $\rT_{1+c}\x\rT_1\x\rT_c$ preserves $\Gamma$ (because $\Gamma$ is Lagrangian, and this vector field corresponds to the 1-form $(1+c)^{-1}\pr_1^*\th_0-\pr_2^*\th_0-c^{-1}\pr_3^*\th_0$). To see this, note that $\xi_{0,c}$ can be obtained the same way as $\xi_0$ with $\vphi$ replaced by $\vphi_c$ from~\eqref{Phi_c}. Now, $\eta_c$ comes from an infinitesimal automorphism of the quadruple $(X,\dd_\thcf,\dd_{c\thcf},\dd_{(1+c)\thcf})$ given by $(h,\vphi_1,\vphi_c,\vphi_{1+c})$. Additivity of $\vphi_c$ in $c$ implies that this automorphism is compatible with the equivalence $\dd_\thcf\cdot\dd_{c\thcf}\isoto \dd_{(1+c)\thcf}$. Therefore $\eta_c$ preserves $\Gamma$, which is what we want.
\end{proof}
%\end{proof}

Now it is easy to see that Corollary~\ref{cor:a th c c'} implies  Proposition ~\ref{prop:a th0}.  Indeed, by definition of~$\rt_c$, for $c\in\k^\x$ there is an isomorphism $\mu_c\colon \rt_c\isoto\rt_1=\loc$ lifting the map $[c]\colon \bof\to\bof$ coming from the $\Gm$-action on $\bof$ and we have $\mu_c^*\Oloc=c\iOm_{\rt_c}$.
(To see this, one should use the definition of $\rt_c$ as $\rmo_{\Tcf,\dd_{c\thcf}}$ and apply Proposition ~\ref{prop:mxg symp} to the isomorphism $(\Tcf,\dd_{c\thcf}) \isoto(\Tcf,\dd_\thcf)$, similarly to the proof of Proposition ~\ref{prop:Liouville}.)  We also have $\mu_c^*\xi_0=\xi_{0,c}$ (here the pullback of a vector field is well-defined since $\mu_c$ is an isomorphism), and hence $\mu_c^*\th_0=\th_{0,c}$. From this we see that substituting $(1,c)$ for $(c,c')$ in \eqref{eq:a th c c'} gives formula~\eqref{a th0}.

\subsection{Proof of Proposition~\ref{prop Jer}}\label{subs:proof Jer}
We want to prove the equality of two sections of $\eF_{\loc}$: $\tth=\delta(\th_0)$.  We begin by showing that the difference
$$\tilde\a_0=\tilde\th-\delta(\th_0)$$
descends to $\bof$ and, moreover, extends to $\rB\fr$:
%Here we prove the following partial result:

\begin{lem}\label{1-form on B}
We have $\tilde\th-\delta(\th_0)= \sP(\chi\pfr^*\b_0')$ where $\chi'$ is the map $\Loc\sm\to\rB\fr$, $\delta$ and $\sP$ are defined in~\ref{extcrv}, and $\b_0'$ is some form on $\rB\frr$.
\end{lem}

\begin{proof}%[Proof {\todo[sketch]}]
We already know that
$\sQ(\delta(\th_0))=d\th_0=\Oloc=\sQ(\tilde\th)$. So $\sQ(\tilde\a_0)=0$, which means that $\tilde\a_0=\sP(\a_0)$ for some $\a_0\in\Gamma((\Loc\sm)\fr,\Om^1_{(\Loc\sm)\fr})$. We want to prove that $\a_0=\chi\pfr^*\b_0'$ for some $\b_0'$. We'll show that $\a_0|_{\Loc^\circ}$ is a pullback of some $1$-form $\b_0$ on $\bof$. Using the properties of $\chi'$, we can then deduce that $\b_0$ extends to the whole $\rB\fr$ since $\chi\pfr^*\b_0$ extends to $\Loc\sm$.

Now let $\Gamma$ be the graph of addition in $(\Higgs^\circ)\fr \xx_{\rB\fr}\Loc^\circ \xx_{\rB\fr}\Loc^\circ$. The argument of the proof of Proposition~\ref{a th0} applied to $c=0$ shows that on $\Gamma$ we have $\pr_1^*\th\fr+\pr_2^*\th_0=\pr_3^*\th_0$ where $\th$ is the canonical $1$-form on $\Higgs=T^*\Bun$. (We use that the vector field $\xi_{0,0}$ on $\Higgs$ coincides with the Euler vector field, and therefore $\th_{0,0}=\th$.) Applying $\delta$ yields $\delta(\pr_1^*\th\fr)+\delta(\pr_2^*\th_0)=\delta(\pr_3^*\th_0)$. On the other hand, we know from~\eqref{tth id} that $\delta(\pr_1^*\th\fr) + \pr_2^*\tilde\th = \pr_3^*\tilde\th$, so subtracting the previous equation, we get $\pr_2\fr^*\tilde\a_0=\pr_3\fr^*\tilde\a_0$, so $\tilde\a_0|_{\Loc^\circ}$ is a pullback of some $\tilde\b_0'\in\Gamma(\bof,\eF_{\bof})$. Since $\tilde\a_0\in\Im(\sP)$, we must have $\tilde\b_0'\in\Im(\sP)$, so $\tilde\b_0'=\sP(\b_0)$ for some $\b_0$.

We conclude the proof by noticing that  the restriction of the map $\chi$ to the maximal open substack $\Locchism$ such that $\chi|_{\Locchism}$ is smooth is surjective, which means that since $(\chi^\circ)^*\b_0$ extends from $\loc$ to $\Loc\sm$, the form $\b_0$ must extend from $\bof$ to $\rB\fr$.
\end{proof}

Now we will investigate the behavior of $\tth$~and~$\th_0$ with respect to the projection $\Loc\to\Bun$.

 Denote by $\Bun^{[d]}$ the connected component of~$\Bun$ consisting of vector bundles of degree~$d$ and $\Loc^{[d]}$ its preimage in~$\Loc$.  Note that $\Loc^{[d]}$ is nonempty only for $d$~divisible by~$p$.

We will deduce Proposition~\ref{prop Jer} from Lemma ~\ref{1-form on B} and the following two statements:

\begin{lem}\label{lem transv}
  For sufficiently large~$p$ (relative to $N$ and the genus of~$C$) and generic curve~$C$ the fibers of the maps $q\colon \Loc \to \Bun$ and $\chi\colon \Loc \to \rB\fr$ are transversal generically on $\Loc^{[0]}$.
\end{lem}
\begin{lem}\label{lem:vanishing on fibers}
  The 1-form $\b_0 = \chi\fr^*\b_0'$ (where $\b_0'$ is defined in Lemma~\ref{1-form on B}) vanishes on $q^{-1}(b)$ if $b\in\Bun$ is such that $q$ is smooth over~$b$.
\end{lem}
Clearly, together with Lemma~ \ref{1-form on B}, these two statements imply the desired equality $\b_0'=0$.

\begin{proof}[Proof of Lemma~\ref{lem transv}] %{\todo[sketch]}]
First let us note that the map $(q,\chi)\colon \Loc\to \Bun\x\rB\fr$ can be degenerated into the map $(V\circ\pi\fr,h\fr)\colon \Higgs\fr\to \Bun\xx\rB\fr$ where $h$ is the standard Hitchin map, $\pi\colon \Higgs\to\Bun$ is the natural projection, and $V\colon \Bun\fr\to\Bun$ is the map given by pullback of bundles by $\Fr_C$.  The degeneration is realized by the family $\{\Loc'_\la\}$ where $\Loc'_\la$ is the stack parametrizing triples $(\cE,\na,K)$ where $(\cE,\na)\in\Loc$ and $K\colon \cE\to\cE\tsr\Om_C\fr[C]$ is a map satisfying $\la K= \curv_p(\na)$.  (This is equivalent to considering $\D_{T^*C,\la\thcf}$-modules of the appropriate kind.)  For $\la\ne0$ it identifies with $\Loc$, and for $\la=0$ we have $\Loc'_\la\isom\Higgs\fr$.

Now it is enough to prove that $V\circ\pi\fr$ and $h\fr$ are generically submersions (\ie smooth morphisms) and their fibers are transversal at a generic point. This would follow from two statements: namely, (a) $\pi\fr$ and $h\fr$ (or equivalently $\pi$ and $h$) are generically submersions and have generically transversal fibers, and (b) the differential of~$V$ is generically an isomorphism (\ie $V$ is generically \'etale).  The first one is true in characteristic~$0$ (because otherwise these fibers generically would have to have a curve in their intersection, which would contradict the fact that one of them is affine, and the other is proper).  Therefore it must be true for sufficiently large~$p$ and generic~$C$.

For the second one, let us notice that the differential of~$V$ at a point of~$\Bun$ corresponding to a vector bundle~$\cE'$ on~$C\fr$ is the map $\REnd(\cE') [1] \to \REnd(\Fr_C^*\cE') [1]$ given by $\Fr_C^*$.  Since the condition that this map be an isomorphism is open (in $\cE'$ and~$C$), and $\Bun^{[0]}$ is connected, it is enough to check it for one point of $\Bun^{[0]}$.  Take the point corresponding to the trivial vector bundle $\cE'= \O_{C\fr}^N$: then the differential of~$V$ at that point is given by the direct sum of $N$ copies of the map $\Fr_C^*\colon \RGam(\O_{C\fr}) [1] \to \RGam(\O_C) [1]$, which is an isomorphism iff  $C$ is ordinary (not supersingular).
%
%Let $\cE'$ be a rank~$N$ vector bundle on~$C\fr$.  Consider the vector bundle $\cE := \Fr_C^*\cE'$ equipped with the canonical connection~$\na$ as on Frobenius pullback.  Also let $e$ be the point of~$\Loc$ corresponding to $(\cE,\na)$.  Then we have $e \in \chi^{-1}(0)$ where $0\in\rB\fr$ is the fixed point of the canonical $\Gm$-action on $\rB\fr$.  Suppose, moreover, that the point $b:= q(e)\in \Bun$ corresponding to~$\cE$ is a smooth point of the stack~$\Bun$, this guarantees that $q$ is smooth near $b$ (since $\Loc\to\Bun$ is a twisted cotangent bundle).  We will derive a criterion for the fibers of $q$~and~$\chi$ to be transversal at~$e$, and then show that it is satisfied by a generic vector bundle~$\cE'$ of degree~$0$, assuming that $C$ is ordinary.
%
%
\end{proof}

Now we turn to the proof of Lemma~\ref{lem:vanishing on fibers}.
Consider the stack $\wtld\Loc$ over $\AA^1$ whose fiber over $\la \in \AA^1(\k) = \k$ is the stack $\Loc_\la$ of rank~$N$ bundles on~$C$ with $\la$-connection; in particular
  $$\tLoc\xx_{\AA^1}\{1\}\isom\Loc, \quad
  \tLoc\xx_{\AA^1}\{0\}\isom\Higgs.$$
The stack $\wtld\Loc$ also has a canonical $\Gm$-action lifting the dilation action on~$\AA^1$.  Consequently, we have an isomorphism
$$\wtld\Loc \xx_{\AA^1} \Gm  \isom  \Loc \x \Gm.$$
Let $t$ be the coordinate function on~$\AA^1$.  Denote by $\Loc\sm$ the smooth part of~$\Loc$ and by $\wtld\Loc\sm$ the maximal open substack in~$\wtld\Loc$ smooth over~$\AA^1$.  Note that the fiber of $\wtld\Loc\sm$ is \emph{not} equal to the maximal open substack $\Higgs\sm$ of the central fiber but is rather the intersection of $\Higgs\sm$ with the union $\Higgs^{[p\ZZ]}$ of connected components parametrizing Higgs bundles of degree divisible by~$p$.

The stack $\wtld\Loc$ defines a filtration on functions, differential forms, etc.~on $\Loc$ (and on open substacks thereof): a form~$\eta$ on $\Loc$ belongs to the $k$'th filtered piece iff the pullback of~$\eta$ to $\Loc \x \Gm \into \wtld\Loc$ has pole of order not greater than~$k$ along $\Higgs = \wtld\Loc \xx_{\AA^1} \{0\}$.  This filtration is compatible with the de~Rham differential.  Similarly, we get a filtration on $\Gamma(\eF_{\Loc\sm}) = \Gamma(\Om^1_{\Loc\sm} / d\O_{\Loc\sm})$.  All these filtrations will be denoted by~$F^\blt$.

For example, there is a relative symplectic form on $\wtld\Loc\sm / \AA^1$ of weight~1 with respect to the $\Gm$-action.  Its restriction to the fibers over $0,1 \in \AA^1$ are the standard symplectic forms on $\Higgs\sm$ and $\Loc\sm$.  This means that $\Oloc \in F^1\Om^2(\Loc\sm)$.  It is also straightforward to check that
\begin{equation}\label{eq:degree of tth}
    \tilde\th \in F^p\eF(\Loc\sm) .
\end{equation}

%Choose a point $b \in \Bun$. First we will prove a weaker statement---namely, that the restriction of~$\b_0$ to the fiber $\Loc_b$ of~$\Loc$ at~$b$ (this fiber is an affine space) is a translation-invariant 1-form. Note that $\tilde\th|_{\Loc_b} = 0$ and therefore $\b_0|_{\Loc_b} = \sC(\th_0|_{\Loc_b})$. Let $F^\blt \O(\Loc_b)$ denote the filtration by degree and identify $\Om^1(\Loc_b) \isom \O(\Loc_b) \tsr T_b\Bun$. It is known that for any closed 1-form~$\a$ on $\Loc_b$ and any positive integer~$k$, if $\a \in F^{pk-2} \O(\Loc_b) \tsr T_b\Bun$ then $\sC(\a) \in F^{k-2} \O(\Loc_b\fr) \tsr (T_b\Bun)\fr$. So it is sufficient (for the weaker statement) to prove that the coefficients of~$\th_0|_{\Loc_b}$ have degree $\le 2p-2$. In fact we'll prove a stronger estimate.

\begin{lem}\label{lem:degree leq p}
  We have $\th_0 \in F^{p+1} \Om^1(\Loc\sm)$.  Equivalently, $\xi_0 \in F^p \cT(\Loc\sm)$.
\end{lem}

\begin{proof}
  We need to prove that $t^p \xi_0$ extends to $\wtld\Loc\sm$.  Recall that the value of~$\xi_0$ at a point corresponding to a bundle with connection (or $\O$-coherent $\D$-module) $\cM = (\cE,\na)$ is given by the infinitesimal deformation $\overline \cM = (\overline\cE,\overline\na)$  of~$\cM$ constructed in~\ref{subs:alt constr}.

  We will think of~$\overline \cM$ as an extension of~$\cM$ by itself.  Recall that the construction of~$\overline \cM$ in~\ref{subs:alt constr} uses the splitting of the gerbe~$\g$ on the 1st infinitesimal neighborhood of zero section in $T^*C\fr$. This splitting can be thought of as an extension of $\D$-modules $0\to \cT_C\fr[C] = \cT_C^{\tsr p} \to \cM_0 \to \O_C \to0$. Denote by $v \in \Ext^1_\D (\O_C, \cT_C\fr[C])$ the class of this extension.  We will also need the $p$-curvature of~$\cM$ thought of as a map of $\D$-modules $\curv_p(\cM) \colon \cM \tsr \cT_C\fr[C] \to \cM$. By unwinding the definition of~$\overline \cM$, it is not hard to check the following:

  \begin{clm}
    The class of~$\overline \cM$ in $\Ext^1_\D (\cM,\cM)$ is given by
    \begin{equation}\label{eq:class of M-bar}
        \on{class} (\overline \cM) = \curv_p(\cM) \cdot (\id_\cM \tsr v)
    \end{equation}
    where $\cdot$ denotes the composition $\Hom_\D (\cM \tsr \cT_C\fr[C], \cM) \tsr \Ext^1_\D (\cM, \cM \tsr \cT_C\fr[C]) \to \Ext^1_\D (\cM,\cM)$.  More precisely, the exact sequence $0\to \cM \to \overline\cM \to \cM \to0$ is canonically isomorphic to the pullback of $\cM \tsr (0\to \cT_C\fr[C] \to \cM_0 \to \O_C \to0)$ by $\curv_p(\cM)$.
  \end{clm}

  In order to construct the vector field~$\tilde\xi_0$ on $\wtld\Loc\sm$ extending $t^p \xi_0$, recall the notion of $p$-curvature of a $\la$-connection from~\ref{lconn}.  It allows to extend the above construction of~$\overline \cM$ to $\la$-connections to get the desired vector field~$\tilde\xi_0$.  (We just need to multiply the connection on~$\cM_0$ by~$\la$ and use tensor product of $\la$-connections.)
\end{proof}

\begin{proof}[Proof of Lemma~\ref{lem:vanishing on fibers}]
  From Lemma~\ref{lem:degree leq p} and formula~\eqref{eq:degree of tth} we see that $\sP(\b_0) = \tilde\th - \delta(\th_0) \in F^{p+1}\eF(\Loc\sm)$.  So we must have
  \begin{equation}\label{eq:b0 in F1}
    \b_0 \in F^1\Om^1((\Loc\sm)\fr)
  \end{equation}
  (otherwise $\sP(\b_0) \bx 1_{\O(\Gm)}$ extended to $\wtld\Loc\sm$ would have a pole of order $\ge 2p > p+1$ along $\Higgs$).  Let $\rY$ be the open part of~$\Higgs$ given by $\rY := \Higgs\sm^{[p\ZZ]} = \wtld\Loc\sm \xx_{\AA^1} \{0\}$.  Then we get a 1-form $\b_1$ on~$\rY\fr$ obtained by extending the (relative) 1-form $t\b_0$ from $(\Loc\sm)\fr \x \Gm$ to $(\wtld\Loc\sm)\fr$ and then restricting to~$\rY\fr$.  The form $\b_1$ has weight~$1$ with respect to the $\Gm$-action on $\rY\fr \ctd \Higgs\fr$.  For $b \in \Bun$ as in the statement of Lemma~\ref{lem:vanishing on fibers}, \eqref{eq:b0 in F1} implies that $\b_0|_{\Loc_b}$ is a translation-invariant form on $\Loc_b$, and the restriction of~$\b_1$ to the fiber $\Higgs_b$ of $\Higgs$ at~$b$ is the corresponding translation-invariant form on~$\Higgs_b$ (recall that $\Loc_b$ is an affine space over the vector space $\Higgs_b$).

  Denoting by $\Eu$ the differential of the $\Gm$-action on~$\rY$, we get a function $F = \iota_{\Eu} \b_1$ on~$\rY\fr$ of $\Gm$-weight $1$.  The restriction of~$F$ to $\Higgs_b$ is a linear function whose differential is $\b_1|_{\Higgs_b}$.  Since the projection $\rY \to \rB$ is proper over~$\bo$, the restriction of~$F$ to each component of~$\Higgs$ must be a pullback of a function~$F'$ on~$\rB$. This function must also have degree~$1$ with respect to the standard $\Gm$-action on~$\rB$. We want to show that $F' = 0$ and hence $F = 0$.  This will imply that $\b_1|_{\Higgs_b} = 0$ and therefore $\b_0|_{\Loc_b} = 0$.

  Since we know that $\b_0$ descends to~$\rB$, the function~$F'$ does not depend on the choice of connected component of~$\Higgs$. Now consider the Serre duality involution~$\sigma$ on~$\Lochf$.  Via the identification $\Lochf \isoto \Loc$ given by $\cM \mapsto \cM \tsr \omega^{\tsr(p-1)/2}$ it corresponds to an involution on~$\Loc$ given by $\cM \mapsto \cM\chk \tsr \omega^{\tsr p}$ which we will also denote by~$\sigma$.  It is easy to see that the determinant line bundle with connection $\Lna_{\det}$ is invariant under~$\sigma$, hence so is its curvature $\Oloc$ and extended curvature~$\tilde\th$. The vector field~$\xi_0$ can also be shown to be invariant under~$\sigma$.  So the $1$-form $\th_0$ is $\sigma$-invariant as well.  Recalling that $\sP(\b_0) = \tilde\th - \delta(\th_0)$, we see that $\b_0$ must also be $\sigma$-invariant.  Thus for the function~$F$ we get that it is invariant under an analogous involution on~$\Higgs$.  But then for~$F'$ it means that it should be invariant under the action of $-1 \in \Gm$ on~$\rB$, whereas in the preceding paragraph we saw it is \emph{anti-}invariant under the same element.  The desired equality $F'=0$ follows.
\end{proof}

\begin{rem}
  There is another proof of Proposition~\ref{prop Jer} which works for any $p>2$.  The idea is to reduce it to the corresponding local statement about the connection on the pullback oh determinant line bundle on the affine Grassmannian $\Gr_G$ to the twisted cotangent bundle associated with this determinant bundle (cf.\ Remark~\ref{appx|rem:tw cot Gr}).  To construct an analog of the vector field $\xi_0$, one should choose a splitting of the gerbe $\dd_D$ (classifying $\D$-modules on $D := \Spec\k\dbkts t$) on the 1st infinitesimal neighborhood of the zero section in $T^*D\fr$, and also choose a trivialization of the extension $\cM_0$ defined before formula~\eqref{eq:class of M-bar}, as an $\O$-module.  Every $G$-bundle on~$C$ together with a trivialization on the formal neighborhood of a point $x\in C$ gives a map $\Gr_G\to\Bun$ and a correspondence between the twisted cotangent bundles, which allows to deduce the global statement from the local one.  To prove the local statement, one can first show that the resulting vector field on $\tilde T^*_{\Ldet^D}\Gr_G$ is Liouville, and prove an analog of the identity~\eqref{tth id}.  The rest of the argument is similar to the global one presented here but the generic transversality lemma (the analog of Lemma~\ref{lem transv}) is much easier for the local $p$-Hitchin map and can be proved in any characteristic.

  Since Proposition~\ref{prop Jer} is the only place in the proof of the main theorem (Theorem~\ref{main_thm}), the stronger version of Proposition~\ref{prop Jer} would allow to relax the conditions on $p$ in the theorem to just $p>2$.
\end{rem}

\appendix
\addtocontents{toc}{\SkipTocEntry}
\section%[Twisted cotangent bundle to $\Coh$]%
    {Twisted cotangent bundle to the moduli stack of coherent sheaves}
\addtocontents{toc}{\protect\contentsline{section}%
  {\protect\tocsection{Appendix}{\thesection}%
    {Twisted cotangent bundle to $\Coh$}}%
  {\thepage}{appendix.\thesection}}
\label{appx}

The main goal of this appendix is to show that the twisted cotangent bundle to the stack of coherent sheaves on a smooth projective curve is identified with the stack of half-form twisted coherent $\D$-modules on that curve  and establish some properties of the corresponding tautological connection on a line bundle this cotangent bundle.  This is a characteristic-independent statement, except that we have to assume that the characteristic is not~2.  In fact, we prove it for arbitrary base scheme~$S$ defined over~$\ZZ[1/2]$.
In characteristic~$0$ these statements are classical (see, e.g., \cite[Chapter~IV]{Faltings}; in particular, Theorem~\ref{appx|main thm} corresponds to Lemma~IV.4 in \emph{loc.~cit.}).

We will understand all derived categories in the higher-categorical (\ie $(\infty,1)$ or DG) sense, so that it makes sense to talk about homotopies between (1\nbhp-)morphisms. (In fact, the $(2,1)$-categorical level would suffice: all our morphism spaces will be 1-groupoids.)

\subsection{Twisted cotangent bundles to stacks}\label{appx|subs:general}
Let $S$ be a scheme and  $\rX \xra\pi S$ a smooth Artin stack over $S$. For an $R$-point $x$ of $\rX$ where $R$ is a commutative ring, consider the groupoid $\gT_{\rX/S,x}$ of all dotted arrows in the diagram
\begin{equation}\label{appx|dd}
    \cxym{\Spec R \ar[r]^x \ar[d]_\iota & \rX \ar[d]^\pi \\
        \Du_R \ar@{.>}[ur]|-{\tilde x} \ar[r]_{\pi \circ x \circ p} & S}
\end{equation}
where $\Du_R = \Spec (R[\eps]/\eps^2)$ and $\iota\colon \Spec R \to \Du_{R}$, $p\colon \Du_R \to \Spec R$ are the natural morphisms. This is an $R$-linear Picard groupoid, so it corresponds to a 2-step complex of $R$-modules $\cT_{\rX/S,x}^\blt$ living in degrees $0$ and $-1$. This complex is perfect, compatible with derived base change and satisfies descent, and therefore defines a perfect object $\Tblt_{\rX/S} \in D^b(\rX)$. It is called the {\em tangent complex} of $\rX$ over $S$. The {\em cotangent stack} is then defined as
\[
    T^*(\rX/S) := \Spec_\rX (\Sym_{\O_\rX} \cH^0(\Tblt_{\rX/S}))
\]
where $\cH^0$ is taken with respect to the natural t-structure on $D^b(\rX)$.

\ncmd{\TXL}{\tTblt_{\rX/S,\L}}
Now let $\L$ be a line bundle on $\rX$ and $\rP \to \rX$ the corresponding principle $\Gm$-bundle. Then $\Tblt_{\rP/S}$ is a $\Gm$-equivariant complex and therefore descends to a complex $\TXL$ on~$\rX$ which fits into an exact triangle
\begin{equation}\label{appx|tT ex.tr}
\O_\rX \xra{i} \TXL \to \Tblt_{\rX/S} \xra\delta \O_\rX[1] \,
\end{equation}
We will sometimes refer to~$\TXL$ as the {\em extended tangent complex}. Now define the {\em twisted cotangent stack} $\tilde T^*_\L(\rX/S)$ as
\begin{equation}\label{appx|def tT*L}
   \tilde T^*_\L(\rX/S) := \Spec_\rX (\Sym_{\O_\rX} \cH^0(\TXL)) \xx_{\AA^1_S} \{1\}_S
\end{equation}
where the morphism from the first factor to $\AA^1_S$ is induced by $i$.

\begin{rem}\label{appx|rem:tT* interp}
  There are several alternative interpretations of the stack $\tilde T^*_\L(\rX/S)$:
  \begin{enumerate}
    \item It is the spectrum of the quotient of $\Sym \cH^0(\Tblt_{\rX/S})$ by the ideal generated by $1 - i(1_{\O_\rX})$.
    \item Its $R$-points lying over $x\colon \Spec R \to \rX$ are given by ``splittings'' of the pullback under~$x$ of triangle~\eqref{appx|tT ex.tr}.
    \item\label{appx|null-htpy} These splittings are the same as null-homotopies of~$\delta_x$.
  \end{enumerate}
\end{rem}

There is a closed substack $\rZ \ctd \rX$ consisting of points $x \in \rX$ where $i$ acts non-trivially on local cohomology. Then the map $\tilde T^*_\L(\rX/S) \to \rX$ factors through~$\rZ$ and over $\rZ$ it looks like a torsor for $T^*(\rX/S)$.

Another way to define the complex $\TXL$ is as follows. For an $R$-point~$x$ of~$\rX$ let $B\GG_a(R)$ denote the classifying groupoid for the group $\GG_a(R) = (R,+)$. We will define a functor $\delta'\colon \gT_{\rX/S,x} \to B\GG_a(R)$. Namely, for any $\tilde x\colon \Du_{R} \to \rX$ as in~\eqref{appx|dd}, we set $\delta'_x(\tilde x)$ to be the torsor of isomorphisms $\tilde x^*\L \isoto p^*x^*\L$ whose restriction to $\Spec R \ctd \Du_{R}$ is $\id_{x^*\L}$. The action of $\GG_a(R)$ on $\delta'_x(\tilde x)$ is given by the composition $\GG_a(R) \isoto 1+\eps R \ctd (R[\eps]/\eps^2)^\x = \O(\Du_R)^\x \isoto \Aut(\tilde x^*\L)$. The map $\delta'_x$ corresponds to a map $\delta_x\colon \Tblt_{\rX/S,x} \to R[1]$. The maps $\delta_x$ for all $R$ and $x$ glue to a map $\delta\colon \Tblt_{\rX/S} \to \O_\rX[1]$. Then $\TXL$ can be reconstructed as ``the'' cone of $\delta$.%
    \footnote{In order to reconstruct $\TXL$ canonically, one needs to understand the derived categories in the $(\infty,1)$-categorical (or DG) sense.}
Below we write $\delta^\L$ instead of $\delta$ to show explicitly the dependence of~$\delta$ on~$\L$.

We will need the following lemma whose straightforward proof is omitted.

\begin{lem}\label{appx|lem:delta base chge}
  Let $f\colon \rX \to \rY$ be a morphism of smooth Artin stacks and $\L$ a line bundle on~$\rY$. Then we have a commutative diagram
  \[
      \cxym{\Tblt_{\rX/S} \ar[d]_{df} \ar[r]^{\delta^{f^*\L}} & \O_\rX [1] \ar@{=}[d] \\
            f^*\Tblt_{\rY/S}          \ar[r]^{f^*\delta^\L}   & f^*\O_\rY [1]  }
  \]
\end{lem}

\subsection{The stack \texorpdfstring{$\Coh(C)$}{Coh(C)} and its twisted cotangent bundle}
Now let $C\to S$ be a smooth proper family of algebraic curves. We work with schemes over $S$ throughout, so we will often drop ``$S$'' from the notation writing $\cT_\rX$, $T^*\rX$, $\x$, etc.\ instead of $\cT_{\rX/S}$, $T^*(\rX/S)$, $\xx_S$,~etc.  If $R$  is a ring, and $s \colon \Spec R \to S$ is an $R$-point of~$S$, we will denote by $C_s$ or $C_R$ the base change of~$C$ by~$s$, that is, $C_s = C_R := C \xx_{S,s} \Spec R$.

We consider the stack $\Coh(C)$ of coherent sheaves on fibers of $C$.  Its groupoid of $R$-points is given by
\[
   \Coh(C)(R) = \left\{(s,\cF) \biggm|
   \begin{matrix} s\in S(R) ; \\
   \text{$\cF$ is a $(\Spec R)$-flat coherent sheaf on $C_s$}
   \end{matrix}\right\} .
\]
Below we abbreviate $\Coh = \Coh(C)$. Consider the {\em determinant bundle} $\Ldet$ on $\Coh$. Its fiber at an $R$-point $x$ of~$\Coh$ corresponding to a coherent sheaf~$\cF_x$ on $C_R$ is given by
\[
    (\Ldet)_x = \det_R \RGam(C_R,\cF_x) .
\]

From now on we assume that
\begin{equation}\label{appx|2inv}
    \text{2 is invertible on $S$,\quad i.e.,\quad} 2\cdot 1_{\O(S)} \in \O(S)^\x. \tag{$*$}
\end{equation}

We will be interested in the twisted cotangent stack corresponding to $\Ldet$. Namely we will prove the following.

\begin{thm}\label{appx|main thm}
  There is a canonical isomorphism of stacks over $\Coh$:
  \begin{equation}\label{appx|main eq}
    \tilde T^*_{\Ldet} \!\Coh  \isom  \Connhf
  \end{equation}
  where $\Connhf$ is the moduli stack of $\O_C$-coherent modules over the algebra $\D_{C,\ohf}$ of differential operators in $\omega_C^{\tsr1/2}$.
\end{thm}

\ncmd{\tTCoh}{\tTblt\SmSub[,]{\Coh}}
Denote $\tilde \cT^\blt_{\Coh,\Ldet}$ by $\tTCoh$. Now it follows from the formula~\eqref{appx|def tT*L} that the datum of an $R$-point of $\tilde T^*_{\Ldet}\!\Coh$ lying over $x \in \Coh(R)$ is equivalent to the datum of a nilhomotopy of the map
\begin{equation*}
    \delta_x\colon T^\blt_{\Coh,x} \to R .
\end{equation*}
We can therefore reduce the study of~$\tilde T^*_{\Ldet} \!\Coh$ to the study of~$\delta_x$.

It is known\footnote{At least for the open part of~$\Coh$ parametrizing vector bundles; but the same proof works for all~of $\Coh$.} that the tangent complex $\Tblt_{\Coh,x}$ to $\Coh$ at a point $x \in \Coh(R)$ is canonically isomorphic to $(\REnd \cF_x)[1] \in D^b(R\hmod)$. So we have to study the map $\delta_x[-1] \colon \REnd \cF_x \to R$. Using the Serre duality, we get a map
\begin{equation*}
    \a_x\colon \cF_x \to \cF_x \tsr \omega_{C_R/R} [1] .
\end{equation*}
The map $\a_x$ corresponds to an extension
\begin{equation}\label{appx|Phi_x exseq}
    0 \to \cF_x \tsr \omega_{C_R/R} \to \Phi_s(\cF_x) \to \cF_x \to 0 .
\end{equation}
It is clear that, for any~$R$ and a map $s\colon \Spec R \to S$, the assignment $\cF_x \mapsto \Phi_s(\cF_x)$ defines a functor from the {\em groupoid} of $R$-flat coherent sheaves on $C_R$ to itself compatible with base-change of~$R$. In other words, we have a morphism of stacks $\Phi\colon \Coh \to \Coh$.

We will show that $\Phi_s$ extends to {\em non-invertible} morphisms of sheaves and, moreover, has the following description:

\begin{prop} \label{appx|prop:desc Phi}
  Let $\Delta_C^{\angs2}$ be the 2nd infinitesimal neighborhood of the diagonal $\Delta_C \ctd C \x C$ and $p,q \colon \Delta_C^{\angs2} \toto\nlbk C$ the restriction of the two projections $C\x C \toto C$. Denote also by $p_R$, $q_R$ their base-change by $s\colon \Spec R \to\nlbk S$. We have a canonical isomorphism
  \begin{equation}\label{appx|eq:desc Phi}
    \Phi_s(\cF) \isom q_{R*}
                \bigl( p_R^*\cF \tsr (s \x \id_{\Delta_C^{\angs2}})^* (p^*\omega_C \tsr q^*\omega_C^{\tsr{-1}})^{\tsr1/2} \bigr)
  \end{equation}
  where $^{\tsr1/2}$ denotes the canonical square root of the line bundle on $\Delta_C^{\angs2}$ which is trivial on~$\Delta_C$. (It exists and is essentially unique due to the assumption~\eqref{appx|2inv}.)
\end{prop}

Denote by $\tilde\Phi_s(\cF)$ the right-hand side of~\eqref{appx|eq:desc Phi}.

\begin{proof}[Proposition~\ref{appx|prop:desc Phi} implies Theorem~\ref{appx|main thm}] \sloppy
  According to Remark~\ref{appx|rem:tT* interp}, point~\ref{appx|null-htpy}, for any $x\colon \Spec R \to \Coh$ the elements of $\Hom_{\Coh} (\Spec R, \tilde T^*_{\Ldet} \!\Coh)$ correspond bijectively to null-homotopies of the map~$\delta_x$. By the Serre duality isomorphism, these are the same as null-homotopies of~$\a_x$ and therefore the same as splittings of~\eqref{appx|Phi_x exseq}. Now, using~\eqref{appx|eq:desc Phi}, it is not hard to see that these splittings correspond to $(s \x \id_C)^* \D_{C,\ohf}$-module structures on~$\cF_x$.
\end{proof}

\begin{lem}\label{appx|lem:Phi fnctr}
  For any $s\colon \Spec R \to S$ the map of groupoids $\Phi_s\colon \Coh(R) \to \Coh(R)$ extends to a self-functor of the category of $R$-flat coherent sheaves on $C \xx_s \Spec R$. In other words, $\Phi_s$ extends to non-invertible morphisms.
\end{lem}

\begin{proof}
  Let $\pr_1,\pr_2\colon \Coh \x \Coh  \toto  \Coh$ be the projections and $\Sigma\colon \Coh \x \Coh  \to  \Coh$ the map classifying direct sum of coherent sheaves. By the properties of determinant, we have a canonical isomorphism
  \begin{equation}\label{appx|eq:Sigma*Ldet}
      \pr_1^*\Ldet \tsr \pr_2^*\Ldet  \isom  \Sigma^*\Ldet .
  \end{equation}

  Applying Lemma~\ref{appx|lem:delta base chge} to $\Ldet$~and~$\Sigma$ yields a commutative diagram on the left:
  \begin{equation*}
      \xymatrix{\Tblt_{\Coh\x\Coh} \ar[d]_{d\Sigma} \ar[r]^-{\delta^{\Sigma^*\Ldet}} & \O_{\Coh\x\Coh} [1] \ar@{=}[d] \\
                \Sigma^*\Tblt_{\Coh}                \ar[r]^-{\Sigma^*\delta^{\Ldet}} & \O_{\Coh\x\Coh} [1]            }
      \qquad
      \fbox{\xymatrix{\REnd(\cF_x) \oplus \REnd(\cF_y) \ar[d] \ar[r]^-{\delta^{\Sigma^*\Ldet}_{(x,y)}} & R[1] \ar@{=}[d] \\
                      \REnd(\cF_x \oplus \cF_y)         \ar[r]^-{\Sigma^*\delta^{\Ldet}_{\Sigma(x,y)}} & R[1]            }}
  \end{equation*}
  Pulling it back under some $(x,y) \colon \Spec R \to \Coh \x \Coh$, we obtain the diagram on the right (boxed). From~\eqref{appx|eq:Sigma*Ldet} we see that the top arrow of this diagram is equal (canonically homotopic) to $\delta^{\Ldet}_x p_1 + \delta^{\Ldet}_y p_2$ where $p_1$,~$p_2$ are projections to the summands. Note also that the left vertical arrow is given by direct sum of (derived) endomorphisms.
  Now if we apply the Serre duality to the boxed diagram, we get
  \begin{equation*}
      \a_{\Sigma(x,y)} = \begin{pmatrix}\a_x &*\\{*}& \a_y\end{pmatrix}  \colon\:  \cF_x \oplus \cF_y  \longto  (\cF_x \oplus \cF_y) \tsr \omega_{C_R/R} [1] .
  \end{equation*}
  (Note that the `equals' sign here again really means `canonically homotopic.') Since we already know that $\Phi$ is functorial under isomorphisms, $\Phi(\cF_x\oplus\cF_y)$ must have an action of the automorphism (algebraic) group of $\cF_x\oplus\cF_y$, in particular, of $\Gm\x\Gm$ acting diagonaly.  One can see from this that the non-diagonal entries marked by~$*$ in the above formula are canonically homotopic to~$0$, so that $\a_{\Sigma(x,y)}=\a_x\oplus\a_y$. In other words, we have an isomorphism
  \begin{equation*}
      \Phi_s(\cF_x \oplus \cF_y) \isom \Phi_s(\cF_x) \oplus \Phi_s(\cF_y) .
  \end{equation*}

  We already know that $\Phi_s$ is functorial with respect to isomorphisms, so $\Phi_s(\cF_x \oplus \cF_y)$ is acted on by the automorphism group of $\cF_x \oplus \cF_y$. Moreover, since $\Phi$ is compatible with base change, we have an action of the {\em group-scheme} over~$R$ given by $R' \mapsto \Aut((\cF_x \oplus \cF_y) \tsr_R R')$. In particular, consider the automorphism $a_f = \left( \begin{smallmatrix} \id_{\cF_x} & 0 \\ f & \id_{\cF_y} \end{smallmatrix} \right)$ for some $f\colon \cF_x \to \cF_y$. It can be shown that $\Phi(a_f)$ has the form $\left( \begin{smallmatrix} \id_{\Phi(\cF_x)} & 0 \\ \tilde f & \id_{\Phi(\cF_y)} \end{smallmatrix} \right)$. Now we let $\Phi(f) = \tilde f$. Using triple direct sums, one can prove that $\Phi(g \circ f) = \Phi(g) \circ \Phi(f)$ for any $g\colon \cF_y \to \cF_z$.
\end{proof}

\begin{lem}\label{appx|lem:skyscraper}
  Let $\L$ be a line bundle on $\Spec R$ and $p$ an $R$-point of~$C$. Let $\gamma_p=(p,\id_{\Spec R}) \colon \Spec R \to C_R = C \x \Spec R$ be the embedding of the graph of~$p$. Then Proposition~\ref{appx|prop:desc Phi} holds for $\cF = \gamma_{p*} \L$.
\end{lem}

\begin{proof}
  Let $x\colon \Spec R \to \Coh$ be the point corresponding to~$\cF$ and $\Gamma_p = \gamma_p(\Spec R) = \supp\cF \ctd C_R$. First we will show that the two extensions of~$\cF$ by $\cF \tsr\nlbk \omega$ in the two sides of~\eqref{appx|eq:desc Phi} have the same class in $\Ext^1(\cF, \cF \tsr\nlbk \omega)$. Note that for the ``relative skyscraper'' as in the lemma, this $R$-module is identified with~$R$. We claim that the classes of both extensions are equal to $1$ under this identification. For the RHS of~\eqref{appx|eq:desc Phi} we have pushforward of a line bundle on the 2nd infinitesimal neighborhood of~$\Gamma_p$ whose restriction to~$\Gamma_p$ is~$\L$. Now the statement can be easily seen from the construction of the identification.

  For the LHS we are interested in the image of~$\a_x$ under the projection
  \begin{equation*}
      \RHom(\cF,\cF\tsr\nlbk\omega[1]) \to \Ext^1(\cF,\cF\tsr\nlbk\omega).
  \end{equation*}
  By Serre duality, this map is dual to $(\End\cF)[1] \isom H^{-1} (\Tblt_{\Coh,x}) [1] \to (\REnd\cF)[1] \isom \Tblt_{\Coh,x}$. Therefore we have to study the restriction of~$\delta_x$ to the $(-1)$st cohomology of~$\Tblt_{\Coh,x}$. This restriction is responsible for the action of infinitesimal automorphisms of~$x$ on~$(\Ldet)_x$. The group scheme of automorphisms of~$x$ is identified with~$\GG_{mR}$, and it acts on~$(\Ldet)_x$ via the tautological character. This means that $\delta_x$ restricted to $H^{-1} (\Tblt_{\Coh,x}) [1] \isom (\End\cF) [1] \isom R[1]$ is the identity. Now the statement about the class in~$\Ext^1$ follows from the fact that the Serre duality pairing of the canonical generators of $\End\cF$ and $\Ext^1(\cF, \cF \tsr\nlbk \omega)$ is equal to~$1$.

  Thus we showed that the two extensions in~\eqref{appx|eq:desc Phi} are non-canonically isomorphic. The set of isomorphisms is a torsor for the group~$R$ which is compatible with base change. So we get a $\GG_a$-torsor~$\cA$ on the open piece $\Coh^{[1]} \isom C \x B\Gm$ of~$\Coh$ classifying length~$1$ sheaves, and we need to construct a canonical trivialization of~$\cA$.

  Consider the involution~$\sigma$ on~$\Coh$ given by Serre duality: $\cF \mapsto \cF\chk \tsr\nlbk \omega$. We have $\sigma^*\Ldet \isom \Ldet$. One can check that the isomorphism $\RHom(\cF_x, \cF_x \tsr \omega) \isom (\Tblt_{\Coh,x})\chk \isoto (\Tblt_{\Coh,\sigma(x)})\chk \isom \RHom(\cF_{\sigma(x)}, \cF_{\sigma(x)} \tsr \omega)$ is given by taking dual morphism. Hence the morphism~$\delta_x$ is dual to~$\delta_{\sigma(x)}$, and therefore $\Phi_s(\cF_x) \isom (\Phi_s(\cF_{\sigma(x)}))\chk$.

  Restricting to $\Coh^{[1]}$, for the torsor~$\cA$ we get $(\sigma|_{\Coh^{[1]}})^* \cA \isom -\cA$. On the other hand, any $\GG_a$-torsor on $\Coh^{[1]} \isom C \x B\Gm$ descends to a torsor~$\tilde\cA$ on~$C$. We get $\tilde\cA \isom -\tilde\cA$ and thus $\tilde\cA$ is trivial due to~\eqref{appx|2inv}, so we are done.\footnote{Actually, we will also need the statement that the isomorphism~\eqref{appx|eq:desc Phi} in the case in question is compatible with morphisms of~$\L$'s that are not necessarily isomorphisms.}
\end{proof}

\begin{lem}\label{appx|lem:line bundle}
  Suppose $x$ is an $R$-point of~$\Coh$ corresponding to a line bundle~$\L$ on $C_R$. Then Proposition~\ref{appx|prop:desc Phi} holds for~$x$.
\end{lem}

\begin{proof}
  Consider the $S$-scheme $S' := C_R$ and let $\Delta \colon C_R \to C_{S'} = C \xx_S C \xx_S \Spec R$ be the diagonal embedding. Let $\L_{S'}$ be the line bundle on~$C_{S'}$ obtained from~$\L$ by base-change along $S' \to \Spec R$. Now consider the map of coherent sheaves on~$C_{S'}$:
  \begin{equation*}
      f\colon \L_{S'} \to \Delta_*\L = \Delta_*\Delta^*\L_{S'} .
  \end{equation*}
  According to Lemma~\ref{appx|lem:Phi fnctr}, we have a map
  \begin{equation*}
      \Phi_{S'}(f) \colon \pr_1^*\Phi_s(\L) \to \Phi_{S'}(\Delta_*\L)
  \end{equation*}
  (we used that $\Phi$ is compatible with base-change), and therefore, by adjunction, a map
  \begin{equation*}
      g \colon \Phi_s(\L) \to \pr_{1*}\Phi_{S'}(\Delta_*\L) .
  \end{equation*}

  By Lemma~\ref{appx|lem:skyscraper}, we have $\Phi_{S'}(\Delta_*\L) \isom \tilde\Phi_{S'}(\Delta_*\L)$. Also, from the definition of~$\tilde\Phi$ one can see that $\pr_{1*}\tilde\Phi_{S'}(\Delta_*\L) \isom \tilde\Phi_s(\L)$. So $g$ is a map $\Phi_s(\L) \to \tilde\Phi_s(\L)$. By construction, $g$ is a map between extensions of~$\L \tsr \omega$ by~$\L$, so it gives the desired isomorphism.
\end{proof}

\begin{proof}[Proof of Proposition~\ref{appx|prop:desc Phi} (sketch)]
  Locally on~$\Spec R$, we can find a resolution of a coherent sheaf~$\cF$ on $C_R$ by direct sums of line bundles. Any two such resolutions are related by a sequence of homotopy equivalences. Thus, due to the functoriality of~$\Phi$ (Lemma~\ref{appx|lem:Phi fnctr}), we can reduce Proposition~\ref{appx|prop:desc Phi} to Lemma~\ref{appx|lem:line bundle}.
\end{proof}

\subsection{The determinant bundle with connection on~\texorpdfstring{$\Connhf$}{Conn\_\{1/2\}}}
For a smooth stack~$\rX$ and a line bundle~$\L$ on~$\rX$, the pullback of~$\L$ to~$\tilde T^*_\L \rX$ carries a canonical connection.  Applying this observation to $\rX = \Coh(C)$ and $\L = \Ldet$, and taking Theorem~\ref{appx|main thm} into account, we get a connection~$\na_{\det}$ on the pullback~$\Ldet'$ of~$\Ldet$ to $\Connhf$.  In this subsection we state some properties of~$\Ldet$ and $\na_{\det}$.

In the main body we use the following statement regarding the behavior of $\na_{\det}$ on short exact sequences of coherent sheaves with $\ohf$-connection.

Suppose we have an $S$-scheme $S'$, and a short exact sequence of $S'$-families of coherent $\D_{\ohf}$-modules on~$C$, \ie of $S'$-flat coherent sheaves on~$C_{S'}$:
\begin{equation}\label{appx|eq:F'FF''}
    0 \to\cF' \to\cF \to\cF'' \to 0
\end{equation}
with compatible (relative over~$S'$) $\ohf$-connections.  Let $x,x',x'' \in \Coh(C)(S')$ be the corresponding points of~$\Coh$.  Then we have the following relation between the pullbacks of the determinant bundle:
\begin{equation}\label{appx|eq:Ldet factzn}
    (\Ldet)_x \isom (\Ldet)_{x'} \tsr (\Ldet)_{x''}.
\end{equation}
Now from the connection~$\na_{\det}$ we get connections (over~$S$) on both sides of the above isomorphism.

\begin{lem}\label{appx|lem:Ldet ses-compat}
  The isomorphism~\eqref{appx|eq:Ldet factzn} is compatible with connections.
\end{lem}

\begin{proof}
  We first reduce the statement to the case where the exact sequence~\eqref{appx|eq:F'FF''} splits.  Namely, note that for a general exact sequence we can construct a family of coherent sheaves with $\ohf$-connections on~$C$ parametrized by $S' \x \AA^1/\Gm$ whose restriction to $C_{S'} \isoto C_{S'} \x (\AA^1\setm\{0\})/\Gm \into C_{S'} \x \AA^1/\Gm$ is isomorphic to~$\cF$ and whose restriction to $C_{S'} \x \{0\} \into C_{S'} \x \AA^1 \to C_{S'} \x \AA^1/\Gm$ is isomorphic to $\cF' \oplus \cF''$.  Denote by~$\tilde x$ the corresponding map $S' \x (\AA^1/\Gm) \to \Connhf$.  Pulling back $(\Ldet',\na_{\det})$ under~$\tilde x$ we get a line bundle with connection on $S' \x (\AA^1/\Gm)$ over $\AA^1/\Gm$.  We want to compare this connection with the one obtained as direct sum of pullbacks by $x'$~and~$x''$.  The difference is an $\AA^1/\Gm$-family of one-forms on~$S'/S$.  But any such family is necessarily constant.  So it is enough to prove that it is~$0$ on $\{0\}/\Gm$ where the exact sequence splits.

  For split exact sequences, the statement is equivalent to compatibility of the isomorphism~\eqref{appx|main eq} with symmetric monoidal structure on both sides given by direct sum.
\end{proof}

\subsection{Curvature of \texorpdfstring{$\na_{\det}$}{nabla det}}
For any vector bundle $\cE$ with connection $\na$ on a smooth stack~$\rX$, one can define the \emph{curvature} of~$\na$ as follows.  Suppose $x$ is an $R$-point of~$\rX$, and $\xi$ is a tangent vector at~$x$ to~$\rX$.  As explained in~\ref{appx|subs:general}, $\xi$ corresponds to a lift of $x$ to an $R[\eps]/\eps^2$-point~$\tilde x_\xi$ of~$\rX$ as in diagram~\eqref{appx|dd}.  The connection $\na$ then gives an identification $\na_\xi \colon \tilde x_\xi^*\cE \isoto p^*x^*\cE$ where $p$ is as in the line below diagram~\eqref{appx|dd}.  Now assume that we have two tangent vectors $\xi,\eta \in \gT_{\rX,x}$ to $\rX$ at~$x$.  Then the corresponding lifts $\tilde x_\xi,\tilde x_\eta$ can be glued to a map $\Du_R\sqcup_{\Spec R}\Du_R = \Spec ((R[\eps]/\eps^2)\x_R (R[\eps]/\eps^2)) \to \rX$.  Since $\rX$ is smooth by assumption, we can extend the latter map to $\Tilde{\tilde x}\colon \Du_R\x_{\Spec R}\Du_R \to \rX$.  Note that $\Tilde{\tilde x}$ can be regarded both as a tangent vector $\tilde\eta$ at $\tilde x_\xi$ and as a tangent vector $\tilde\xi$ at $\tilde x_\eta$.  Let $p_1,p_2\colon \Du_R\x_{\Spec R}\Du_R \toto \Du_R$ and $p_{12}\colon \Du_R\x_{Spec R}\Du_R \to \Spec R$ be the natural projections.  Now we have a chain of isomorphisms of vector bundles on $\Du_R\x_R\Du_R$:
\[
    p_{12}^* x^* \cE \stackrel{p_2^*\na_\xi^{-1}}\isoto p_2^*\tilde x_\eta^* \cE \stackrel{\na_{\tilde\eta}^{-1}}\isoto \Tilde{\tilde x}^*\cE
    \stackrel{\na_{\tilde\xi}}\isoto p_1^*\tilde x_\xi^* \cE \stackrel{p_1^*\na_\eta}\isoto p_{12}^* x^* \cE.
\]
One can check that this isomorphism is identity on $\Du_R\sqcup_{\Spec R}\Du_R \ctd \Du_R\xx_{\Spec R}\Du_R$, so it has the form $\id + (\eps\tsr\eps)F$ where $F\in\End_R(x^*\cE)$.  Also, one can prove that $F$ does not depend on the choice of~$\Tilde{\tilde x}$ and is bilinear and anti-symmetric in $\xi$~and~$\eta$.  So it gives an element $F_{\na,x} \in \Lambda^2 H^0(T_x\rX)^*$ called the curvature of~$\na$ at~$x$.  If we let $R$ and $x$ vary, we'll get an $\cEnd\cE$-valued 2-form $F_\na \in \Gamma(\rX,\Om^2_\rX \tsr\cEnd\cE)$.  (One can define $F_\na$ for singular~$\rX$ using K\"ahler differentials.)  As usual, if $\cE$ is a line bundle, we have $\cEnd\cE = \O_\rX$ and therefore we can think of~$F_\na$ as just a $2$-form on~$\rX$.

On the other hand, if we take $\rX$ to be $\Connhf(C)$ (or the smooth part thereof), then we have an identification $\cT^\blt_{\rX,x} \isoto (\REnd_{\D_{\ohf}} \cF_x)[1]$ where $\cF_x$ is the $R$-flat coherent sheaf on $C_R=C\x_{\Spec R}S$ with (relative) $\ohf_C$-connection corresponding to~$x$.  On this complex, we have the Poincar\'e duality pairing (since twisted $\O$-coherent $\D$-modules form a $2$-Calabi--Yau category): $((\REnd_{\D_{\ohf}} \cF_x)[1])^{\tsr2} \to R$ which turns out to be anti-symmetric.  Therefore it gives a $2$-form on $(\Connhf)\sm$.

\begin{prop}\label{appx|prop:curv Ldet}
  The curvature $\Oloc$ of~$\na_{\det}$ on $(\Lochf)\sm$ coincides with the $2$-form $\Oloc'$ constructed from the Poincar\'e duality pairing.
\end{prop}

\begin{proof}
  For simplicity, we will assume that $S = \Spec\k$ for an algebraically closed field~$\k$.  Then the equality of $2$-forms can be checked on $\k$-points, so we may assume that $R=\k$.  Let $b'$ be a $\k$-point of $(\Lochf)\sm$ and $(\cE,\na)$ the corresponding bundle with $\ohf$-connection on~$C$.  Denote by~$b$ the image of~$b'$ in $\Bun$.  For any point~$x$ of~$C$, we have an evaluation map $(T^\blt_{\Bun,b})^* \isoto \RGam(C,\omega_C\tsr\cEnd\cE) \to \omega_{C,x}\tsr \End\cE_x$.  The dual map $\rho_x\colon T_{C,x} \tsr \End\cE_x \to T^\blt_{\Bun,b}$ can be explicitly described via ``infinitesimal modifications'' of~$\cE$ near~$x$.  Since the intersection of kernels (on 0th cohomology) of the evaluation maps is~$0$, the images of~$\rho_x$ for all~$x$ generate $H^0T^\blt_{\Bun,b}$.

  Now let $\xi',\eta'$ be two tangent vectors to $\Lochf$ at~$b'$, and $\xi,\eta$ their images in $\gT_{\Bun,b}$.  By the above, we can find two finite subsets $I,J\ctd C(\k)$ and two collections $\{P^{(x)}\in T_{C,x}\tsr\End\cE_x\}_{x\in I}$ and $\{Q^{(y)}\in T_{C,y}\tsr\End\cE_y\}_{y\in J}$ so that
  \begin{equation}\label{appx|eq:xi=sum eta=sum}
      \xi = \sum_{x\in I} \rho_x(P^{(x)}), \quad \eta = \sum_{y\in J} \rho_y(Q^{(y)}).
  \end{equation}
  Moreover, one can arrange so that $I\cap J = \eset$.  Now $\xi'$ corresponds to a $\Du_\k$-family of bundles with $\ohf$-connection on~$C$, which we denote by $(\cE_1,\na_1)$.  Since $\cE_1$ corresponds to $\xi$, the above presentation gives a canonical identification $\cE_1|_{U_I\x\Du} \isoto \cE|_{U_I} \bx \O_\Du$ where $U_I=C\setm I$, and under this identification $\na_1|_{U_I\x\Du}$ takes the form $\na|_{U_I}\bx1 + A\bx\eps$ for some $A\in (\cEnd\cE|_{U_I}) \tsr \omega_{U_I}$.  Similarly, $\eta'$ corresponds to $(\cE_2,\na_2)$ and $(\cE_2,\na_2)|_{U_J\x\Du} \isom (\cE|_{U_J} \bx \O_\Du, \na|_{U_J}\bx1 + B\bx\eps)$ where $B\in (\cEnd\cE|_{U_J}) \tsr \omega_{U_J}$.  One can check from the construction of the Poincar\'e pairing that
  \begin{equation}\label{appx|eq:OLoc'}
      \Oloc'(\xi',\eta') = \sum_{x\in I} \on{Tr} P^{(x)}B_x - \sum_{y\in J} \on{Tr} Q^{(y)}A_y.
  \end{equation}
  (Here $B_x\in(\End\cE_x)\tsr\omega_{C,x}$ and $A_y\in(\End\cE_y)\tsr\omega_{C,y}$ are values of $B$~and~$A$ at points $x$~and~$y$, respectively.)

  Now we have to compute the curvature of~$\na_{\det}$ evaluated at $\xi'$~and~$\eta'$.  For that purpose, consider the ``space'' (in fact an ind-scheme) $\rH_{I\cup J,b}$ parametrizing pairs $(\cE',\phi)$ where $\cE'$ is a vector bundle on~$C$, and $\phi$ is an identification $\cE|_{U_{I\cup J}} \isoto \cE'|_{U_{I\cup J}}$.  We will denote the point of this space corresponding to $(\cE,\id_{\cE|_{U_{I\cup J}}})$ by $b_{I\cup J}$.  By a well-known argument, $\rH_{I\cup J,b}$ is isomorphic to a product of $\#(I\cup J)$ copies of the affine Grassmannian for $\GL_N$ (it is the fiber at $I\cup J$ of a twisted version of the Beilinson--Drinfeld Grassmannian):
  \[
      \rH_{I\cup J,b}  =  \prod_{x\in I\cup J} \rH_{x,b}.
  \]
  We have the natural forgetful map $q_b\colon \rH_{I\cup J,b} \to \Bun$, and it is known that $q_b^*\Ldet\tsr(\Ldet)_b^{\tsr-1}$ (where the second factor is just a vector space) decomposes into a product of ``local factors'':
  \begin{equation}\label{appx|eq:q*Ldet=bigboxtimes}
      q_b^*\Ldet\tsr(\Ldet)_b^{\tsr-1} \isom \bigboxtimes_{x\in I\cap J} \Ldet^{\rH_{x,b}}.
  \end{equation}

  We can compute the cotangent bundle $T^*\rH_{x,b}$ using Serre duality: it is identified with the moduli space of triples $(\cE',\phi,a)$ where $(\cE,\phi) \in \rH_{x,b}$ and $a$ is Higgs field on~$\cE'$ defined on the formal neighborhood $\hat x$ of~$x$.\footnote{One should modify the definition of tangent sheaf for the case of ind-schemes; see~\cite[Sect.~7.11]{BD} where the tangent sheaf is defined for any formally smooth ind-scheme of ind-finite type.  One can then define the cotangent bundle (as an \emph{ind-scheme} rather than a sheaf) by the usual ``$\Spec\Sym$'' construction.}  One can show that the twisted cotangent bundle $\tilde T^*\rH_{x,b}$ to $\rH_{x,b}$ corresponding to $\Ldet^{\rH_{x,b}}$ is canonically identified with the moduli $\Lochf^{x,b}$ of triples $(\cE',\phi,\na)$ where $(\cE',\phi) \in \rH_{x,b}$, and $\na$ is an $\ohf$-connection on $\cE'|_{\hat x}$.  The symplectic form on $\Lochf^{x,b}$ will be denoted $\Oloc^{x,b}$.  The map $T^*\Bun \x_{\Bun}\rH_{I\cup J,b} \to T^*\rH_{I\cup J,b}$ is given by sending a Higgs field the collection of its Taylor expansion at the points in $I\cup J$, and similarly for the $\Ldet$-twisted cotangent bundles.

  It follows that the pullback of $(\Ldet',\na_{\det})$ to $\Lochf \x_{\Bun}\rH_{I\cup J,b}$ is identified (up to the constant factor $(\Ldet)_b$) with the pullback of the universal line bundle with connection on $\tilde T^*\rH_{I\cup J,b} \isom \Lochf^{I\cup J,b}$.  Therefore the same is true for the corresponding curvatures.  Now, our tangent vectors $\xi,\eta \in \gT_{\Bun,b}$ were given by infinitesimal modification of~$\cE$ near $I\cup J$ (in other words, the restriction of the corresponding self-extension of~$\cE$ to $U_{I\cup J}$ is trivialized), so they come from some tangent vectors $\xi_{I\cup J},\eta_{I\cup J}$ to $\rH_{I\cup J,b}$ at~$b_{I\cup J}$.  (In fact, the map $\rho_x\colon T_{C,x} \tsr \End\cE_x \to T^\blt_{\Bun,b}$ lifts to a map $\tilde\rho_x\colon T_{C,x} \tsr \End\cE_x \to T_{\rH_{x,b},b_x}$, and we have $\xi_{I\cup J}=\sum_{x\in I}\tilde\rho_x(P^{(x)})$ and $\xi_{I\cup J}=\sum_{y\in J}\tilde\rho_y(Q^{(y)})$.) Combining these with the lifts $\xi,'\eta'$ of $\xi,\eta$ to $\gT_{\Lochf,b'}$, we get a pair $\xi''_{I\cup J},\eta''_{I\cup J}$ of tangent vectors to $\Lochf \x_{\Bun}\rH_{I\cup J,b}$.  Let $\xi'_{I\cup J},\eta'_{I\cup J}$ be the images of $\xi''_{I\cup J},\eta''_{I\cup J}$ in the tangent space to $\Lochf^{I\cup J,b}$.  Using the isomorphism $\Lochf^{I\cup J,b} \isoto \prod_{x\in I\cup J} \Lochf^{x,b}$, we can define for every $x\in I\cup J$ the $\Lochf^{x,b}$-components of the vectors $\xi'_{I\cup J},\eta'_{I\cup J}$, which we'll denote by $\xi'_x,\eta'_x$.  %The projections of $\xi'_x,\eta'_x$ to $\rH_{x,b}$ equal $\xi_x=\tilde\rho_x(P^{(x)})$ for $x\in I$,
  Now the equality of pullbacks of curvatures $\Oloc$ and $\Oloc^{I\cup J,b}$ to $\Lochf \x_{\Bun}\rH_{I\cup J,b}$ gives:
  \[
      \Oloc(\xi',\eta') = \Oloc^{I\cup J,b}(\xi'_{I\cup J},\eta'_{I\cup J})  = \sum_{x\in I\cup J}\Oloc^{x,b}(\xi'_x,\eta'_x).
  \]

  Now, from~\eqref{appx|eq:xi=sum eta=sum} we see that for $x\in I$ the vector $\eta'_x$ is vertical with respect to the projection $\Lochf^{x,b}\to\rH_{x,b}$, and for $y\in J$, $\xi'_y$ is vertical.  Take $x\in I$.  Then we get that $\Oloc^{x,b}(\xi'_x,\eta'_x)$ is given by the pairing of the cotangent vector to $\rH_{x,b}$ at~$b_x$ corresponding to the vertical vector $\eta'_x$ with the projection~$\xi_x$ of $\xi'_x$ to the tangent space to $\rH_{x,b}$.  Now, the cotangent vector corresponding to~$\eta'_x$ is described by the Taylor expansion of the Higgs field $B$ to the formal neighborhood of~$x$, and $\xi_x=\tilde\rho_x(P^{(x)})$.  Now, from the definition of $\rho_x$ and $\tilde\rho_x$ we see that $\Oloc^{x,b}(\xi'_x,\eta'_x) =\on{Tr}P^{(x)}B_x$.  Similarly, we get $\Oloc^{y,b}(\xi'_y,\eta'_y)= -\on{Tr}Q^{(y)}A_y$ for $y\in J$.  Substituting these to the above formula for $\Oloc(\xi',\eta')$, we get the same expression as in~\eqref{appx|eq:OLoc'}, as desired.
\end{proof}

\begin{rem}\label{appx|rem:tw cot Gr}
  In the course of the proof of this proposition, we use certain fact about the twisted cotangent bundle to the affine Grassmannian for $\GL_N$ for the determinant line bundle.  Namely, let $D$ be a scheme isomorphic to $\Spec\k\dbkts t$ and let $\Gr_{G,D}$ be the ind-scheme (the \emph{affine Grassmannian} for~$G$) parametrizing $G$-bundles on~$D$ trivialized at the generic point $D^\circ\isom\Spec\k\dprts t$ of~$D$, where $G=\GL_N$.  In other words, $\Gr_{G,D}$ classifies free $\O := \O(D)$-submodules of full rank (a.k.a.\ lattices) $\cK^{\oplus N}$ (where $\cK=\O(D^\circ) \isom \k\dprts t$.  On this ind-scheme, we have the determinant bundle $\Ldet^D$ defined by the usual formula: the fiber at the lattice $L$ is given by $\det(L/L_0) \tsr \det(\O^{\oplus N}/L_0)^{\tsr-1}$ for any lattice $L_0\ctd L\cap\O^{\oplus N}$ (one can canonically identify all such lines for all $L_0$'s).

  The statement that we need is that the fiber at~$L$ of the twisted cotangent bundle to $\Gr_{G,D}$ corresponding to $\Ldet^D$ is canonically isomorphic to the affine space of $\ohf_D$-twisted connections on~$L$.  Moreover, this isomorphism is equivariant under the action of the loop group $G(\cK) \isom G\dprts t$ ($\Ldet$ is acted on by a central extension of $G(\cK)$, so the twisted cotangent is equivariant with respect to $G(\cK)$ itself), and is independent of the isomorphism $\O\isoto\k\dbkts t$. Moreover, we need this isomorphism to be compatible with the isomorphism $T^*_{\Ldet}\Bun \isoto \Lochf$ obtained from the one in Theorem~\ref{appx|main thm} in the following sense.  Given a finite subset $I\ctd C$,  and a point $b\in\Bun$, consider the space $\rH_{b,I}$ defined in the above proof.  It is isomorphic to $\prod_{x\in I} \Gr_{G,\hat x}$, and the isomorphism is defined canonically up to the action of the product of loop groups $\prod_{x\in I} G(\hat x^\circ)$. Our compatibility is that the map $\tilde T^*_{\Ldet}\Bun\x_{\Bun}\rH_{b,I} \to \prod_{x\in I} \tilde T^*_{\Ldet^{\hat x}}\Gr_{G,\hat x}$ (constructed from the identification in~\eqref{appx|eq:q*Ldet=bigboxtimes}) should correspond to the map sending an $\ohf_C$-connection to its restriction to $\hat I = \coprod_{x\in I} \hat x$.

  This statement can be proved using the same ideas as those used in the proof of Theorem~\ref{appx|main thm}.  Namely, given a rank~$N$ vector bundle $\cE$ on~$D$ with a trivialization $\phi\colon \O_{D^\circ} \isoto \cE|_{D^\circ}$, the tangent space to $\Gr_{G,D}$ at the corresponding point is identified with $(\cK/\O) \tsr_\O \End_\O\cE$, so by Serre duality the cotangent space is identified with $(\End_\O\cE)\tsr_\O \omega(D) =\Hom_\O(\cE,\cE\tsr\omega_D)$.  Now the fiber of the $\Ldet^D$-twisted cotangent bundle gives torsor over this (linearly compact) vector space, and hence a canonical (non-canonically trivial) extension $0\to \cE\tsr\omega_D\to \Phi(\cE)\to \cE\to 0$.  Moreover, from the action of the loop group on the twisted cotangent bundle, we see that $\Phi(\cE)$ does not depend on the trivialization $\phi$, \ie $\Phi(\cE)$ is functorial under isomorphisms of $\cE$'s.  Then by an argument as in Lemma~\ref{appx|lem:Phi fnctr}, we can extend the functoriality to non-invertible morphisms.  One can then extend it to non-free coherent sheaves and use a similar argument to the one for Theorem~\ref{appx|main thm} to finish the argument.
\end{rem}

\subsection{The case of characteristic~$p$}\label{appx|subs:charp}
Now suppose that $S$ is a scheme in characteristic~$p$, \ie for a prime~$p$ we have $p\O_S = 0$.  Then there is a component~$\rI$ of~$\Connhf$ classifying irreducible $\D$-modules.  We identify $\Connhf$ with $\Conn$ by ${\cdot}\tsr \omega^{\tsr(p-1)/2}$.  For further reference, denote this isomorphism by
\[
    \tau\colon \Conn \to \Connhf
\]
and the corresponding morphism $\Coh \to \Coh$ will be denoted by~$\tilde\tau$.  Recall the Azumaya algebra $\tilde\D_{C/S}$ on~$T^*(C\fr[S]/S)$ whose pushforward to~$C\fr[S]$ is isomorphic to~$\Fr_{C/S*}\D_{C/S}$.  For an $S$-scheme $S'$, the $S'$-points of~$\rI$ correspond to pairs $(y,\cE)$ where $y$ is an $S'$-point of~$T^*C\fr[S]$, and $\cE$ is a splitting of~$y^*\tilde\D_{C/S}$.  In other words, $\rI$ is isomorphic to (the ``total space'' of) the $\Gm$-gerbe $\dd_C$ on~$T^*C\fr$ corresponding to the algebra~$\tilde\D_{C/S}$.  We want to describe the restriction of $(\Ldet',\na_{\det})$ to~$\rI$.  In the text below we write $C\fr$, $\x$, $T^*C$, etc.\ for $C\fr[S]$, $\x_S$, $T^*(C/S)$, etc.

Because of the equivalence $\tilde\D_C \hmod \isoto \D_{T^*C,\th} \hmod$ (where $\th$ is the canonical 1-form on~$T^*C$), the gerbe~$\dd_C$ also classifies the irreducible modules over $\D_{T^*C,\th}$.  Therefore we can consider the universal object: this is a coherent sheaf on $\rI \x T^*C$ with connection in the $T^*C$-direction and with support given by $\rI \xx_{T^*C\fr} T^*C$.  Now if we apply the relative Frobenius twist over~$\rI$ to this sheaf, the resulting sheaf on $\rI \x T^*C\fr$ will have connection along both factors.  So we get a $\D$-module on $\rI \x T^*C\fr$ supported on the Frobenius neighborhood of the ``diagonal'' or, more precisely, of the graph of $v \colon \rI \to T^*C\fr$.  The restriction of this $\D$-module to the graph itself is a line bundle on~$\rI$ with connection which we will denote by $(\L_{\on{univ}}, \na_{\on{univ}})$.

\begin{prop}\label{appx|prop:det on I}
  There is a canonical isomorphism of line bundles with connections on~$\rI$:
  \begin{equation}\label{appx|eq:det on I}
    (\Ldet',\na_{\det})|_\rI  \isom  (\L_{\on{univ}}, \na_{\on{univ}} - v^*\th).
  \end{equation}
\end{prop}

\begin{proof}
  We will first construct an isomorphism between the LHS and the RHS, and then prove it is compatible with connections.

  To construct the isomorphism, in turn, we'll prove that the two sides become isomorphic after pullback by $F := \Fr_{C/S} \x_{C\fr} \id_\rI$, and then show that it descends to $\rI$.  Pick a point $\bar y\colon S'\to \rI$, and let us compare the fibers (\ie pullbacks) of both sides of~\ref{appx|eq:det on I} at~$\bar y$.  Denote by~$y$ the image of~$\bar y$ in~$T^*C\fr$, and by~$x$ the corresponding $S'$-point of~$C\fr$.  By definition of the stack~$\rI$, the point~$\bar y$ corresponds to a coherent sheaf~$\cF_{\bar y}$ on $C \x S'$ endowed with a $\ohf$-connection ``in the $C$-direction''; moreover, $\cF_{\bar y}$ is a pushforward of a line bundle from $C\x_{C\fr,x}S'$.  Then the pullback of~$\Ldet'$ under~$\bar y$ is given by
  \begin{equation*}
      \bar y^*\Ldet' = \det R\pi_{S'*} \cF_{\bar y} = \det \pi_{S'*} \cF_{\bar y}
  \end{equation*}
  where $\pi_{S'}$ denotes the projection $C \x S'\to S'$.  Here in the last equality we used that the support of~$\cF_{\bar y}$ is affine (actually, finite) over~$S'$.  Since $\bar y$ belongs to~$\rI$, the support of~$\cF_{\bar y}$ lies inside $C \xx_{C\fr} \Gamma_x \ctd C \x S'$ where $\Gamma_x \ctd C\fr \x S'$ is the graph of~$x$.

  Now we want to use the idea that a sheaf supported on the Frobenius neighborhood of a point carries a natural filtration (induced by the one on the structure sheaf given by powers of the maximal ideal).  Namely, suppose that we have a lift~$\tilde x$ of~$x$ from~$C\fr$ to~$C$.  Then $\cF_{\bar y}$ is filtered by powers of the ideal of $\Gamma_{\tilde x} \ctd C\x S'$.  Moreover, if $\L$ denotes the pullback of $\cF_{\bar y}$ to $\Gamma_{\tilde x}$ (note that $\L$ is a line bundle) then the associated graded of this filtration is supported on $\Gamma_{\tilde x} \isom S'$ and is isomorphic to $\dsum_{i=0}^{p-1} (\L\tsr \tilde x^*(\omega_C^{\tsr i}))$.  Now, pushing this sheaf forward to~$S'$ and taking the determinant, we get
  \begin{equation}\label{appx|eq:det-calc}
      \det \pi_{S'*} \cF_{\bar y} \isom \det \left(\dsum_{i=0}^{p-1} (\L\tsr\tilde x^*(\omega_C^{\tsr i}))\right)
        \isom \tprod_{i=0}^{p-1} (\L\tsr\tilde x^*(\omega_C^{\tsr i}))
        \isom \L^{\tsr p}\tsr\tilde x^*\omega_C^{\tsr\frac{p(p-1)}2}
  \end{equation}

  That is, we get that $\bar y^*\Ldet' \isom \L^{\tsr p}\tsr\tilde x^*\omega_C^{\tsr p(p-1)/2}$.  We claim that $\bar y^*\L_{\on{univ}}$ is isomorphic to the same thing.  Indeed, recall that in our definition of $(\L_{\on{univ}}, \na_{\on{univ}})$, we first identified $\Connhf$ with $\Conn$.  After this identification, the point~$\bar y$ corresponds to the sheaf $\cF'_{\bar y} := \cF_{\bar y} \tsr (\omega_C^{\tsr(p-1)/2} \bx\O_{S'})$ equipped with a connection~$\na_{\cF'_{\bar y}}$.  Under the equivalence $\D_C\hmod \isoto \D_{T^*C,\th}$ the $\D$-module $(\cF'_{\bar y},\na_{\cF'_{\bar y}})$ corresponds to a $\D$-module $(\cF''_{\bar y},\na_{\cF''_{\bar y}})$ of $p$-curvature $\th$.  One can factor this equivalence as pullback under $T^*C\to C$ followed by central reduction from $\D_{T^*C}$ to $\D_{T^*C,\th}$.  It is easy to see that for (families of) irreducible $\D_C$-modules this central reduction is equivalent to restriction to the corresponding Frobenius neighborhood of a point in $T^*C$.  Thus $\cF''_{\bar y}$ is isomorphic to the pullback of $\cF'_{\bar y}$ to $T^*C \xx_{T^*C\fr} \Gamma_y$.  Finally, to get $(\L_{\on{univ}})_y$, we must apply the Frobenius twist (relative to~$S'$) to~$\cF''_{\bar y}$ and then restrict to~$\Gamma_y$.  Denote by $\gamma_x = (\id_{S'},x) \colon S' \to S'\x C\fr$ the embedding of the graph, and define $\gamma_{\tilde x} \colon S'\to S'\x C$ similarly.  Now, by the base-change formula, we see that $(\L_{\on{univ}})_y \isom \gamma_x^*\cF_{\bar y}\pfr[S'] \isom (\gamma_{\tilde x}^*\cF'_{\bar y})\fr[S'] \isom (\gamma_{\tilde x}^*\cF_{\bar y} \tsr x^*\omega_C^{\tsr(p-1)/2})\fr[S'] \isom \L^{\tsr p} \tsr x^*\omega_C^{\tsr p(p-1)/2}$, which coincides with the right-hand side of~\eqref{appx|eq:det-calc}.

  The constructed isomorphism is evidently compatible with pullbacks, but it uses a lift~$\tilde x$ of~$x$ to~$C$.  This means that we've constructed an isomorphism~$\tilde\a$ of pullbacks of $\Ldet'$ and $\L_{\on{univ}}$ under $\tilde\rI := C\xx_{C\fr}\rI \to \rI$.  So it remains to show that this isomorphism descends to~$\rI$, and then that this descended isomorphism is compatible with connections.

  For the former, let us further pull back~$\tilde\a$ to \ncmd{\ttI}{{\Tilde{\tilde\rI}}}%
  $\ttI := T^*C\xx_{T^*C\fr}\rI \xra{\phi'} \tilde\rI \xra{\phi''} \rI$.  Note that $\ttI$ is a $\Gm$-gerbe on~$T^*C$ with canonical connection~$\na_\ttI$ (as on a Frobenius pullback), and, denoting $\phi=\phi''\circ\phi'$, line bundles $\phi^*\Ldet'$ and $\phi^*\L_{\on{univ}}$ have canonical connections ``in the horizontal direction with respect to $\na_\ttI$.''  Moreover, the morphism $\phi'^*\tilde\a$ descends to~$\rI$ if and only if it is compatible with these connections (and in this case the resulting morphism is, of course, unique).  The fact that $\phi'^*\tilde\a$ descends to~$\tilde\rI$ means that the two connections coincide in the vertical direction in $T^*C$.  So the difference between the connections is a 1-form $\nu_C$ on~$T^*C$ vanishing on the fibers which is closed (because the connections are flat).

  Now note that $\nu_C$ depends on~$C$ in a way compatible with base-change in~$S$ and \'etale base-change in~$C$.  One can prove that the only such ``natural'' one-forms (without the closedness assumption) are of the form $\la\th$ where $\th$ is the canonical 1-form on the cotangent bundle and $\la$ is a scalar (actually $\la\in\Fp$ since $\nu_C$ is defined for any $S$~and~$C$ in characteristic~$p$).  But since we know that $\nu_C$ is closed, we must have $\la=0$, hence $\nu_C=0$ for any~$C$, which finishes the proof that $\tilde\a$ descends to~$\rI$.  We denote by $\a\colon \Ldet' \isoto \L_{\on{univ}}$  the descended isomorphism.

%  We also note that $\tilde\a$ descends to the stack~$\rJ\ctd\Coh$ parametrizing cyclic coherent sheaves on~$C$ with support of the form $C\xx_{C\fr,x}S'$ (here $S'$ is any test scheme, and $x\colon S'\to C\fr$ is any point).  More precisely,

  Now we prove that $\a$ is compatible with connections as indicated in~\eqref{appx|eq:det on I}.  Similarly to the previous paragraph, we get a 1-form on~$T^*C\fr$ which is ``natural in $S$~and~$C$.''  Hence, by the same reasoning, this form must be equal to $\la\th\fr$ for some constant $\la$, so that we have
  \begin{equation}\label{appx|eq:na+la th}
       \a \circ \na_{\det} \circ \a^{-1} = \na_{\on{univ}} - \la v^*\th\fr,
  \end{equation}
  and we need to show that $\la=1$.  We are going to deduce this from the linearity of $\na_{\det}$ along the fibers of the projection $\Connhf \to \Coh$.

  Namely, consider the map $\rho\colon C\fr \to \Coh$ which associates to a point~$x$ of~$C$ the structure sheaf of the Frobenius neighborhood of (the graph of) $x$.  In other words, $\rho$ corresponds to the coherent sheaf on $C\x C\fr$ given by the pushforward of~$\O$ from the graph of $\Fr_C$.  Then, for any open $U\ctd C$, maps $s\colon U\fr\to\Conn$ lifting~$\rho|_{U\fr}$ correspond bijectively to connections $\na = d+\omega$ on~$\O_U$; we will denote this correspondence by $\omega \mapsto s_\omega'$ and let $s_\omega=\tau\circ s_\omega'$.  For any such~$\omega$, $s_\omega^*\na_{\det}$ defines a connection on~$\rho^*\tilde\tau^*\Ldet$ which depends linearly on~$\omega$ (this linearity follows from the construction of~$\na_{\det}$).  On the other hand, let us pull back the right-hand side of~\eqref{appx|eq:na+la th} under $s_\omega$.  One can check that $s_\omega^*\Lna_{\on{univ}} \isom (\O_{U\fr},d+\omega\fr)$, and the composite isomorphism $\rho^*\tilde\tau^*\Ldet \xra{s_\omega^*\a} s_\omega^*\L_{\on{univ}} \isoto \O_{U\fr}$ does not depend on~$\omega$.  For the second summand, note that the composition $v\circ s_\omega$ is the section of $T^*C\fr\to c\fr$ corresponding to the $1$-form given by the $p$-curvature $\curv_p(\O_C,d+\omega) = \omega\fr-\sC(\omega)$.  So the RHS of~\eqref{appx|eq:na+la th} can be rewritten as
  \[
      s_\omega^* (\L_{\on{univ}}, \na_{\on{univ}} - \la v^*\th\fr) \isom (\O_{U\fr}, d + (1-\la) \omega\fr + \la\sC(\omega)).
  \]
  Since we know that the connection in the LHS must be linear in~$\omega$, we should have $\la=1$, which finishes the proof.
\end{proof}

\end{document}